%% file: Paper_AC_MCFalpha_2D_arxiv_v2.tex
\documentclass[toc=bibnumbered,ngerman,english,titlepage=false,twoside]{scrartcl}
\usepackage[lmargin=107.551428pt,rmargin=71.700952pt]{geometry} % das und twoside in documentoptions machts zweiseitig
\usepackage[utf8]{inputenc}
\usepackage{babel}
\usepackage[T1]{fontenc}
\usepackage{lmodern}
\usepackage{microtype}

\usepackage{amsmath,amsthm,amssymb}
\usepackage{bm}
\usepackage{textcomp}
\usepackage{csquotes}
\usepackage{enumitem}
\usepackage[color links=true,allcolors=black]{hyperref}

\usepackage{graphicx}
\usepackage{float}
\usepackage{times}

\usepackage[automark,footwidth=\textwidth,headwidth=\textwidth]{scrlayer-scrpage}

\numberwithin{equation}{section} %bei align-Nummerierung section-nummer davor

\newcommand{\N}{\ensuremath{\mathbb{N}}}

\newcommand{\R}{\ensuremath{\mathbb{R}}}

\newcommand{\C}{\ensuremath{\mathbb{C}}}
\newcommand{\K}{\ensuremath{\mathbb{K}}}

\newcommand{\diverg}{\textup{div}}

\newcommand{\tr}{\textup{tr}}

\newcommand{\supp}{\textup{supp}}

\newcommand{\Hc}{\ensuremath{\mathcal{H}}}

\newcommand{\Lc}{\ensuremath{\mathcal{L}}}

\newcommand{\Nc}{\ensuremath{\mathcal{N}}}
\newcommand{\Oc}{\ensuremath{\mathcal{O}}}

\setkomafont{disposition}{\normalcolor\bfseries}
\setkomafont{section}{\LARGE} 
\setkomafont{subsection}{\Large} 
 
\begin{document}
\thispagestyle{plain}
%%Nutzt mehr von einer A4-Seite
%\oddsidemargin 6pt \evensidemargin 6pt %Platz auf den Rändern
%\marginparwidth 48pt\marginparsep 10pt %for marginal notes

\topmargin -18pt\headheight 12pt\headsep 25pt

\ifx\cs\documentclass \footheight 12pt \fi \footskip 30pt
%für fußnoten

\textheight 625pt\textwidth 431pt\columnsep 10pt\columnseprule 0pt 
                                                                      
\clearpairofpagestyles%Leert scrheadings-stil bzw. plain.scrheadings-stil. Mit folgendem neu definierbar: in eckigen Klammern wird plain festgelegt, in geschweiften klammern scrheadings-stil
\automark*[section]{section}
\automark*[subsection]{}
\renewcommand\sectionmarkformat{\thesection.\enskip}
\ohead[]{\scshape\headmark}
\ofoot*{\pagemark}
\pagestyle{scrheadings}

 \renewcommand{\headfont}{\slshape}      % Kopfzeilen stil
 \renewcommand{\pnumfont}{\upshape}      % Seitenzahlen stil
\setcounter{secnumdepth}{5}             %wie tief überschriften nummeriert werden
\setcounter{tocdepth}{5}             %wie tief in Inhaltsverzeichnis

\newtheorem{Definition}{Definition}[section]
\newtheorem{Satz}[Definition]{Satz}
\newtheorem{Lemma}[Definition]{Lemma}
\newtheorem{Korollar}[Definition]{Korollar}
\newtheorem{Corollary}[Definition]{Corollary}
\newtheorem{Bemerkung}[Definition]{Bemerkung}
\newtheorem{Remark}[Definition]{Remark}
\newtheorem{Proposition}[Definition]{Proposition}
\newtheorem{Beispiel}[Definition]{Beispiel}
\newtheorem{Theorem}[Definition]{Theorem}

\pdfbookmark[1]{Titlepage}{title}
\begin{center}
	{\LARGE Convergence of the Allen-Cahn equation with a nonlinear\\ 
		Robin Boundary Condition to Mean Curvature Flow\\[0.5ex]
		with Contact Angle close to $90$°}\\[2ex]
	\textsc{Helmut Abels \quad \& \quad Maximilian Moser}\\[1ex]
	Fakultät für Mathematik, Universität Regensburg, Universitätsstraße 31,\\ D-93053 Regensburg, Germany\\
	helmut.abels@mathematik.uni-regensburg.de\\
	maximilian1.moser@mathematik.uni-regensburg.de\\[1ex]
\end{center}

\begin{abstract}
	\textbf{Abstract.} This paper is concerned with the sharp interface limit for the Allen-Cahn equation with a nonlinear Robin boundary condition in a bounded smooth domain $\Omega\subset\R^2$. We assume that a diffuse interface already has developed and that it is in contact with the boundary $\partial\Omega$. The boundary condition is designed in such a way that the limit problem is given by the mean curvature flow with constant $\alpha$-contact angle. For $\alpha$ close to $90$° we prove a local in time convergence result for well-prepared initial data for times when a smooth solution to the limit problem exists. Based on the latter we construct a suitable curvilinear coordinate system and carry out a rigorous asymptotic expansion for the Allen-Cahn equation with the nonlinear Robin boundary condition. Moreover, we show a spectral estimate for the corresponding linearized Allen-Cahn operator and with its aid we derive strong norm estimates for the difference of the exact and approximate solutions using a Gronwall-type argument.\\
	
\end{abstract}

\section{Introduction}\label{sec_intro}\pagenumbering{arabic}\label{sec_ACalpha}
Let $\Omega\subset\R^2$ be a bounded, smooth domain with outer unit normal $N_{\partial\Omega}$. Moreover, let $\alpha\in(0,\pi)$ be fixed. The parameter $\alpha$ will correspond to a constant angle later. Furthermore, we consider $\varepsilon>0$ small. For $u_{\varepsilon,\alpha}:\overline{\Omega}\times[0,T]\rightarrow\R$ we are interested in the following Allen-Cahn equation with nonlinear Robin boundary condition which we will refer to via \hypertarget{ACalpha}{(AC$_\alpha$)}:
\begin{alignat}{2}\label{eq_ACalpha1}\tag{AC$_\alpha1$}
\partial_tu_{\varepsilon,\alpha}-\Delta u_{\varepsilon,\alpha}+\frac{1}{\varepsilon^2}f'(u_{\varepsilon,\alpha})&=0&\qquad&\text{ in }\Omega\times(0,T),\\\label{eq_ACalpha2}\tag{AC$_\alpha2$}
\partial_{N_{\partial\Omega}}u_{\varepsilon,\alpha}+\frac{1}{\varepsilon}\sigma_\alpha'(u_{\varepsilon,\alpha})&=0&\qquad&\text{ on }\partial\Omega\times(0,T),\\
u_{\varepsilon,\alpha}|_{t=0}&=u_{0,\varepsilon,\alpha}&\qquad&\text{ in }\Omega.\label{eq_ACalpha3}\tag{AC$_\alpha 3$}
\end{alignat} 

Here $f:\R\rightarrow\R$ is an appropriate smooth double well potential with wells of equal depth, for instance $f(u)=\frac{1}{2}(1-u^2)^2$, see Figure \ref{fig_double_well}. The concise conditions we impose are
\begin{align}\label{eq_AC_fvor1}
f\in C^\infty(\R),\quad f'(\pm1)=0,\quad f''(\pm1)>0,\quad 
f(-1)=f(1),\quad f>f(1)\text{ in }(-1,1)
\end{align}
and the following sign condition for $f'$ outside a large ball:
\begin{align}\label{eq_AC_fvor2}
uf'(u)\geq0\quad\text{ for all }|u|\geq R_0\text{ and an }R_0\geq 1.
\end{align} 
We use \eqref{eq_AC_fvor2} later to deduce uniform a priori bounds for classical solutions $u_{\varepsilon,\alpha}$, cf.~Section \ref{sec_DC_prelim_bdd_scal}.
\begin{figure}[H]
	\centering
	\def\svgwidth{0.4\linewidth}
	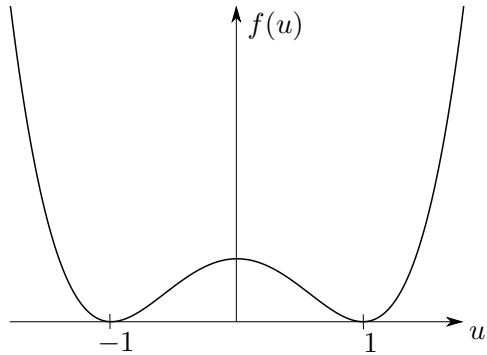
	\caption{Typical example for the double-well potential, $f(u)=\frac{1}{2}(1-u^2)^2$.}\label{fig_double_well}
\end{figure}
Furthermore, $\sigma_\alpha:\R\rightarrow\R$ is assumed to be smooth with $\supp\,\sigma_\alpha'\subset(-1,1)$ and
\begin{align}\label{eq_angle_comp1}
\cos \alpha=\frac{\sigma_\alpha(-1)-\sigma_\alpha(1)}{\int_{-1}^1\sqrt{2(f(r)-f(-1))}\,dr}.
\end{align}
For the typical shape of $\sigma_\alpha$ see Figure \ref{fig_sigma_alpha} below. In order to fulfil the compatibility condition \eqref{eq_angle_comp1} and to have smoothness of $\sigma_\alpha$ with respect to $\alpha$ we choose $\sigma_\alpha$ for simplicity as follows:
\begin{Definition}\label{th_ACalpha_sigma_def}\upshape
	Let $\hat{\sigma}:\R\rightarrow\R$ be smooth with $\supp\,\hat{\sigma}'\subset(-1,1)$ and such that 
	\[
	\hat{\sigma}(-1)-\hat{\sigma}(1)=\int_{-1}^1\sqrt{2(f(r)-f(-1))}\,dr.
	\] 
	Then we define $\sigma_\alpha:=\cos\alpha\,\hat{\sigma}$.
\end{Definition}

\begin{figure}[H]
	\centering
	\def\svgwidth{0.8\linewidth}
	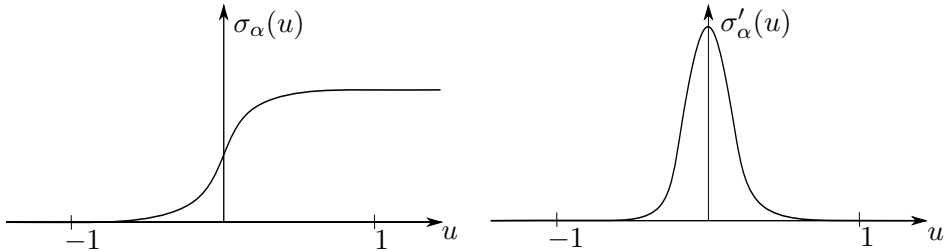
	\caption{Typical shape of $\sigma_\alpha$ and $\sigma_\alpha'$.}\label{fig_sigma_alpha}
\end{figure}

\begin{proof}[Motivation of \hyperlink{ACalpha}{(AC$_\alpha$)}]
By direct computation one can show that equation \eqref{eq_ACalpha1}-\eqref{eq_ACalpha3} is the $L^2$-gradient flow corresponding to the energy
\begin{align}\label{eq_ACalpha_energy}
E_{\varepsilon,\alpha}(u):=\int_\Omega\frac{1}{2}|\nabla u|^2+\frac{1}{\varepsilon^2}f(u)\,dx+\int_{\partial\Omega}\frac{1}{\varepsilon}\sigma_\alpha(u)\,d\Hc^{1}.
\end{align}
In the case $\alpha=\frac{\pi}{2}$ only the first part in the energy remains, which is (up to a scaling in $\varepsilon$) the standard scalar Ginzburg-Landau energy or Modica-Mortola energy, see~Modica \cite{Modica1}. The corresponding $L^2$-gradient flow in this case is the standard Allen-Cahn equation, see Allen, Cahn \cite{AC} and the introduction in Bronsard, Reitich \cite{BronsardReitich} for some motivations. The new term in the energy is a boundary contact energy. The idea is to adjust the gradient flow in such a way that distinct values of $u$ are penalized differently when attained at the boundary.\phantom{\qedhere}

With formal arguments (system of fast reaction/slow diffusion or an energy argument; see Moser \cite{MoserACvACND} for some details in the case $\alpha=\frac{\pi}{2}$) one can show that \eqref{eq_ACalpha1}-\eqref{eq_ACalpha3} is a diffuse interface model: the $u_{\varepsilon,\alpha}$ is an order parameter, where the values $\pm 1$ represent two distinct phases or components in applications. Typically after a short \enquote{generation} time $\Omega$ is divided into subdomains where the solution $u_{\varepsilon,\alpha}$ of \eqref{eq_ACalpha1}-\eqref{eq_ACalpha3} is near $\pm1$ and diffuse interfaces develop with thickness proportional to $\varepsilon$. Therefore in the limit $\varepsilon\rightarrow0$ we should get an evolving hypersurface $\Gamma=(\Gamma_t)_{t\in[0,T]}$, which moves according to some sharp interface model. This is the meaning behind the notion \enquote{sharp interface limit}. See~Figure \ref{fig_transition_zone_alpha} below for a sketch. A rigorous result concerning the \enquote{generation of interfaces} in the case $\alpha=\frac{\pi}{2}$ is provided by Chen \cite{ChenGenPropInt}.

In general it is important to connect sharp interface models and diffuse interface models via their sharp interface limits. Both model types can be used in various applications and there are motivations from the modelling and analysis perspective as well as from the numerics side. See \cite{MoserACvACND} for a summary of possible applications and motivations.
\end{proof}	

\begin{figure}[H]
	\centering
	\def\svgwidth{0.9\linewidth}
	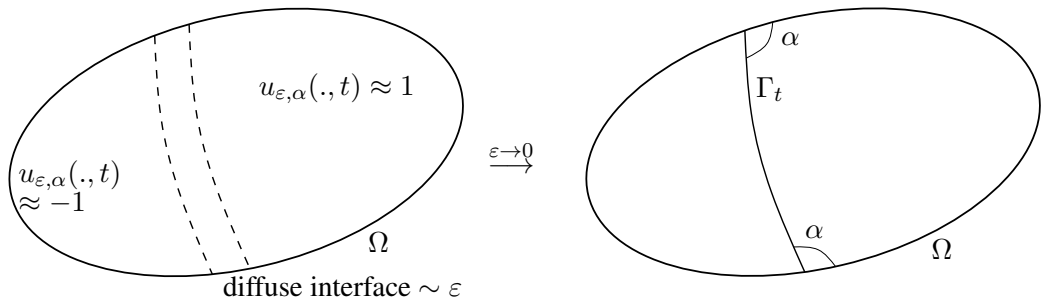
	\caption{Diffuse interface and sharp interface limit.}\label{fig_transition_zone_alpha}
\end{figure}

\begin{proof}[Formal Sharp Interface Limit for \hyperlink{ACalpha}{(AC$_\alpha$)}]
	In Owen, Sternberg \cite{OwenSternberg} formal asymptotic analysis is used to determine the sharp interface limit for some boundary contact energy densities. The calculations can be adapted for the $\sigma_\alpha$ in Definition \ref{th_ACalpha_sigma_def}. It will turn out that $f$ and $\sigma_\alpha$ are balanced suitably through \eqref{eq_angle_comp1}, see also Remark \ref{th_hp_angle_comp} below, such that in the sharp interface limit $\varepsilon\rightarrow 0$ we formally obtain the mean curvature flow
	\begin{align}\label{MCF}
	V_{\Gamma_t}=H_{\Gamma_t} \tag{MCF}
	\end{align}
	together with constant $\alpha$-contact angle. The boundary contact energy densities in \cite{OwenSternberg} are different to our $\sigma_\alpha$ and rather physically motivated. Our choice for $\sigma_\alpha$ is motivated by the goal to obtain \eqref{MCF} with the constant contact angle $\alpha$ in the sharp interface limit. Moreover, $\sigma_\alpha$ is chosen as simple as possible in order to shorten the proofs. Note that \eqref{eq_angle_comp1} is reminiscent of the well-known Young's Equation that can be used to compute the contact angle of three adjacent media through surface tension relations.\phantom{\qedhere}
	
	Finally, note that Modica \cite{Modica2} studied the $\Gamma$-convergence with respect to $\varepsilon\rightarrow 0$ for energies of the form \eqref{eq_ACalpha_energy} (up to a scaling in $\varepsilon$) with mass and nonnegativity constraint. The $\Gamma$-limits are perimeter functionals, where additionally the perimeter of some part of the boundary is added but weighted with a constant related to the potential and the boundary contact energy. This also motivates to study the dynamical problem \eqref{eq_ACalpha1}-\eqref{eq_ACalpha3} associated to \eqref{eq_ACalpha_energy} and the relation to \eqref{MCF} with contact angle distinct from $\frac{\pi}{2}$ in the limit $\varepsilon\rightarrow 0$.
\end{proof}

\begin{proof}[Rigorous Sharp Interface Limit Results for \hyperlink{ACalpha}{(AC$_\alpha$)}]	
	There are many results for the case $\alpha=\frac{\pi}{2}$, i.e.~the usual Allen-Cahn equation. Here it is well-known that the sharp interface limit is \eqref{MCF}, and in case of boundary contact, there is a $\frac{\pi}{2}$-contact angle. Let us summarize the results, see \cite{MoserACvACND} for a more detailed description. On the one hand there are local in time results, proving quite strong assertions (e.g.~norm estimates) for times before singularities appear and as long as the interface stays smooth. See \cite{ChenGenPropInt} (via the construction of sub- and supersolutions and the comparison principle) and for closed interfaces \cite{deMS} (by rigorous asymptotic expansions, see below for a description). The latter was refined in \cite{CHL}, \cite{ALiu} and the situation of boundary contact for the diffuse interfaces was considered in \cite{AbelsMoser} for 2D and \cite{MoserACvACND} for ND. Finally note \cite{FischerLauxSimon}, where a relative entropy method was used. On the other hand, there are global in time results that use a weak formulation for the limit system, cf.~\cite{ESS}, \cite{KKR} (viscosity solutions), \cite{Ilmanen},\cite{MizunoTonegawa}, \cite{Kagaya} (varifold solutions), \cite{LauxSimon} (conditional result in BV-setting).\phantom{\qedhere}
	
	To the best of our knowledge, for the case $\alpha\neq\frac{\pi}{2}$ there is no rigorous result so far in the literature. However, note that on the energy level there is a preparatory result in a varifold setting, see Kagaya, Tonegawa \cite{KagayaTonegawa}, in particular the remarks in \cite{KagayaTonegawa}, Section 5.3. We prove a rigorous sharp interface limit result local in time for $\alpha$ close to $\frac{\pi}{2}$ with the method of de Mottoni and Schatzman \cite{deMS}, see Theorem \ref{th_ACalpha_conv} below. It extends the result in Abels, Moser \cite{AbelsMoser}, where the case $\alpha=\frac{\pi}{2}$ is considered. The result is part of the PhD thesis of the second author, cf.~Moser \cite{MoserDiss}. Note that also \cite{MoserACvACND} is part of \cite{MoserDiss}.
\end{proof}

\begin{proof}[The Method of de Mottoni and Schatzman] 
	The method relies on the local in time existence of a smooth solution to the sharp interface problem. The latter can usually be proven. Then 
	\begin{enumerate} 
		\item Based on the evolving hypersurface that is (part of) the solution to the sharp interface problem, one carries out a rigorous asymptotic expansion of the diffuse interface model in order to obtain a suitable approximate solution.\phantom{\qedhere}
		\item Then one uses a Gronwall argument in order to control the difference of the exact and approximate solutions. This then entails the need for a spectral estimate of an operator obtained from the diffuse interface model by some linearization at the approximate solution.  
	\end{enumerate}
	With this method the typical profile of the solution is also obtained and comparison principles are not required in an essential manner. Therefore the approach has been applied to many other diffuse interface models as well, see \cite{MoserACvACND} for a detailed list of references. For an overview of the rigorous asymptotic expansions and the spectral estimates in the applications we refer to the introductions of Sections 5 and 6 in \cite{MoserDiss}. The novelty in the contributions \cite{AbelsMoser} (and \cite{MoserACvACND}) is the consideration of boundary contact for the diffuse interfaces within the method. 
\end{proof}

\begin{proof}[Mean Curvature Flow \eqref{MCF} with Contact Angle]
	In our 2D case mean curvature is simply the curvature of the curve. For our convergence result we will make the assumption that \eqref{MCF} with constant contact angle $\alpha$ has a smooth solution on some time interval $[0,T_0]$. This is a precondition for the method of de Mottoni and Schatzman \cite{deMS}.\phantom{\qedhere}
	
	The local in time well-posedness starting from suitable initial curves is basically well-known. See for example Katsoulakis, Kossioris, Reitich \cite{KKR}, Section 2, for a parametric approach. Another idea is to reduce the evolution to a parabolic PDE by writing it over a reference curve via suitable coordinates. The typical procedure in the case of a closed hypersurface is described e.g.~in Prüss, Simonett \cite{PruessSimonett}. For curvilinear coordinates in the situation of boundary contact see Vogel \cite{Vogel} and Section \ref{sec_coord} below. 
	
	We need some notation concerning the coordinates for the formulation of the main theorem.\phantom{\qedhere}
	
	\begin{Remark}[\textbf{Domain, Sharp Interface and Coordinates}]\label{th_intro_coord}\upshape
		The details are given in Section \ref{sec_coord} below. For a sketch of the situation see Figure \ref{fig_TubularNeighbNotation}.
		\begin{enumerate}
			\item \textit{Domain.} Let $\Omega\subset\R^2$ be a bounded, smooth domain with outer unit normal $N_{\partial\Omega}$. For $T>0$ let $Q_T:=\Omega\times(0,T)$ and $\partial Q_T:=\partial\Omega\times[0,T]$. 
			\item \textit{Sharp Interface.} Let $T_0>0$ and $\Gamma=(\Gamma_t)_{t\in[0,T_0]}$ be an evolving curve (with boundary, smooth, compact, connected) parametrized appropriately over a reference interval $I$ and with boundary contact of $\partial\Gamma$ at $\partial\Omega$ with constant contact angle $\alpha\in(0,\pi)$. For $\Gamma$ let $V_{\Gamma_t}$ be the normal velocity and $H_{\Gamma_t}$ be the curvature at time $t\in[0,T_0]$ with respect to a unit normal $\vec{n}$ of $\Gamma$. Cf.~Section \ref{sec_coord_surface_requ} below for the exact prerequisites.
			\item \textit{Coordinates.} We construct suitable curvilinear coordinates $(r,s)$ taking values in a trapeze with angle $\alpha$ based on the rectangle $[-2\delta,2\delta]\times I$ for $\delta>0$ small (cf.~Figure \ref{fig_trapez} below for a sketch) parametrizing a neighbourhood of $\Gamma$ in $\overline{\Omega}\times[0,T_0]$. For the precise statements see Theorem \ref{th_coord2D}. Here $r$ works as a signed distance function and $s$ has the role of a tangential component. The set $\overline{Q_{T_0}}=\overline{\Omega}\times[0,T_0]$ is divided by $\Gamma$ into two connected sets $Q_{T_0}^\pm$ corresponding to the sign of $r$, i.e.~the following disjoint union holds:
			\[
			\overline{Q_{T_0}}=\Gamma\cup Q_{T_0}^- \cup Q_{T_0}^+.
			\]
			Finally, we define tubular neighbourhoods $\Gamma(\eta):=r^{-1}((-\eta,\eta))$ for $\eta\in(0,2\delta]$ and introduce a normal derivative $\partial_n$ and some tangential gradient $\nabla_\tau$ on $\Gamma(\eta)$, cf.~Remark \ref{th_coord2D_rem2}.	
		\end{enumerate} 
\end{Remark}\end{proof}

\begin{figure}[H]\centering
	\def\svgwidth{0.8\linewidth}
	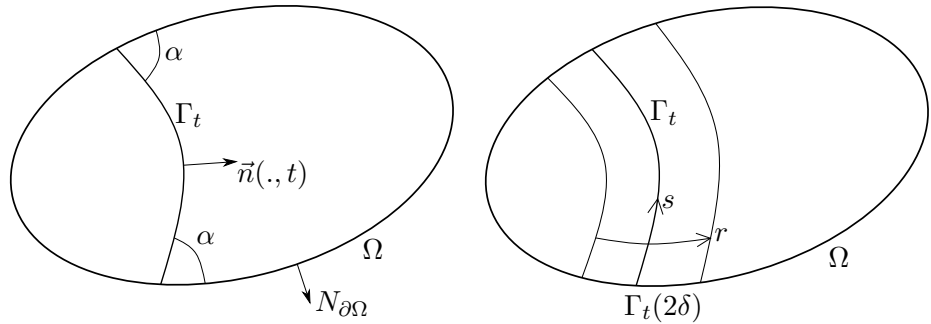
	\caption{Sharp interface with $\alpha$-contact angle and curvilinear coordinates.}\label{fig_TubularNeighbNotation}
\end{figure}

\begin{proof}[The Main Theorem] We obtain the following rigorous sharp interface limit result for \hyperlink{ACalpha}{(AC$_\alpha$)}:\phantom{\qedhere}
\begin{Theorem}[\textbf{Convergence of (AC$_\alpha$) to (MCF) with $\alpha$-Contact Angle}]\label{th_ACalpha_conv}
	There is an $\overline{\alpha}_0>0$ small such that the following holds.
	Let $\Omega$, $N_{\partial\Omega}$, $Q_T$ and $\partial Q_T$ for $T>0$ be as in Remark \ref{th_intro_coord},~1. Furthermore, let $\Gamma=(\Gamma_t)_{t\in[0,T_0]}$ for some $T_0>0$ be a smooth evolving curve with $\alpha$-contact angle as in Remark \ref{th_intro_coord},~2.~for fixed $\alpha\in\frac{\pi}{2}+[-\overline{\alpha}_0,\overline{\alpha}_0]$ and let $\Gamma$ solve \eqref{MCF}. Moreover, let $\delta>0$ be small and $Q_{T_0}^\pm$, $\Gamma(\delta)$, $\nabla_\tau$, $\partial_n$ be defined as in Remark \ref{th_intro_coord},~3. Finally, let $f$ fulfill \eqref{eq_AC_fvor1}-\eqref{eq_AC_fvor2} and let $\sigma_\alpha$ be as in Definition \ref{th_ACalpha_sigma_def}. Let $M\in\N$ with $M\geq 3$.

	Then there are $\delta_0\in(0,\delta]$, $\varepsilon_0>0$ and $u^A_{\varepsilon,\alpha}:\overline{\Omega}\times[0,T_0]\rightarrow\R$ smooth for $\varepsilon\in(0,\varepsilon_0]$ (depending on $M$) such that $\lim_{\varepsilon\rightarrow 0}u^A_{\varepsilon,\alpha}=\pm 1$ uniformly on compact subsets of $Q_{T_0}^\pm$ and:
	\begin{enumerate}
		\item If $M\geq 4$, then consider $u_{0,\varepsilon,\alpha}\in C^2(\overline{\Omega})$ with $\partial_{N_{\partial\Omega}} u_{0,\varepsilon,\alpha}+\frac{1}{\varepsilon}\sigma_\alpha'(u_{0,\varepsilon,\alpha})=0$ on $\partial\Omega$ for $\varepsilon\in(0,\varepsilon_0]$ and such that for some $R>0$ and all $\varepsilon\in(0,\varepsilon_0]$ it holds
		\begin{align}\label{eq_ACalpha_conv1}
		\sup_{\varepsilon\in(0,\varepsilon_0]}\|u_{0,\varepsilon,\alpha}\|_{L^\infty(\Omega)}<\infty\quad\text{ and }\quad\|u_{0,\varepsilon,\alpha}-u^A_{\varepsilon,\alpha}|_{t=0}\|_{L^2(\Omega)}\leq R\varepsilon^M.
		\end{align}
		Then for solutions $u_{\varepsilon,\alpha}\in C^2(\overline{Q_{T_0}})$ of \eqref{eq_ACalpha1}-\eqref{eq_ACalpha3} for $\varepsilon\in(0,\varepsilon_0]$ starting from the initial data $u_{0,\varepsilon,\alpha}$ there are $\varepsilon_1\in(0,\varepsilon_0]$, $C>0$ such that
		\begin{align}\begin{split}\label{eq_ACalpha_conv2}
		\sup_{t\in[0,T]}\|(u_{\varepsilon,\alpha}-u^A_{\varepsilon,\alpha})(t)\|_{L^2(\Omega)}+\|\nabla(u_{\varepsilon,\alpha}-u^A_{\varepsilon,\alpha})\|_{L^2(Q_T\setminus\Gamma(\delta_0))}&\leq C\varepsilon^M,\\
		\sqrt{\varepsilon}\|\nabla_\tau(u_{\varepsilon,\alpha}-u^A_{\varepsilon,\alpha})\|_{L^2(Q_T\cap\Gamma(\delta_0))}+\varepsilon\|\partial_n(u_{\varepsilon,\alpha}-u^A_{\varepsilon,\alpha})\|_{L^2(Q_T\cap\Gamma(\delta_0))}&\leq C\varepsilon^M\end{split}
		\end{align}
		for all $\varepsilon\in(0,\varepsilon_1]$ and $T\in(0,T_0]$.
		\item If $M\geq 4$, then there is a $\tilde{R}>0$ small such that the assertion in 1.~is valid, provided that $R,M$ in \eqref{eq_ACalpha_conv1}-\eqref{eq_ACalpha_conv2} are displaced by $\tilde{R},3$.
		\item If $M=3$, then there is $T_1\in(0,T_0]$ such that the analogous result in 1.~is true, but \eqref{eq_ACalpha_conv2} is only obtained for all $\varepsilon\in(0,\varepsilon_1]$ and $T\in(0,T_1]$. 
	\end{enumerate}
\end{Theorem}

\begin{Remark}\phantomsection{\label{th_ACalpha_conv_rem}}\upshape\begin{enumerate}
		\item \textit{Interpretation of Theorem \ref{th_ACalpha_conv}.} The $u^A_{\varepsilon,\alpha}$ in the theorem can be viewed as the representation of a diffuse interface moving with $\Gamma$ because $u^A_{\varepsilon,\alpha}$ is smooth but converges for $\varepsilon\rightarrow 0$ to a step function with jump set equal the solution $\Gamma$ to \eqref{MCF} with $\alpha$-contact angle starting from $\Gamma_0$. The initial data $u_{0,\varepsilon,\alpha}$ in Theorem \ref{th_ACalpha_conv} are \enquote{well-prepared}, i.e.~the generation of diffuse interfaces in the evolution is skipped and it is assumed that a diffuse interface is located at the initial sharp interface $\Gamma_0$ at time $t=0$. Therefore Theorem \ref{th_ACalpha_conv} basically proves that the qualitative behaviour of diffuse interfaces with boundary contact, generated by \hyperlink{ACalpha}{(AC$_\alpha$)}, is that of \eqref{MCF} with $\alpha$-contact angle, at least as long as the evolution of the latter remains smooth. Moreover, Theorem \ref{th_ACalpha_conv} yields the typical profile of solutions across diffuse interfaces.
		\item \textit{Layout of the Proof.} 
		The new model problems, a nonlinear elliptic problem on $\R\times(0,\infty)$ and the linearized problem are considered in Section \ref{sec_hp_alpha} below. Note that interestingly, condition \eqref{eq_angle_comp1} turns out to be a necessary (and at least for $\alpha$ close to $\frac{\pi}{2}$ sufficient) condition for the solvability of the nonlinear equation, see Remark \ref{th_hp_angle_comp}. The asymptotic expansions can be found in Section \ref{sec_asym_ACalpha} and the approximate solution $u^A_{\varepsilon,\alpha}$ in Section \ref{sec_asym_ACalpha_uA}. Note that $M$ is related to the number of terms in the expansion. The spectral estimate is done in Section \ref{sec_SE_ACalpha} and the difference estimate is shown in Section \ref{sec_DC_ACalpha_DE}. Finally, Theorem \ref{th_ACalpha_conv} is proven in Section \ref{sec_DC_ACalpha_conv}. 
		\item \textit{Origin of $\overline{\alpha}_0$.} Theorem \ref{th_ACalpha_conv} is only shown for a small but uniform $\overline{\alpha}_0>0$. Let us comment at this point, where this restriction comes from. First, note that there is no restraint arising from the construction of the curvilinear coordinates in Section \ref{sec_coord2D}. The first restriction enters when we use the elliptic problem on $\R\times(0,\infty)$ from the $\frac{\pi}{2}$-case and the Implicit Function Theorem with respect to $\alpha$ to solve the model problems in Section \ref{sec_hp_alpha}. The second restriction origins from the proof of the spectral estimate in Section \ref{sec_SE_ACalpha}. The reason is that we adapt the proof from the $\frac{\pi}{2}$-case in \cite{AbelsMoser} and choose $\overline{\alpha}_0>0$ small such that similar arguments work, see also Remark \ref{th_SE_ACalpha_rem},~1. The precise restriction on $\overline{\alpha}_0$ is manifested in Remark \ref{th_asym_ACalpha_decay_param} and Theorem \ref{th_SE_ACalpha}.
		\item \textit{Well-Posedness of \hyperlink{ACalpha}{(AC$_\alpha$)}.} In Theorem \ref{th_ACalpha_conv} existence of solutions $u_{\varepsilon,\alpha}\in C^2(\overline{Q_{T_0}})$ of \hyperlink{ACalpha}{(AC$_\alpha$)} is assumed, but this is in principle well-known. The nonlinear boundary condition makes the analysis more difficult compared to the case $\alpha=\frac{\pi}{2}$. Nevertheless, $\sigma_\alpha'$ is zero outside $(-1,1)$ and one can still obtain a priori boundedness of classical solutions, see Section \ref{sec_DC_prelim_bdd_scal} below. One possibility is to construct weak solutions via time-discretization. Moreover, equation \eqref{eq_ACalpha1}-\eqref{eq_ACalpha3} fits for example in the abstract semigroup setting of Lunardi \cite{LunardiOptReg}, Section 8.5.3. There local well-posedness in a Hölder-setting is obtained by linearization at the initial value. Higher regularity and smoothness can be obtained with linear theory, cf.~Lunardi, Sinestrari, von Wahl \cite{LunardiSvW}. In order to show global existence of solutions, one has to prove boundedness of the solution in the Hölder space $C^{2,\beta}(\overline{\Omega})$ with respect to time for some $\beta>0$ small. Then the existence interval can be taken uniformly. It should be possible to prove this with the a priori uniform boundedness and a bootstrap argument.
		\item The level sets $\{u^A_{\varepsilon,\alpha}=0\}$, $\{u_{\varepsilon,\alpha}=0\}$ can be seen as approximations for $\Gamma$. In the explicit construction of $u^A_{\varepsilon,\alpha}$ in Section \ref{sec_asym_ACalpha} below the error from $\{u^A_{\varepsilon,\alpha}=0\}$ to $\Gamma$ is of order $\varepsilon$. This is of interest for numerical computations, cf.~also Caginalp, Chen, Eck \cite{CaginalpChenEck}. Note that for the case $\alpha=\frac{\pi}{2}$ this approximation order is $\varepsilon^2$, if $f$ is even, cf.~\cite{AbelsMoser} and \cite{MoserACvACND}.
		\item Basically also estimates of stronger norms can be obtained in the situation of Theorem \ref{th_ACalpha_conv}, but the initial values have to satisfy better estimates. The idea is to use interpolation estimates of the already estimated norms with stronger norms that can be controlled for exact solutions by some negative $\varepsilon$-orders, cf.~Alikakos, Bates, Chen \cite{ABC}, Theorem 2.3 in the case of the Cahn-Hilliard equation. Nevertheless, this does not improve the approximation of $\Gamma$ in the sense of 5.
		\end{enumerate}
\end{Remark}
\end{proof}

\section{Notation and Function Spaces}\label{sec_fct}
We denote with $\N$ the natural numbers and $\N_0:=\N\cup\{0\}$. Moreover, $\K=\R$ or $\C$ is fixed depending on the context. Furthermore, for simplicity we write $|.|$ for the Euclidean norm in $\R^m$, $m\in\N$ and the Frobenius norm in $\R^{m\times n}$, $m,n\in\N$. We say that $\Omega\subseteq\R^n$, $n\in\N$ is a domain, if $\Omega$ is open, nonempty and connected. The symbol \enquote{$|_.$} denotes restrictions or evaluations of functions. Moreover, we use the convention that the differential operators $\nabla$, $\diverg$ and $D^2$ just on spatial variables. For some set $X$ and a normed space $Y$ we define
$B(X,Y):=\{ f:X\rightarrow Y\text{ bounded}\}$. If $X,Y$ are normed spaces over $\K$, then $\Lc(X,Y)$ is the set of bounded linear operators $T:X\rightarrow Y$. Finally, we employ the constant convention.

Let $n,k\in\N$ and let $\Omega\subseteq\R^n$ be open and nonempty. We consider a Banach space $B$ over $\K$. We use the usual notation $C^k(\Omega,B), C^k(\overline{\Omega},B), C_b^k(\Omega,B), C_b^k(\overline{\Omega},B), C^{k,\gamma}(\overline{\Omega},B)$ for continuous, $k$-times continuously differentiable and Hölder-continuous functions with values in $B$ (and variants with bounded and/or continuously extendible functions and derivatives), see \cite{MoserACvACND}, Definition 2.1 for details. If $B=\K$ and it is obvious from the context if $\K=\R$ or $\C$, then we omit $B$ in the notation. Moreover, $C_0^\infty(\Omega)$ is the set of $f\in C^\infty(\Omega,\R)$ with compact support $\text{supp}\,f\subseteq\Omega$. Furthermore, $C_0^\infty(\overline{\Omega})$ denotes $\{f|_{\overline{\Omega}}:f\in C_0^\infty(\R^n)\}$.

Let $(M,\mathcal{A},\mu)$ be a $\sigma$-finite, complete measure space and $B$ be a Banach space over $\K=\R$ or $\C$. We refer to Amann, Escher \cite{AmannEscherIII}, Chapter X, for the definitions and properties of ($\mu$- or strongly-) measurable and (Bochner-)integrable functions $f:M\rightarrow B$, the \textit{Bochner(-Lebesgue)-Integral} and the \textit{Lebesgue-spaces} $L^p(M,B)$ for $1\leq p\leq\infty$. 

\begin{Definition}\upshape
	\begin{enumerate}
	\item Let $\Omega\subseteq\R^n$, $n\in\N$ be open and nonempty. Moreover, let $k\in\N_0$, $1\leq p\leq\infty$ and $B$ be a Banach space. Then we denote with $W^{k,p}(\Omega,B)$ the standard \textit{Sobolev-spaces}, where $W^{0,p}(\Omega,B):=L^p(\Omega,B)$. Finally, let $H^k(\Omega,B):=W^{k,2}(\Omega,B)$.
	\item Let $n\in\N$. Then $H^\beta(\R^n)$ for $\beta>0$ are the usual $L^2$-\textit{Bessel-Potential spaces} and $W^{k+\mu,p}(\R^n)$ for $k\in\N_0$, $\mu\in(0,1)$ and $1\leq p<\infty$ the \textit{Sobolev-Slobodeckij spaces}.
    \end{enumerate}
\end{Definition}

If $B=\K$ and it is clear from the context if $\K=\R$ or $\C$, then we leave out $B$ in the notation. Concerning definitions and properties of these scalar-valued function spaces, for example embeddings, interpolation theorems and trace results we refer to Adams, Fournier \cite{AdamsFournier}, Alt \cite{AltFA}, Leoni \cite{Leoni} and Triebel \cite{Triebel_Interpol_Theory}, \cite{Triebel_Fct_SpacesI}. For transformation theorems and the behaviour on product sets consider also \cite{MoserACvACND}, Section 2.

Finally, let us introduce the needed spaces with exponential weight.
\begin{Definition}\label{th_exp_def1}\upshape
	Let $1\leq p\leq\infty, k\in\N_0, \mu\in(0,1)$ and $\beta, \beta_1,\beta_2\geq 0$.
	\begin{enumerate}
		\item We define and equip with canonical norms
		\begin{align*}
		L^p_{(\beta_1,\beta_2)}(\R^2_+)&:=\{u\in L^1_\text{loc}(\R^2_+):e^{\beta_1|R|+\beta_2 H}u\in L^p(\R^2_+)\},\\
		W^{k,p}_{(\beta_1,\beta_2)}(\R^2_+)&:=\{u\in L^1_\text{loc}(\R^2_+):D^\gamma u\in L^p_{(\beta_1,\beta_2)}(\R^2_+)\text{ for all }|\gamma|\leq k\}.
		\end{align*}
		Occasionally, we replace $W^{k,2}$ with $H^k$. Let $C^k_{(\beta_1,\beta_2)}(\overline{\R^2_+}):=C_b^k(\overline{\R^2_+})\cap W^{k,\infty}_{(\beta_1,\beta_2)}(\R^2_+)$.
		\item Similarly we define $L^p_{(\beta)}(\R)$, $L^p_{(\beta)}(\R_+)$, $W^{k,p}_{(\beta)}(\R)$, $W^{k,p}_{(\beta)}(\R_+)$, $C^k_{(\beta)}(\R)$, $C^k_{(\beta)}(\overline{\R_+})$. 
		\item
		We fix a smooth $\eta:\R\rightarrow\R$ such that $\eta(R)=|R|$ for all $|R|\geq\overline{R}$ and an $\overline{R}>0$. Then let
		\[
		W^{k+\mu,p}_{(\beta)}(\R):=\{u\in L^1_\text{loc}(\R): e^{\beta\eta(R)}u\in W^{k+\mu,p}(\R)\}
		\]
		for $1\leq p<\infty$ with natural norm.
	\end{enumerate}
\end{Definition}
Many properties of these spaces are similar as for (and can be shown with) the features of the unweighted spaces. See \cite{MoserACvACND}, Lemma 2.22, for some properties, in particular the Banach space property, equivalent norms, density assertions, embeddings and traces as well as an $L^2$-Poincaré inequality and a kind of reverse fundamental theorem on $\R_+$.

\section{Curvilinear Coordinates}\label{sec_coord}
We consider a bounded, smooth domain $\Omega\subseteq\R^2$ with outer unit normal $N_{\partial\Omega}$. In this section we show the existence of a curvilinear coordinate system parametrizing a neighbourhood of an appropriate evolving curve\footnote{~For the definition of an evolving hypersurface see Depner \cite{Depner}, Definition 2.31.} in $\overline{\Omega}$ with boundary contact at $\partial\Omega$ with a constant contact angle $\alpha\in(0,\pi)$. For a sketch see Figure \ref{fig_TubularNeighbNotation} above.

\subsection{Requirements for the Evolving Curve}\label{sec_coord_surface_requ}
As parametrization domain for the evolving curve we choose $I:=[-1,1]$ for simplicity. We require that there is some $X_0:I\times[0,T]\rightarrow\overline{\Omega}$ smooth such that $X_0(.,t)$ is an injective immersion for all $t\in[0,T]$. Due to technicalities, we also consider a slightly larger open interval $I_0\supset I$ and a smooth extension of $X_0$ to $\tilde{X}_0:I_0\times(-\tau_0,T+\tau_0)\rightarrow\R^2$ for some $\tau_0>0$ such that $\tilde{X}_0(.,t)$ is an injective immersion for all $t\in(-\tau_0,T+\tau_0)$. Finally, let $\tilde{I}$ be a compact interval such that $I\subsetneq\tilde{I}^\circ$ and $\tilde{I}\subsetneq I_0$. The existence of such $I_0,\tau_0,\tilde{X}_0$ follows via compactness arguments.
 
Because continuous bijections of compact into Hausdorff topological spaces are homeomorphisms, we obtain that $X_0(.,t)$ is an embedding and $\Gamma_t:=X_0(I,t)\subset\R^2$ is a smooth, compact and connected curve with boundary for all $t\in[0,T]$. Furthermore,
\[
\Gamma:=\bigcup_{t\in[0,T]}\Gamma_t\times\{t\}
\]
is a smooth evolving curve and 
\[
\overline{X}_0:=(X_0,\textup{pr}_t):I\times[0,T]\rightarrow\Gamma:(s,t)\mapsto(X_0(s,t),t)
\] 
is a homeomorphism. Let $\vec{n}:I\times[0,T]\rightarrow\R^2$ be a smooth normal field, i.e.~$\vec{n}$ is smooth and $\vec{n}(.,t)$ represents a normal field on $\Gamma_t$. Because of \cite{Depner}, Lemma 2.40 the related normal velocity is 
\[
V(s,t):=V_{\Gamma_t}(s):=\vec{n}(s,t)\cdot\partial_tX_0(s,t)\quad\text{ for }\quad(s,t)\in I\times[0,T].
\]
Furthermore, let $H(s,t):=H_{\Gamma_t}(s)$ for $(s,t)\in I\times[0,T]$ be the (mean) curvature. Using the previous definitions for $\tilde{X}_0$ on $\tilde{I}\times[-\frac{\tau_0}{2},T+\frac{\tau_0}{2}]$ we obtain extensions of $\Gamma_t, \Gamma, \vec{n}, V$ and $H$. For the normal we employ the same symbol $\vec{n}$.

Finally, we assume $(\Gamma_t)^\circ\subseteq\Omega$ and $\partial\Gamma_t\subseteq\partial\Omega$. Then we define the \textit{contact angle} of $\Gamma_t$ with $\partial\Omega$ in any boundary point $X_0(s,t), (s,t)\in\partial I\times[0,T]$ with respect to $\vec{n}(s,t)$ via
\[
|\measuredangle(N_{\partial\Omega}|_{X_0(s,t)},\vec{n}(s,t))|\in(0,\pi),
\]
where $\measuredangle(N_{\partial\Omega}|_{X_0(s,t)},\vec{n}(s,t))$ is taken in $(-\pi,\pi)$.

\subsection{Existence of Curvilinear Coordinates}\label{sec_coord2D}
Let the assumptions in the last Section \ref{sec_coord_surface_requ} hold with a constant contact angle $\alpha\in(0,\pi)$ for all times $t\in[0,T]$. We define smooth tangent and normal fields on the evolving curve $\Gamma$ by
\[
\vec{\tau}(s,t):=\frac{\partial_sX_0(s,t)}{|\partial_sX_0(s,t)|}\quad\text{ and }\quad\vec{n}(s,t):=
\begin{pmatrix}
0 & 1\\
-1 & 0
\end{pmatrix}\vec{\tau}(s,t)\quad\text{ for all }(s,t)\in I\times[0,T].
\]
The natural extensions to $\tilde{I}\times(-\tau_0,T+\tau_0)$ are denoted with the same symbols.
Moreover, the contact points are $p^\pm(t):=X_0(\pm 1,t)$ and we set $\overline{p}^\pm(t):=(p^\pm(t),t)$ for $t\in[0,T]$.

\begin{Remark}\upshape\label{th_coord2D_rem1}
We assume $|\partial_s X_0(s,t)|=1$ for all $s\in I\setminus[-\frac{1}{2},\frac{1}{2}]$ and $t\in[0,T]$. This can be achieved by reparametrization. More precisely, consider
\[
B:I\times[0,T]\rightarrow I:(s,t)\mapsto\frac{2}{L(t)}\int_{-1}^s|\partial_sX_0(\sigma,t)|\,d\sigma-1,\quad L(t):=\int_{-1}^1|\partial_sX_0(\sigma,t)|\,d\sigma.
\]
Then $B$ is smooth and $\partial_sB>0$. Hence $B(.,t)$ is invertible for all $t\in[0,T]$ and the Inverse Mapping Theorem applied to a smooth extension of $(B,\textup{pr}_t)$ on $I\times[0,T]$ yields the smoothness of the inverse in $(s,t)$. Hence $\tilde{X}_0(s,t):=X_0(B(.,t)^{-1}|_s,t)$ is a parametrization with $|\partial_s\tilde{X}_0(s,t)|\equiv L(t)/2$. Then another simple transformation yields the desired reparametrization.

The above condition on $\partial_sX_0$ is only needed for the case $\alpha\neq\frac{\pi}{2}$. More precisely, we use $|\partial_sX_0(\pm 1,t)|=1$ in this Section \ref{sec_coord2D} and for the asymptotic expansion of (AC$_\alpha$) at the contact points, see Section \ref{sec_asym_ACalpha_cp_bulk_m2}. Finally, the above condition on $\partial_sX_0$ is used for the proof of the spectral estimate for (AC$_\alpha$), see Section \ref{sec_SE_ACalpha}.
\end{Remark}

For the coordinates we choose a domain of definition that takes into account the contact angle structure. More precisely, for $\delta>0$ consider the trapeze
\begin{align}\label{eq_coord2D_trapeze}
S_{\delta,\alpha}:=\left\{
	(r,s)\in\R^2: r\in(-\delta,\delta),s\in[-1+\frac{\cos\alpha}{\sin\alpha}r,1-\frac{\cos\alpha}{\sin\alpha}r]
	\right\}
\end{align}
with upper and lower boundary 
\begin{align}\label{eq_coord2D_trapeze_bdry}
S_{\delta,\alpha}^\pm:=\left\{(r,s^\pm(r)):r\in(-\delta,\delta)\right\},\quad\text{ where }\quad s^\pm(r):=\pm1\mp\frac{\cos\alpha}{\sin\alpha}r.
\end{align}
For $\alpha=\frac{\pi}{2}$ we have $S_{\delta,\frac{\pi}{2}}=(-\delta,\delta)\times I$ and $S_{\delta,\frac{\pi}{2}}^\pm=(-\delta,\delta)\times\{\pm1\}$. 
For a sketch see Figure \ref{fig_trapez}.

\begin{figure}[H]
	\centering
	\def\svgwidth{0.4\linewidth}
	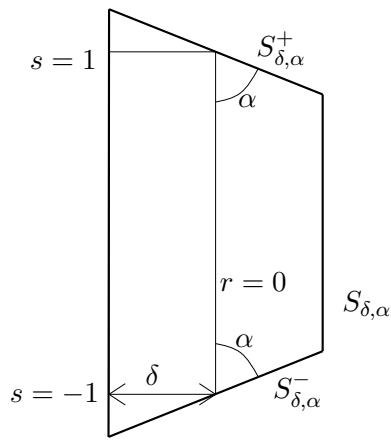
	\caption{$S_{\delta,\alpha}$ and $S_{\delta,\alpha}^\pm$.}
	\label{fig_trapez}
\end{figure}

\begin{Theorem}[\textbf{Curvilinear Coordinates}]\label{th_coord2D}
	Let the above requirements hold. Then there is a $\delta>0$ and a smooth map $\overline{S_{\delta,\alpha}}\times[0,T]\ni(r,s,t)\mapsto X(r,s,t)\in\overline{\Omega}$
	with the following features:
	\begin{enumerate} 
	\item The map $\overline{X}:=(X,\textup{pr}_t)$ is a homeomorphism onto a neighbourhood of $\Gamma$ in $\overline{\Omega}\times[0,T]$. Additionally, one can extend $\overline{X}$ to a smooth diffeomorphism mapping from an open neighbourhood of $\overline{S_{\delta,\alpha}}\times[0,T]$ in $\R^3$ onto an open set in $\R^3$. Moreover, 
	\[
	\Gamma(\tilde{\delta}):=\overline{X}(S_{\tilde{\delta},\alpha}\times[0,T])
	\]
	is an open neighbourhood of $\Gamma$ in $\overline{\Omega}\times[0,T]$ for $\tilde{\delta}\in(0,\delta]$.
	\item It holds $X|_{r=0}=X_0$ and $X$ agrees with the usual tubular neighbourhood coordinate system for $s\in[-1+\mu_0,1-\mu_0]$ for some $\mu_0\in(0,\frac{1}{2})$ small. Moreover, for $(r,s,t)\in\overline{S_{\delta,\alpha}}\times[0,T]$ the property $X(r,s,t)\in\partial\Omega$ is equivalent to $(r,s)\in\overline{S_{\delta,\alpha}^+\cup S_{\delta,\alpha}^-}$. 
	\item Denote the inverse of $\overline{X}$ with $(r,s,\textup{pr}_t)$. Then 
	\[
	|\nabla r|^2|_\Gamma=1,\quad \partial_r(|\nabla r|^2\circ\overline{X})|_{r=0}=0\quad\text{ and }\quad\nabla r\cdot\nabla s|_\Gamma=0.
	\] 
	Finally, we can achieve $\nabla s\circ\overline{X}_0=\partial_sX_0/|\partial_sX_0|^2$ and $\nabla r\circ\overline{X}_0=\vec{n}$. In this case we have $V=-\partial_tr\circ\overline{X}_0$ and $H=-\Delta r\circ\overline{X}_0$.
	\end{enumerate}
\end{Theorem}

\begin{proof}
	The proof is similar to the case $\alpha=\frac{\pi}{2}$ which was done in \cite{AbelsMoser}, Theorem 2.1. See \cite{MoserDiss}, Section 3.2, for the details. The idea is as follows. First one shows that there are graph descriptions of $\partial\Omega$ viewed from the tangential lines to $\partial\Omega$ at the contact points $p^\pm(t)$ for $t\in[0,T]$ in uniform neighbourhoods, cf.~\cite{MoserDiss}, Lemma 3.5. To use this for the definition of the curvilinear coordinate system, one has to introduce a suitable reparametrization over the upper and lower boundary of the trapeze, cf.~\cite{MoserDiss}, Lemma 3.6. For $\alpha=\frac{\pi}{2}$ this is trivial. Then the idea for the definition of $X$ is to extend the obtained mapping on the upper and lower boundary of the trapeze in such a way that it coincides with the usual tubular neighbourhood coordinate system outside a neighbourhood of the boundary and such that the claimed properties are satisfied.
\end{proof}

\begin{Remark}\phantomsection{\label{th_coord2D_rem2}}\upshape
	\begin{enumerate}
		\item Let $Q_T:=\Omega\times(0,T)$. There are unique connected $Q_T^\pm\subseteq\overline{Q_T}=\overline{\Omega}\times[0,T]$ such that $\overline{Q_T}=Q_T^-\cup Q_T^+\cup\Gamma$ (disjoint) and $\textup{sign}\,r=\pm 1$ on $Q_T^\pm\cap\Gamma(\delta)$. Additionally, let
		\begin{align*}
		\Gamma^\pm(\tilde{\delta},\mu)&:=\overline{X}((S_{\tilde{\delta},\alpha}^\circ\cap\{\pm s> 1-\mu\})\times[0,T]),\\
		\Gamma(\tilde{\delta},\mu)&:=\Gamma(\delta)\setminus[\overline{\Gamma^-(\tilde{\delta},\mu)\cup\Gamma^+(\tilde{\delta},\mu)}]=\overline{X}((-\tilde{\delta},\tilde{\delta})\times(-1+\mu,1-\mu)\times[0,T]) 
		\end{align*}
		for $\tilde{\delta}\in(0,\delta]$ and $\mu\in(0,1]$. For $t\in[0,T]$ fixed we define $\Gamma_t(\tilde{\delta}), \Gamma^\pm_t(\tilde{\delta},\mu)$ and $\Gamma_t(\tilde{\delta},\mu)$ to be the respective sets obtained by intersection with $\R^2\times\{t\}$ and then projection to $\R^2$. Here $\Gamma(\tilde{\delta})$ is as in Theorem \ref{th_coord2D}.
		
		\item Let $\tilde{\delta}\in(0,\delta]$. For a sufficiently smooth $\psi:\Gamma(\tilde{\delta})\rightarrow\R$ we introduce the \emph{tangential} and \emph{normal derivative} via 
		\[
		\nabla_\tau\psi:=\nabla s[\partial_s(\psi\circ\overline{X})\circ\overline{X}^{-1}]\quad\text{ and }\quad\partial_n\psi:=\partial_r(\psi\circ\overline{X})\circ\overline{X}^{-1},
		\] 
		respectively. In the part of $\Gamma(\delta)$ where the coordinate system is the orthogonal one, these definitions coincide with the ones in Abels, Liu \cite{ALiu}:
		\[
		\nabla_\tau\psi=\frac{\nabla s}{|\nabla s|}\frac{\nabla s}{|\nabla s|}\cdot\nabla\psi\quad\text{ and }\quad\partial_n\psi=\nabla r\cdot\nabla\psi\quad\text{ on }\Gamma(\tilde{\delta},\mu_0).
		\]
		This follows from $\nabla\psi|_{\overline{X}}=\nabla r|_{\overline{X}}\partial_r(\psi\circ\overline{X})+\nabla s|_{\overline{X}}\partial_s(\psi\circ\overline{X})$. For $t\in[0,T]$ fixed and $\psi:\Gamma_t(\tilde{\delta})\rightarrow\R$ smooth enough, we define $\nabla_\tau\psi$ and $\partial_n\psi$ analogously. In the orthogonal region similar identities as above hold. The same notation is used when $\psi$ is only defined on open subsets of $\Gamma(\tilde{\delta})$ or $\Gamma_t(\tilde{\delta})$, $t\in[0,T]$. The same properties as before are valid in the orthogonal parts of the coordinate system.
		\item For transformation arguments later we set 
		\[
		J(r,s,t):=J_t(r,s):=|\det D_{(r,s)}X(r,s,t)|\quad\text{ for } \quad(r,s,t)\in\overline{S_{\delta,\alpha}}\times[0,T].
		\] 
		$J$ is smooth and $c\leq J\leq C$ for some $c,C>0$. Moreover, from the proof of Theorem \ref{th_coord2D} it follows that
		\[
		J_t(r,s)^{-2}=\left[|\nabla r|^2|\nabla s|^2-(\nabla r\cdot\nabla s)^2\right]|_{\overline{X}(r,s,t)},
		\] 
		in particular $J_t(0,s)=|\partial_sX_0(s,t)|$ for all $(s,t)\in I\times[0,T]$. 
	\end{enumerate}
\end{Remark}

\section{Model Problems}\label{sec_model_problems}

\subsection{Some Scalar-valued ODE Problems on $\R$}\label{sec_ODE_scalar}
In this section we summarize existence and regularity results required for ODEs appearing in the inner asymptotic expansion for (AC$_\alpha$), $\alpha\in(0,\pi)$. Here we will only need \eqref{eq_AC_fvor1} for the potential $f$.
 
\subsubsection{The ODE for the Optimal Profile}\label{sec_ODE_scalar_nonlin}
We consider the nonlinear ODE
\begin{align}\label{eq_theta_0_ODE}
-w''+f'(w)=0,\quad w(0)=0,\quad \lim_{z\rightarrow\pm\infty}w(z)=\pm1.
\end{align}
\begin{Theorem}\label{th_theta_0}
	Let $f:\R\rightarrow\R$ satisfy \eqref{eq_AC_fvor1}. Then there is a unique solution $\theta_0\in C^2(\R)$ of \eqref{eq_theta_0_ODE}. Additionally, $\theta_0$ is smooth, $\theta_0'=\sqrt{2(f(\theta_0)-f(-1))}>0$ and
	\[
	D_z^k(\theta_0\mp 1)(z)=\Oc(e^{-\beta|z|})\quad\text{ for }z\rightarrow\pm\infty\text{ and all }k\in\N_0, \beta\in\left(0,\sqrt{\min\{f''(\pm1)\}}\right).
	\]
\end{Theorem}
\begin{proof}
	This follows from the proof of Schaubeck \cite{Schaubeck}, Lemma 2.6.1.
\end{proof}

$\theta_0$ is known as the \emph{optimal profile}. A rescaled variation will represent the typical profile of the solutions for \eqref{eq_ACalpha1}-\eqref{eq_ACalpha3} from Section \ref{sec_ACalpha} across the diffuse interface away from the contact points, see Section \ref{sec_asym_ACalpha_in} below. If $f$ is even, then $\theta_0$ is even, $\theta_0'$ is odd and $\theta_0''$ even and so on. In the case of the standard double-well potential $f(u)=\frac{1}{2}(1-u^2)^2$ depicted in Figure \ref{fig_double_well}, the optimal profile is given by $\theta_0=\tanh$, cf.~Figure \ref{fig_opt_profile}.

\begin{figure}[H]
	\centering
	\def\svgwidth{\linewidth}
	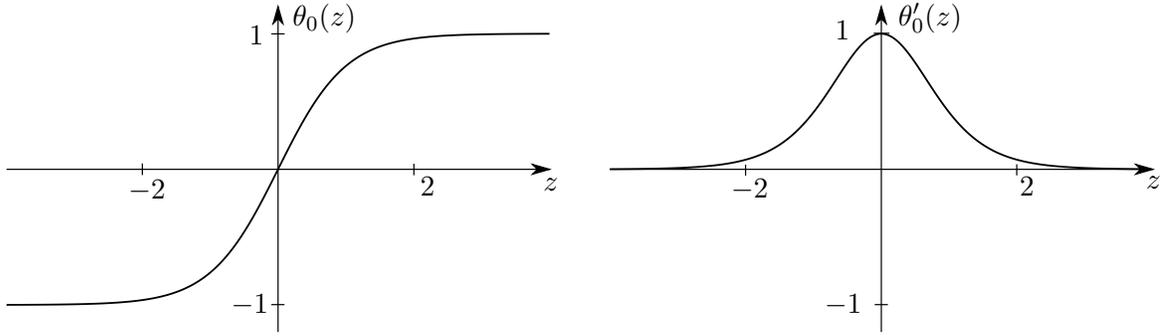
	\caption{Typical optimal profile $\theta_0=\tanh$ and the derivative $\theta_0'$.}\label{fig_opt_profile}
\end{figure}

\subsubsection{The Linearized ODE}\label{sec_ODE_scalar_lin}
Now we consider the linear ODE obtained via linearization of \eqref{eq_theta_0_ODE} at $\theta_0'$, i.e.
\begin{align}\label{eq_ODE_lin}
-w''+f''(\theta_0)w=A\quad\text{ in }\R,\quad w(0)=0.
\end{align}
We have the following solvability theorem.
\begin{Theorem}\phantomsection{\label{th_ODE_lin}}
	\begin{enumerate}
	\item Let $A\in C_b^0(\R)$. Then \eqref{eq_ODE_lin} has a solution $w\in C^2(\R)\cap C_b^0(\R)$ if and only if $\int_{\R}A\theta_0'\,dz=0$. In this  case $w$ is unique. Additionally, let $A(z)-A^\pm=\Oc(e^{-\beta|z|})$ for $z\rightarrow\pm\infty$ hold for a $\beta\in(0,\sqrt{\min\{f''(\pm1)\}})$, then 
	\[
	D_z^{l}\left[w-\frac{A^\pm}{f''(\pm 1)}\right]=\Oc(e^{-\beta|z|})\quad\text{ for }z\rightarrow\pm\infty,\quad l=0,1,2.
	\]
	\item Let $U\subseteq\R^d$ be arbitrary (e.g.~a point) and $A:\R\times U\rightarrow\R$, $A^\pm:U\rightarrow\R$ be smooth (i.e.~locally smooth extendible) and satisfy uniformly in $U$:
	\[
	D_x^k D_z^l\left[A(z,.)-A^\pm\right]=\Oc(e^{-\beta|z|})\quad\text{ for }z\rightarrow\pm\infty,\quad k=0,...,K, l=0,...,L,
	\] 
	for a $\beta\in(0,\sqrt{\min\{f''(\pm1)\}})$ and $K,L\in\N_0$. Then $w:\R\times U\rightarrow\R$, where $w(.,x)$ is the solution of \eqref{eq_ODE_lin} for $A(.,x)$ for all $x\in U$, is also smooth and fulfills uniformly in $U$
	\[
	D_x^k D_z^l\left[w(z,.)-\frac{A^\pm}{f''(\pm1)}\right]=\Oc(e^{-\beta|z|})\quad\text{ for }z\rightarrow\pm\infty, m=0,...,K, l=0,...,L+2.
	\] 
\end{enumerate}
\end{Theorem}
In our case $A^\pm =0$ will be sufficient. 

\begin{proof}
	See \cite{MoserACvACND}, Theorem 4.2 and \cite{Schaubeck}, Lemma 2.6.2.
\end{proof}

\subsection{Nonlinear Elliptic Problem on $\R^2_+$ and the Linearized Problem}\label{sec_hp_alpha}
Let $f:\R\rightarrow\R$ be as \eqref{eq_AC_fvor1}, $\alpha\in(0,\pi)$, $\hat{\sigma}:\R\rightarrow\R$ and $\sigma_\alpha=\cos\alpha\,\hat{\sigma}$ be as in Definition \ref{th_ACalpha_sigma_def}. In the contact point expansion for (AC$_\alpha$) we have to solve the following model problems:

\subsubsection{The Nonlinear Elliptic Problem on $\R^2_+$}\label{sec_hp_alpha_nonlin}
Find a smooth $v_\alpha:\overline{\R^2_+}\rightarrow\R$ such that with {\small$A_\alpha:=\begin{pmatrix}
1 & -\cos\alpha\\
-\cos\alpha & 1
\end{pmatrix}$} and $\theta_0$ as in Theorem \ref{th_theta_0} it holds
\begin{alignat}{2}\label{eq_hp_alpha_modelN1}
	-\diverg A_\alpha\nabla v_\alpha + f'(v_\alpha) &=0&\quad&\text{ for }(R,H)\in\R^2_+,\\\label{eq_hp_alpha_modelN2}
	N_{\partial\R^2_+}\cdot A_\alpha\nabla v_\alpha|_{H=0} + \sigma_\alpha'(v_\alpha)|_{H=0} &=0&\quad&\text{ for }R\in\R,\\\label{eq_hp_alpha_modelN3}
	\partial_R^{k}\partial_H^l[v_\alpha(R,H)-\theta_0(R)]&=\Oc(e^{-c_{k,l}(|R|+H)})&\quad&\text{ for all }k,l\in\N_0.
\end{alignat}
Here $N_{\partial\R^2_+}=(0,-1)^\top$. We choose $v_\frac{\pi}{2}(R,H)=\theta_0(R)$ for all $(R,H)\in\overline{\R^2_+}$.

\begin{Remark}[\textbf{Compatibility Condition for $\sigma_\alpha$}]\upshape\label{th_hp_angle_comp}
The condition \eqref{eq_angle_comp1} on $\alpha,\sigma_\alpha,f$ can be derived as a necessary condition for the existence of a smooth solution $v_\alpha$ of \eqref{eq_hp_alpha_modelN1}-\eqref{eq_hp_alpha_modelN3}. 

This can be seen as follows: Let $\sigma_\alpha:\R\rightarrow\R$ be smooth with\footnote{\label{foot_form}~In this remark the special form in Definition \ref{th_ACalpha_sigma_def} is not needed.} $\supp\,\sigma_\alpha'\subset(-1,1)$ and $v_\alpha$ sufficiently smooth solve \eqref{eq_hp_alpha_modelN1}-\eqref{eq_hp_alpha_modelN3}, where $v_{\frac{\pi}{2}}:=\theta_0$. We multiply \eqref{eq_hp_alpha_modelN1} with $\partial_R v_\alpha$ and get
\[
\frac{d}{dR}\left[\frac{1}{2}(\partial_Rv_\alpha)^2-\frac{1}{2}(\partial_Hv_\alpha)^2-f(v_\alpha)\right]+\frac{d}{dH}\left[\partial_Rv_\alpha\partial_Hv_\alpha-\cos\alpha(\partial_Rv_\alpha)^2\right]=0.
\]
Integrating with respect to $R$ over $\R$ as well as $H$ over $(0,H_0)$ for an arbitrary $H_0>0$ implies 
\[
0=\int_\R[\partial_Rv_\alpha(\partial_Hv_\alpha-\cos\alpha\partial_Rv_\alpha)]_{H=0}^{H_0}\,dR.
\]
By the boundary condition \eqref{eq_hp_alpha_modelN2} it holds $\partial_Hv_\alpha-\cos\alpha\partial_Rv_\alpha|_{H=0}=\sigma_\alpha'(v_\alpha)$. Therefore
\[
0=\int_\R[\partial_Rv_\alpha(\partial_Hv_\alpha-\cos\alpha\partial_Rv_\alpha)]|_{H_0}\,dR- \int_\R\frac{d}{dR}[\sigma_\alpha(v_\alpha|_{H=0})]\,dR.
\]
Using the asymptotics in \eqref{eq_hp_alpha_modelN3} we obtain $0=-\cos\alpha\int_{\R}(\theta_0')^2\,dR-[\sigma_\alpha(1)-\sigma_\alpha(-1)]$ by sending $H_0\rightarrow\infty$. Hence $\sigma_\alpha$ has to fulfil
\begin{align}\label{eq_angle_comp2}
	\cos\alpha=\frac{\sigma_\alpha(-1)-\sigma_\alpha(1)}{\int_\R(\theta_0')^2\,dR}.
\end{align}
Because of Theorem \ref{th_theta_0} it holds 
\[
\int_\R(\theta_0')^2=\int_\R \theta_0'\sqrt{2(f(\theta_0)-f(-1))}=\int_{-1}^1\sqrt{2(f(r)-f(-1))}\,dr
\]
and therefore \eqref{eq_angle_comp2} is equivalent to \eqref{eq_angle_comp1}.
\end{Remark}

\subsubsection{The Linearized Elliptic Problem on $\R^2_+$}\label{sec_hp_alpha_lin} Let $A_\alpha$ be as in the last Section \ref{sec_hp_alpha_nonlin} and $v_\alpha$ be a sufficiently smooth solution to \eqref{eq_hp_alpha_modelN1}-\eqref{eq_hp_alpha_modelN3}, where $v_{\frac{\pi}{2}}=\theta_0$. The linearized problem reads as follows: For $G:\overline{\R^2_+}\rightarrow\R$ and $g:\R\rightarrow\R$ with suitable regularity and exponential decay find a solution $u:\overline{\R^2_+}\rightarrow\R$ with similar decay to
\begin{alignat}{2}\label{eq_hp_alpha_modelL1}
-\diverg A_\alpha\nabla u + f''(v_\alpha)u &=G & \quad &\text{ for }(R,H)\in\R^2_+,\\\label{eq_hp_alpha_modelL2}
N_{\partial\R^2_+}\cdot A_\alpha\nabla u|_{H=0} + \sigma_\alpha''(v_\alpha)u|_{H=0} &=g& &\text{ for }R\in\R.
\end{alignat}
\begin{Remark}[\textbf{Compatibility Condition for the Data}]\label{th_hp_alpha_lin_comp}\upshape 
Let $\sigma_\alpha:\R\rightarrow\R$ be smooth with$^\text{\ref{foot_form}}$ $\supp\,\sigma_\alpha'\subset(-1,1)$ and $v_\alpha$ sufficiently smooth solve \eqref{eq_hp_alpha_modelN1}-\eqref{eq_hp_alpha_modelN3}, where $v_{\frac{\pi}{2}}:=\theta_0$. There is a necessary condition on the data $G,g$ for a solution $u$ of \eqref{eq_hp_alpha_modelL1}-\eqref{eq_hp_alpha_modelL2} to exist:
\begin{align}\label{eq_hp_alpha_lin_comp}
\int_{\R^2_+}G\partial_Rv_\alpha + \int_\R g\partial_Rv_\alpha|_{H=0}=0.
\end{align}
The condition can be derived by testing \eqref{eq_hp_alpha_modelL1} with $\partial_Rv_\alpha$, integration by parts and using \eqref{eq_hp_alpha_modelN1}-\eqref{eq_hp_alpha_modelN2} for $v_\alpha$, see \cite{MoserDiss}, Remark 4.14 for details.
\end{Remark}

\subsubsection{Solution of the Problems for $\alpha$ close to $\frac{\pi}{2}$}\label{sec_hp_alpha_sol}

An energy approach for the nonlinear problem \eqref{eq_hp_alpha_modelN1}-\eqref{eq_hp_alpha_modelN3} is difficult because the domain $\R^2_+$ is unbounded, the solution has non-trivial asymptotics and the energy of the expected solution is infinite. Therefore we take the angle $\alpha$ close to $\frac{\pi}{2}$ and choose the following strategy: we treat $\alpha$ as a parameter in the equations and use the functional analytic setting with exponentially weighted Sobolev spaces in Definition \ref{th_exp_def1}. This framework allows for isomorphisms between the solution and the data for the linear problem in the case $\alpha=\frac{\pi}{2}$, cf.~\cite{MoserACvACND}, Section 4.2.2, in particular \cite{MoserACvACND}, Theorem 4.8. This is because for data $G\in L^2(\R^2_+)=L^2(\R_+,L^2(\R))$ and $g\in L^2(\R)$ orthogonal with respect to $\theta_0'$ in $L^2(\R)$ (for a.e.~$H>0$) one obtains solution operators via weak solutions and the Lax-Milgram-Theorem (and with perturbation arguments also for exponentially weighted Sobolev spaces with small decay parameter). Here the spectral properties of the 1D-operator $L_0:H^2(\R)\subseteq L^2(\R)\rightarrow L^2(\R):u\mapsto[-\frac{d^2}{dR^2}+f''(\theta_0)]u$ are very important. See \cite{MoserDiss}, Lemma 4.2 or \cite{AbelsMoser}, Lemma 2.5 for the latter. For the remaining \enquote{parallel parts} of the data there is the explicit solution formula \cite{MoserACvACND}, (4.7), provided that the data satisfy the compatibility condition \cite{MoserACvACND}, (4.5). The setting with exponentially weighted Sobolev spaces is tailored for this explicit solution formula since $\partial_H: H^1_{(\beta)}(\R_+)\rightarrow L^2_{(\beta)}(\R_+)$ is an isomorphism for $\beta>0$ with inverse $v\mapsto -\int_.^\infty v\,ds$, cf.~\cite{MoserACvACND}, Lemma 2.22, 6.-7. 

With \cite{MoserACvACND}, Theorem 4.8 one can solve the nonlinear problem \eqref{eq_hp_alpha_modelN1}-\eqref{eq_hp_alpha_modelN3} with the Implicit Function Theorem and the linear problem \eqref{eq_hp_alpha_modelL1}-\eqref{eq_hp_alpha_modelL2} with a Neumann series argument, both for $\alpha$ close to $\frac{\pi}{2}$. See \cite{MoserDiss}, Section 4.2.2.4, for the details and Sections \ref{sec_hp_alpha_solN}-\ref{sec_hp_alpha_solL} below for the results. Here a problem to overcome is the compatibility condition \cite{MoserACvACND}, (4.5). In \cite{MoserDiss} this is dealt with in a similar way as in the proof of Theorem 4.7 in \cite{MoserACvACND}, i.e.~first one subtracts suitable terms in the boundary parts such that \cite{MoserACvACND}, (4.5), is fulfilled, but in the end one shows that those terms have to be zero for the solutions. The latter involves similar computations as in the derivations of the compatibility conditions in Remarks \ref{th_hp_angle_comp}-\ref{th_hp_alpha_lin_comp}. Furthermore, one also has to spend some thoughts on the regularity parameters in the weighted Sobolev spaces $m\in\N_0$ that one uses for the isomorphisms in the case $\alpha=\frac{\pi}{2}$ in \cite{MoserACvACND}, Theorem 4.8. More precisely, one can only apply the Implicit Function Theorem and the Neumann series argument in such a setting for finitely many $m$ since otherwise the possible angles $\alpha$ depend on $m$. Moreover, $m$ should be taken as low as possible to reduce the computations. In \cite{MoserDiss} this is done for the linear problem for $m=0$ and $m=1$ in order to have a \enquote{regularity theory} in exponentially weighted spaces due to uniqueness. It turns out that $m=1$ for the nonlinear problem is enough to subsequently use this \enquote{regularity theory} for derivatives of $v_\alpha$ and to rigorously carry out computations as in Remarks \ref{th_hp_angle_comp}-\ref{th_hp_alpha_lin_comp}. Then one uses induction arguments.

\paragraph{The Nonlinear Problem}\label{sec_hp_alpha_solN}
Let $\hat{v}_\alpha:=v_\alpha-\theta_0$. Then the nonlinear equations \eqref{eq_hp_alpha_modelN1}-\eqref{eq_hp_alpha_modelN2} for $v_\alpha$ are equivalent to 
\begin{align}\label{eq_hp_alpha_solN1}
L_\frac{\pi}{2}\hat{v}_\alpha&=(G_\alpha,g_\alpha)(\hat{v}_\alpha),
\end{align} 
where $L_\frac{\pi}{2}:=(-\Delta+f''(\theta_0(R)),-\partial_H|_{H=0})$ and 
\begin{align*}
G_\alpha(v)&:=-2\cos\alpha \partial_R\partial_H v - [f'(\theta_0+v)-f'(\theta_0)-f''(\theta_0)v],\\
g_\alpha(v)&:=-\cos\alpha [\partial_Rv|_{H=0}+\theta_0'] - \sigma_\alpha'(\theta_0 + v)|_{H=0}.
\end{align*}
The first solution theorem for the nonlinear equations \eqref{eq_hp_alpha_modelN1}-\eqref{eq_hp_alpha_modelN2} for $\alpha$ close to $\frac{\pi}{2}$ is as follows.
\begin{Theorem}\label{th_hp_alpha_solN3}
	There are $\overline{\gamma}>0$ and $\overline{\beta}:(0,\overline{\gamma}]\rightarrow(0,\infty)$ non-decreasing such that the following holds. Let $\gamma\in(0,\overline{\gamma}]$, $\beta\in(0,\min\{\overline{\beta}(\gamma),\min\{\sqrt{f''(\pm1)}\}\})$. Then there is an $\overline{\alpha}=\overline{\alpha}(\beta,\gamma)>0$ such that \eqref{eq_hp_alpha_solN1} has a solution $\hat{v}_\alpha \in H^3_{(\beta,\gamma)}(\R^2_+)$ for $\alpha\in\frac{\pi}{2}+[-\overline{\alpha},\overline{\alpha}]$ which is $C^1$ in $\alpha$, $\hat{v}_\frac{\pi}{2}=0$ and $v_\alpha:=\theta_0+\hat{v}_\alpha$ solves \eqref{eq_hp_alpha_modelN1}-\eqref{eq_hp_alpha_modelN2} and it holds $\int_\R \theta_0' \partial_Rv_\alpha|_{H=0}\,dR> 0$.
\end{Theorem}
\begin{proof}
	See \cite{MoserDiss}, Theorem 4.18. Here $\overline{\gamma}$ and $\overline{\beta}(.)$ are from the case $\alpha=\frac{\pi}{2}$, cf.~\cite{MoserACvACND}, Theorem 4.8 (same as \cite{MoserDiss}, Theorem 4.11).
\end{proof}

\begin{Remark}\upshape \label{th_hp_alpha_solN3_rem}
	From now on we fix  $\gamma_0\in(0,\overline{\gamma}]$ and $\beta_0\in(0,\min\{\overline{\beta}(\frac{\gamma_0}{2}),\min\{\sqrt{f''(\pm1)}\}\})$. 
	Moreover, we denote by $\overline{\alpha}$ and $\hat{v}_.$ the constant and the solution, respectively, obtained in Theorem \ref{th_hp_alpha_solN3} for $\beta_0,\gamma_0$. Due to the Lipschitz-continuity of $\hat{v}_.:\frac{\pi}{2}+[-\overline{\alpha},\overline{\alpha}]\rightarrow H^3_{(\beta_0,\gamma_0)}(\R^2_+)$, the identity $v_\alpha:=\theta_0+\hat{v}_\alpha$ and the embedding $H^3_{(\beta_0,\gamma_0)}(\R^2_+)\hookrightarrow C^1_{(\beta_0,\gamma_0)}(\overline{\R^2_+})$, by possibly shrinking $\overline{\alpha}$ we can assume that
	\[
	\int_\R(\partial_\rho v_\alpha)^2|_{Z}\,d\rho\in[\tfrac{3}{4},\tfrac{5}{4}]\|\theta_0'\|_{L^2(\R)}^2\text{ for }Z\geq 0
	\!\!\quad\text{ and }\quad\!\!
	\int_{\R^2_+}\partial_\rho\partial_Zv_\alpha \partial_\rho v_\alpha\,d(\rho,Z)
	\leq \frac{1}{4}\|\theta_0'\|_{L^2(\R)}^2.
	\]
	The latter estimates are not needed in this section, but they will be important for asymptotic expansions and spectral estimates later, see the end of Section \ref{sec_asym_ACalpha_cp_robin_0} and Section \ref{sec_SE_ACalpha}.
\end{Remark}

	Concerning higher regularity we have the following theorem:

\begin{Theorem}\label{th_hp_alpha_sol_reg1}
	Let $\beta_0, \gamma_0, \hat{v}_.$ be as in Remark \ref{th_hp_alpha_solN3_rem}. Then there is an $\alpha_0>0$ small such that $\hat{v}_. :\frac{\pi}{2}+[-\alpha_0,\alpha_0] \rightarrow H^k_{(\beta_0,\gamma_0)}(\R^2_+)$ is well-defined and Lipschitz-continuous for all $k\in\N_0$.
\end{Theorem}
\begin{proof}
	See \cite{MoserDiss}, Theorem 4.23.
\end{proof}

\paragraph{The Linear Problem}\label{sec_hp_alpha_solL} 
We write the linear equations \eqref{eq_hp_alpha_modelL1}-\eqref{eq_hp_alpha_modelL2} as
\begin{align}\label{eq_hp_alpha_solL1}
L_\alpha u &= (G,g), 
\quad\text{ where }L_\alpha :=L_\frac{\pi}{2} + M_\alpha,\\
M_\alpha u &:=(2\cos\alpha\partial_R\partial_H u + [f''(v_\alpha)-f''(\theta_0)]u , [\cos\alpha\partial_Ru + \sigma_\alpha''(v_\alpha) u]|_{H=0} ). \notag
\end{align}

\begin{Theorem}\label{th_hp_alpha_sol_reg3}
	Let $\beta_0, \gamma_0, \hat{v}_.$ be as in Remark \ref{th_hp_alpha_solN3_rem}. Then there is an $\alpha_0>0$ small such that for all $\alpha\in\frac{\pi}{2}+[-\alpha_0,\alpha_0]$, $\beta\in[0,\beta_0]$,  $\gamma\in[\frac{\gamma_0}{2}, \gamma_0]$ and $k\in\N_0$ it holds that 
	\[
	L_\alpha: H^{k+2}_{(\beta,\gamma)}(\R^2_+) \rightarrow 
	\left\{ (G,g)\in H^k_{(\beta,\gamma)}(\R^2_+)\times H^{k+\frac{1}{2}}_{(\beta)}(\R):
	\int_{\R^2_+} G \partial_Rv_\alpha + \int_\R g \partial_Rv_\alpha|_{H=0} =0\right\}
	\]
	is an isomorphism and the norm of the inverse is bounded independent of $\alpha,\beta,\gamma$ for fixed $k$. 
\end{Theorem}
\begin{proof}
	See \cite{MoserDiss}, Theorem 4.25.
\end{proof}

\section{Asymptotic Expansion of (AC$_\alpha$) in 2D}\label{sec_asym_ACalpha}\noindent
In this section we implement the rigorous asymptotic expansion for \hyperlink{ACalpha}{(AC$_\alpha$)} in the setting of the introduction. We build the expansion upon the curvilinear coordinates from Section \ref{sec_coord} and the model problems solved in Section \ref{sec_model_problems} appear. Therefore we will have to restrict $\alpha$ to a small interval around $\frac{\pi}{2}$ in order to use the results in Section \ref{sec_hp_alpha}.

Let $\Omega\subseteq\R^2$ be as in Remark \ref{th_intro_coord},~1.~and $\Gamma:=(\Gamma_t)_{t\in[0,T]}$ be as in Section \ref{sec_coord_surface_requ} with contact angle $\alpha\in(0,\pi)$. We employ the notation from Sections \ref{sec_coord_surface_requ}-\ref{sec_coord2D}. Moreover, let $\delta>0$ be such that the assertions of Theorem \ref{th_coord2D} hold for $\delta$ replaced by $2\delta$. In particular $(r,s):\overline{\Gamma(2\delta)}\rightarrow\overline{S_{\delta,\alpha}}$ are curvilinear coordinates that describe a neighbourhood $\overline{\Gamma(2\delta)}$ of $\Gamma$ in $\overline{\Omega}\times[0,T]$. Here $\overline{S_{\delta,\alpha}}$ is the trapeze with width $\delta$ and angle $\alpha$ defined in \eqref{eq_coord2D_trapeze}. Recall that $r$ can be viewed as a signed distance function and $s$ has the role of a tangential component, both with respect to an extension of $\Gamma$. See also Figure \ref{fig_TubularNeighbNotation} and Figure \ref{fig_trapez}. Finally, let $\Gamma$ evolve by \eqref{MCF}. Additionally, let $f$ fulfill \eqref{eq_AC_fvor1} and $\sigma_\alpha$ be as in Definition \ref{th_ACalpha_sigma_def}. We construct a smooth approximate solution $u^A_{\varepsilon,\alpha}$ to \eqref{eq_ACalpha1}-\eqref{eq_ACalpha3} for $\alpha$ close to $\frac{\pi}{2}$ with $u^A_{\varepsilon,\alpha}=\pm 1$ in $Q_T^\pm\setminus\Gamma(2\delta)$, increasingly steep \enquote{transition} from $-1$ to $1$ for $\varepsilon\rightarrow 0$ and such that $\{ u^A_{\varepsilon,\alpha}=0 \}$ converges to $\Gamma$ for $\varepsilon\rightarrow 0$. 

Roughly the idea is as follows. For the inner expansion in Section \ref{sec_asym_ACalpha_in} we can use the calculations from the case $\alpha=\frac{\pi}{2}$ in \cite{AbelsMoser}, Section 3.1. See also \cite{MoserACvACND}, Section 5.1 for the iteration of the latter expansion (also for $N\geq 2$). This formally yields a suitable approximate solution of \eqref{eq_ACalpha1} on $\{(x,t)\in\overline{\Gamma(2\delta)}:s(x,t)\in I\}$, where $I=[-1,1]$. It would not make sense to use the construction for $s\in I_\mu:=[-1-\mu,1+\mu]$ for some $\mu>0$ since then the height functions would have to satisfy a parabolic equation on $I_\mu$, but we want to impose boundary conditions at $s=\pm 1$ later. However, we can use smooth extensions from $I$ to $I_\mu$ (or $\R$) of the inner expansion terms and the height functions (cf.~\eqref{eq_asym_ACalpha_h_fct} below) obtained on $I$ for some large $\mu>0$. Then also the rescaled variable (cf.~\eqref{eq_asym_ACalpha_rho} below) from the inner expansion is well-defined close to the contact points. But we can only use the estimate on the approximation error for the inner expansion for $s\in I$. Therefore we have to cut off in an appropriate way. If the latter is done $\varepsilon$-independent, then it is difficult to set up a straight-forward ansatz at the boundary points: For the contact point expansion it is natural to addionally rescale $z^\pm_\alpha:=-r\cos\alpha + (1\mp s)\sin\alpha$ which runs in $\R_+$. Since one has to match the inner and the contact point expansion in every $\varepsilon$-order, this would lead to ansatz functions in $(\rho,Z,t)\in\overline{\R^2_+}\times[0,T]$ having non-trivial asymptotic properties for $Z\rightarrow\infty$. However, when using Taylor expansions for \eqref{eq_ACalpha1} this behaviour is a problem, since some of the appearing polynomials will not be multiplied with suitable decaying terms. Therefore the idea is to cut-off the inner expansion with appropriate functions depending on the $\varepsilon$-scaled variables. The contact point expansion is done in Section \ref{sec_asym_ACalpha_cp} and leads to the model problems on $\R^2_+$ we considered in Section \ref{sec_hp_alpha}. In order to use these results we will have to restrict to $\alpha\in\frac{\pi}{2}+[-\alpha_0,\alpha_0]$, where $\alpha_0>0$ is determined in Remark \ref{th_asym_ACalpha_decay_param} below. The compatibility condition \eqref{eq_hp_alpha_lin_comp} will yield the boundary conditions for the height functions at $s=\pm 1$. Altogether for $\alpha\in\frac{\pi}{2}+[-\alpha_0,\alpha_0]$ we obtain a suitable approximate solution $u^A_{\varepsilon,\alpha}$ to \eqref{eq_ACalpha1}-\eqref{eq_ACalpha3}, see Section \ref{sec_asym_ACalpha_uA} below. 

Let $M\in\N$ with $M\geq 2$. For $j=1,...,M$ we introduce height functions
\begin{align}\label{eq_asym_ACalpha_h_fct}
h_{j,\alpha}:I_\mu\times[0,T]\rightarrow\R
\quad\text{ and }\quad h_{\varepsilon,\alpha}:=\sum_{j=1}^M \varepsilon^{j-1}h_{j,\alpha}
\end{align}
for some $\mu>0$, where $I_\mu:=[-1-\mu,1+\mu]$. Furthermore, we set $h_{M+1,\alpha}:=h_{M+2,\alpha}:=0$ and analogously to the case $\alpha=\frac{\pi}{2}$ in \cite{AbelsMoser}
\begin{align}\label{eq_asym_ACalpha_rho}
\rho_{\varepsilon,\alpha}(x,t):=\frac{r(x,t)}{\varepsilon}-h_{\varepsilon,\alpha}(s(x,t),t)\quad\text{ for }(x,t)\in\overline{\Gamma(2\delta)}.
\end{align} 
If $\mu>0$ is large enough, the latter is well-defined.

\subsection{Inner Expansion of (AC$_\alpha$) in 2D}\label{sec_asym_ACalpha_in}
For the inner expansion we consider the same ansatz as in \cite{AbelsMoser}, Section 3.1, and \cite{MoserACvACND}, Section 5.1 for $N=2$. Let $\varepsilon>0$ be small and
\[
u^I_{\varepsilon,\alpha}:=\sum_{j=0}^{M+1}\varepsilon^ju_{j,\alpha}^I,\quad u_{j,\alpha}^I|_{(x,t)}:=\hat{u}_{j,\alpha}^I(\rho_{\varepsilon,\alpha}|_{(x,t)},s|_{(x,t)},t)\quad\text{ in }\{(x,t)\in\overline{\Gamma(2\delta)}:s|_{(x,t)}\in I\},
\]
where $\hat{u}_{j,\alpha}^I:\R\times I\times[0,T]\rightarrow\R:(\rho,s,t)\mapsto \hat{u}_{j,\alpha}^I(\rho,s,t)$ for $j=0,...,M+1$. Moreover, we set $u^I_{M+2,\alpha}:=0$ and $\hat{u}^I_{\varepsilon,\alpha}:=\sum_{j=0}^{M+1}\varepsilon^j\hat{u}_{j,\alpha}^I$. We use the following notation:
\begin{Definition}[\textbf{Notation for Inner Expansion of (AC$_\alpha$) in 2D}]\upshape\phantomsection{\label{th_asym_ACalpha_in_def}}
	\begin{enumerate}
		\item We name $(\theta_0,u^I_{1,\alpha})$ the \textit{zero-th inner order} and $(h_{j,\alpha},u^I_{j+1,\alpha})$ the \textit{$j$-th inner order} for $j=1,...,M$. 
		\item Let $k\in\{-1,...,M+2\}$ and $\beta>0$. Then let $R^I_{k,(\beta),\alpha}$ be the set of smooth functions $R:\R\times I\times[0,T]\rightarrow\R$ that depend just on the $j$-th inner orders for $0\leq j\leq \min\{k,M\}$ and fulfill uniformly in $(\rho,s,t)$:
		\[
		|\partial_\rho^i\partial_s^l\partial_t^n R(\rho,s,t)|=\Oc(e^{-\beta|\rho|})\quad\text{ for all }i,l,n\in\N_0.
		\]
		Finally, we analogously define $\hat{R}^I_{k,(\beta),\alpha}$ with functions $R:\R\times[0,T]\rightarrow\R$.\end{enumerate}
\end{Definition}

We expand \eqref{eq_ACalpha1} for $u_{\varepsilon,\alpha}=u^I_{\varepsilon,\alpha}$ in the same way as in \cite{MoserACvACND}, Section 5.1.1 (for $N=2$ there). This leads to
\begin{align}\label{eq_asym_ACalpha_u01}
\hat{u}^I_{0,\alpha}(\rho,s,t)=\theta_0(\rho)\quad\text{ and }\quad \hat{u}^I_{1,\alpha}(\rho,s,t)=0,
\end{align}
cf.~\cite{MoserACvACND}, Section 5.1.1.1-5.1.1.2. Moreover, from \cite{MoserACvACND}, Sections 5.1.1.3-5.1.1.4 we obtain: Inductively, if for $k=0,...,M-1$ the $j$-th inner order for $j=0,...,k$ is known, smooth and $\hat{u}_{j+1,\alpha}^I\in R^I_{j,(\beta),\alpha}$ for every $\beta\in(0,\min\{\sqrt{f''(\pm1)}\})$, then there is an equation for $h_{k+1,\alpha}$:
\begin{align}\label{eq_asym_ACalpha_in_hk}
\partial_th_{k+1,\alpha}-|\nabla s|^2|_{\overline{X}_0(s,t)}\partial_s^2h_{k+1,\alpha}+a_1\partial_sh_{k+1,\alpha} + a_0 h_{k+1,\alpha} = f_{k,\alpha}\quad\text{ in }I\times[0,T],
\end{align}
where $f_{k,\alpha}:I\times[0,T]\rightarrow\R$ is a smooth function that can be explicitly computed from the $j$-th inner orders for $0\leq j\leq k$ and $a_0, a_1$ are smooth and can be computed from the coordinates. If $h_{k+1,\alpha}$ is smooth and solves \eqref{eq_asym_ACalpha_in_hk}, then we obtain $\hat{u}^I_{k+2,\alpha}$ as the solution of
\begin{align}\label{eq_asym_ACalpha_in_uk}
-\Lc_0\hat{u}_{k+2,\alpha}^I(\rho,s,t)=R_{k+1,\alpha}(\rho,s,t)\quad\text{ for }(\rho,s,t)\in\R\times I\times[0,T],
\end{align}
where $\Lc_0:=-\partial_\rho^2+f''(\theta_0)$ and $R_{k+1,\alpha}\in R^I_{k+1,(\beta),\alpha}$ can be explicitly computed from $h_{k+1,\alpha}$ and the $j$-th inner orders for $j=0,...,k$. Here \eqref{eq_asym_ACalpha_in_hk} is equivalent to the compatibility condition for \eqref{eq_asym_ACalpha_in_uk} and Theorem \ref{th_ODE_lin} yields the solution $\hat{u}^I_{k+2,\alpha}\in R^I_{k+1,(\beta),\alpha}$ for all $\beta\in(0,\min\{\sqrt{f''(\pm1)}\})$.
\begin{Remark}\label{th_asym_ACalpha_feven_rem1}\upshape
	If $f$ is even, then it holds $f_{0,\alpha}=0$. This follows from \cite{MoserACvACND}, Remark 5.5.
\end{Remark}
\begin{Lemma}\label{th_asym_ACalpha_in}
	If for $k=0,...,M$ the $k$-th inner orders are known, smooth, $\hat{u}^I_{k+1}\in R^I_{k,(\beta),\alpha}$ for some $\beta>0$ and the equations \eqref{eq_asym_ACalpha_u01}-\eqref{eq_asym_ACalpha_in_uk} hold, then for some $c,C>0$ we have 
	\[
	\left|\partial_tu^I_{\varepsilon,\alpha}-\Delta u^I_{\varepsilon,\alpha}+\frac{1}{\varepsilon^2}f'(u^I_{\varepsilon,\alpha})\right|\leq C(\varepsilon^M e^{-c|\rho_{\varepsilon,\alpha}|}+\varepsilon^{M+1}) 
	\quad\text{ in }\{(x,t)\in\overline{\Gamma(2\delta)}:s|_{(x,t)}\in I\}.
	\]
\end{Lemma}
\begin{proof}
	This follows from the expansions and remainder estimates in \cite{MoserACvACND}, Sections 5.1.1.1-5.1.1.4.
\end{proof}

\subsection{Contact Point Expansion of (AC$_\alpha$) in 2D}\label{sec_asym_ACalpha_cp}
For the contact point expansion we define $s^\pm:=1\mp s$,
\begin{align}\label{eq_asym_ACalpha_cp_zalpha}
z^\pm_\alpha:=-r\cos\alpha + s^\pm\sin\alpha\quad\text{ and }\quad Z_{\varepsilon,\alpha}^\pm:=\frac{z^\pm_\alpha}{\varepsilon}\quad\text{ on }\overline{\Gamma(2\delta)}.
\end{align}
See also Figure \ref{fig_z_alpha_xi} below. Note that 
\begin{align}\label{eq_asym_ACalpha_cp_s1}
s=\pm 1 \mp s^\pm\quad\text{ and }\quad
s^\pm= \varepsilon\frac{1}{\sin\alpha}\left[Z_{\varepsilon,\alpha}^\pm+(\rho_{\varepsilon,\alpha}+h_{\varepsilon,\alpha})\cos\alpha\right].
\end{align}
This identity will be used later to expand $s$-dependent terms and it motivates us (see Remark \ref{th_asym_ACalpha_cp_cutoff} below) to define the following cut-off function for $u^I_{\varepsilon,\alpha}$: Let $\hat{\chi}:\R\rightarrow[0,1]$ be smooth with $\hat{\chi}(y)=0$ for $y\leq 1$ and $\hat{\chi}(y)=1$ for $y\geq 2$. Then we set for some constant $H_0\geq0$ 
\begin{alignat}{2}\label{eq_asym_ACalpha_cp_cutoff}
\hat{\chi}_\alpha(\rho,Z)&:=\hat{\chi}(Z)\,\hat{\chi}\left(\frac{1}{\sin\alpha}[Z+\rho\cos\alpha-H_0]\right)&\quad&\text{ for all }(\rho,Z)\in\overline{\R^2_+},\\ \chi_\alpha(x,t)&:=\hat{\chi}_\alpha(\rho_{\varepsilon,\alpha}(x,t),Z_{\varepsilon,\alpha}^\pm(x,t))&\quad&\text{ for all } (x,t)\in\overline{\Gamma(2\delta,1)}.\label{eq_asym_ACalpha_cp_cutoff2}
\end{alignat}
See Figure \ref{fig_z_alpha_xi} below for a sketch.

\begin{figure}[H]
	\centering
	\def\svgwidth{\linewidth}
	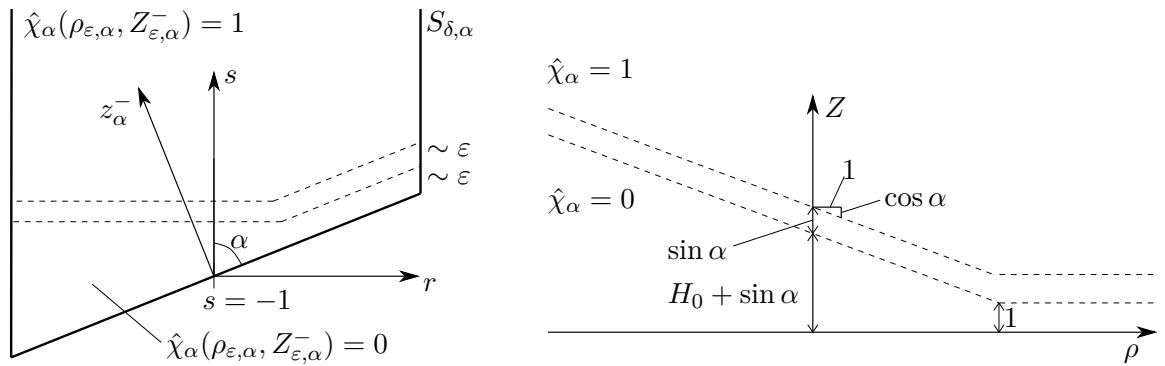
	\caption{Coordinates and cut-off $\hat{\chi}_\alpha$.}\label{fig_z_alpha_xi}
\end{figure}
For the contact point expansion we make the ansatz 
\[
u_{\varepsilon,\alpha}=\chi_\alpha u^I_{\varepsilon,\alpha}+u^{C\pm}_{\varepsilon,\alpha}\quad\text{ in }\overline{\Gamma(2\delta)}
\] 
close to the contact points $p^\pm(t), t\in[0,T]$. Here we define $u^{C\pm}_{\varepsilon,\alpha}:=\sum_{j=1}^M\varepsilon^j u_{j,\alpha}^{C\pm}$ and
\[
u_{j,\alpha}^{C\pm}(x,t):=\hat{u}_{j,\alpha}^{C\pm}(\rho_{\varepsilon,\alpha}(x,t),Z_{\varepsilon,\alpha}^\pm(x,t),t)
\]
for $(x,t)\in\overline{\Gamma(2\delta)}$, where 
\[
\hat{u}_{j,\alpha}^{C\pm}:\overline{\R^2_+}\times[0,T]\rightarrow\R:
(\rho,Z,t)\mapsto \hat{u}_{j,\alpha}^{C\pm}(\rho,Z,t)\quad\text{ for }j=1,...,M.
\] 
Moreover, we set $\hat{u}^{C\pm}_{\varepsilon,\alpha}:=\sum_{j=1}^M\varepsilon^j \hat{u}_{j,\alpha}^{C\pm}$ and $u^{C\pm}_{M+1,\alpha}:=0$.
\begin{Remark}\phantomsection{\label{th_asym_ACalpha_cp_cutoff}}\upshape\begin{enumerate}
	\item Note that $\hat{\chi}(Z_{\varepsilon,\alpha}^\pm)$ is zero on $\overline{\Gamma(2\delta)}$ close to $\partial Q_T$. Therefore there will be no contribution of $u^I_{\varepsilon,\alpha}$ in the expansion of the boundary condition \eqref{eq_ACalpha2}. Nevertheless, this is just for aesthetic reasons. However, the second factor of $\hat{\chi}_\alpha$ in \eqref{eq_asym_ACalpha_cp_cutoff} is crucial. Namely, if $h_{1,\alpha}$ is known independently of $\chi_\alpha$, then we can take $H_0:=2\|h_{1,\alpha}\|_\infty$. Then due to \eqref{eq_asym_ACalpha_cp_s1} it holds:
	\[
	\frac{1}{\sin\alpha}[Z_{\varepsilon,\alpha}^\pm+\rho_{\varepsilon,\alpha}\cos\alpha-H_0]\geq 1 \quad\!\Rightarrow\! \quad \frac{s^\pm}{\varepsilon}\geq 1+\frac{1}{\sin\alpha}[\cos\alpha\, h_{\varepsilon,\alpha}+H_0]\geq 0\!\quad\!\text{in }\overline{\Gamma(2\delta)}
	\] 
	if $\varepsilon>0$ is small depending on $\|h_{2,\alpha}\|_\infty,...,\|h_{M,\alpha}\|_\infty$ and $\alpha$. This is important since then values of $u^I_{\varepsilon,\alpha}$ are only used in the set $\{(x,t)\in\overline{\Gamma(2\delta)}:s|_{(x,t)}\in I\}$ on which we know that $u^I_{\varepsilon,\alpha}$ has appropriate decay and is (at the moment formally) a suitable approximate solution of \eqref{eq_ACalpha1}, cf.~Lemma \ref{th_asym_ACalpha_in}. 
	\item The $\varepsilon$-scaled cut-off function $\hat{\chi}(Z)\hat{\chi}(s^\pm/\varepsilon)$ should also work, but there are even more terms that have to be expanded. 
\end{enumerate}
\end{Remark}

To get an idea for the expansion of \eqref{eq_ACalpha1} for $u_{\varepsilon,\alpha}=\chi_\alpha u^I_{\varepsilon,\alpha}+u^{C\pm}_{\varepsilon,\alpha}$ in $\overline{\Gamma(2\delta)}$, we rewrite
\begin{align}\label{eq_asym_ACalpha_cp1}
0&=(\partial_t-\Delta)\left[\chi_\alpha u^I_{\varepsilon,\alpha}+u^{C\pm}_{\varepsilon,\alpha}\right]+\frac{1}{\varepsilon^2}f'(\chi_\alpha u^I_{\varepsilon,\alpha}+u^{C\pm}_{\varepsilon,\alpha})\\
&=\label{eq_asym_ACalpha_cp2}
u^I_{\varepsilon,\alpha}(\partial_t-\Delta)\chi_\alpha
+2\nabla(u^I_{\varepsilon,\alpha})\cdot\nabla(\chi_\alpha)
+\chi_\alpha\left[(\partial_t-\Delta)u^I_{\varepsilon,\alpha}+\frac{1}{\varepsilon^2}f'(u^I_{\varepsilon,\alpha})\right]\\\label{eq_asym_ACalpha_cp3}
&+(\partial_t-\Delta)u^{C\pm}_{\varepsilon,\alpha}
+\frac{1}{\varepsilon^2}\left[f'(\chi_\alpha u^I_{\varepsilon,\alpha}+u^{C\pm}_{\varepsilon,\alpha})-\chi_\alpha f'(u^I_{\varepsilon,\alpha})\right].
\end{align}
Due to Remark \ref{th_asym_ACalpha_cp_cutoff}, 1.~and Lemma \ref{th_asym_ACalpha_in} it will be possible to control the last term in \eqref{eq_asym_ACalpha_cp2} rigorously in the end. Hence this term can be left out in the expansion of \eqref{eq_asym_ACalpha_cp1}.  
Moreover, the lowest order will be important. Therefore we set
\[
w^{C\pm}_{\alpha}:=\chi_\alpha u^I_{0,\alpha}+u^{C\pm}_{0,\alpha}\quad\text{ and }\quad 
\hat{w}^{C\pm}_\alpha(\rho,Z,t):=\hat{\chi}_\alpha(\rho,Z)\theta_0(\rho)+\hat{u}^{C\pm}_{0,\alpha}(\rho,Z,t),
\] 
$\tilde{u}^I_{\varepsilon,\alpha}:=u^I_{\varepsilon,\alpha}-u^I_{0,\alpha}$, $\tilde{u}^{C\pm}_{\varepsilon,\alpha}:=u^{C\pm}_{\varepsilon,\alpha}-u^{C\pm}_{0,\alpha}$ as well as $\hat{\tilde{u}}^{C\pm}_{\varepsilon,\alpha}:=\hat{u}^{C\pm}_{\varepsilon,\alpha}-\hat{u}^{C\pm}_{0,\alpha}$.
Then we rewrite \eqref{eq_asym_ACalpha_cp1}-\eqref{eq_asym_ACalpha_cp3} without the last term in \eqref{eq_asym_ACalpha_cp2} as follows:
\begin{align}\begin{split}\label{eq_asym_ACalpha_cp4}
0&=(\partial_t-\Delta)[w^{C\pm}_\alpha+\tilde{u}^{C\pm}_{\varepsilon,\alpha}] 
-\chi_\alpha(\partial_t-\Delta)u^I_{0,\alpha}
+\tilde{u}^I_{\varepsilon,\alpha}(\partial_t-\Delta)\chi_\alpha\\
&+2\nabla(\tilde{u}^I_{\varepsilon,\alpha})\cdot\nabla(\chi_\alpha)
+\frac{1}{\varepsilon^2}\left[f'(w^{C\pm}_\alpha+\chi_\alpha\tilde{u}^I_{\varepsilon,\alpha}+\tilde{u}^{C\pm}_{\varepsilon,\alpha})-\chi_\alpha f'(u^I_{\varepsilon,\alpha})\right].\end{split}
\end{align}
We will expand the \enquote{bulk equation} \eqref{eq_asym_ACalpha_cp4} in $\overline{\Gamma(2\delta)}$ into $\varepsilon$-series with coefficients in $(\rho_{\varepsilon,\alpha},Z_{\varepsilon,\alpha}^\pm,t)$  up to $\Oc(\varepsilon^{M-2})$ and the nonlinear Robin boundary condition \eqref{eq_ACalpha2} for $u_{\varepsilon,\alpha}=\chi_\alpha u^I_{\varepsilon,\alpha}+u^{C\pm}_{\varepsilon,\alpha}$ on $\partial Q_T\cap\overline{\Gamma(2\delta)}$ into $\varepsilon$-series with coefficients in $(\rho_{\varepsilon,\alpha},t)$ up to $\Oc(\varepsilon^{M-1})$. Note that in order to yield a suitable approximate solution, the contact point expansion has to match the inner expansion. To this end we aim for
\begin{align}\label{eq_asym_ACalpha_matching1}
\partial_\rho^i\partial_H^l\partial_t^n[\hat{w}^{C\pm}_\alpha(\rho,H,t)-\theta_0(\rho)]&=\Oc(e^{-(\beta|\rho|+\gamma H)}),\\ \label{eq_asym_ACalpha_matching2}
\partial_\rho^i\partial_H^l\partial_t^n\hat{u}^{C\pm}_{j,\alpha}(\rho,H,t)&=\Oc(e^{-(\beta|\rho|+\gamma H)})
\end{align}
for $j=1,...,M$ and all $i,l,n\in\N_0$, for some $\beta,\gamma>0$ possibly depending on $j,i,l,n$. Later we will use arbitrary $\beta\in[0,\beta_0)$ and $\gamma\in[\frac{\gamma_0}{2},\gamma_0)$, where $\beta_0$, $\gamma_0$ are specified as follows:

\begin{Remark}[\textbf{Decay Parameters $\beta_0$,$\gamma_0$, Angle $\alpha_0$, Lowest Order $v_\alpha$}]\upshape\label{th_asym_ACalpha_decay_param}
	We choose $\beta_0,\gamma_0>0$ as in Remark \ref{th_hp_alpha_solN3_rem} and such that the inequality 
	\begin{align}\label{eq_asym_ACalpha_decay_param}\textstyle
	\beta_0+\gamma_0\leq\min\{\sqrt{f''(\pm 1)}\}
	\end{align} 
	holds. For these $\beta_0,\gamma_0$ we can use all the assertions in Section \ref{sec_hp_alpha}. Hence we obtain an $\alpha_0>0$ small and a solution $v_\alpha$ of \eqref{eq_hp_alpha_modelN1}-\eqref{eq_hp_alpha_modelN2} for all $\alpha\in\frac{\pi}{2}+[-\alpha_0,\alpha_0]$ such that $v_\alpha=\theta_0+\hat{v}_\alpha$ and $\hat{v}_.:\frac{\pi}{2}+[-\alpha_0,\alpha_0]\rightarrow H^k_{(\beta_0,\gamma_0)}(\R^2_+)$ is Lipschitz-continuous for all $k\in\N_0$. Moreover, due to Theorem \ref{th_hp_alpha_sol_reg3} the linearized problem \eqref{eq_hp_alpha_modelL1}-\eqref{eq_hp_alpha_modelL2} can be solved in Sobolev spaces with exponential weight with decay parameters $\beta\in[0,\beta_0]$, $\gamma\in[\frac{\gamma_0}{2},\gamma_0]$. Note that every choice in this remark is independent of $\Omega$ and $\Gamma$.
\end{Remark}

The successive requirement that the coefficients in the expansions disappear will yield equations on $\R^2_+$ of the type as in Subsection \ref{sec_hp_alpha}. It will turn out that for $\hat{w}^{C\pm}_\alpha=v_\alpha$ the lowest order vanishes. The solvability condition \eqref{eq_hp_alpha_lin_comp} for the linear problems in the higher orders will yield the boundary conditions at $s=\pm 1$ for the height functions $h_{j,\alpha}$.

For the expansion we compute in the following lemma how the differential operators act on $(\rho_{\varepsilon,\alpha},Z_{\varepsilon,\alpha}^\pm,t)$-dependent terms like $u^{C\pm}_{\varepsilon,\alpha}$, $\chi_\alpha$ or $\chi_\alpha u^I_{0,\alpha}$.

\begin{Lemma}\label{th_asym_ACalpha_cp_trafo}
	Let $\overline{\R^2_+}\times[0,T]\ni(\rho,Z,t)\mapsto\hat{w}(\rho,Z,t)\in\R$ be smooth enough and define $w:\overline{\Gamma(2\delta)}\rightarrow\R$ via $w(x,t):=\hat{w}(\rho_{\varepsilon,\alpha}(x,t),Z_{\varepsilon,\alpha}^\pm(x,t),t)$ for all $(x,t)\in\overline{\Gamma(2\delta)}$. Then
	\begin{align*}
	\partial_tw&=\partial_\rho\hat{w}\left[\frac{\partial_tr}{\varepsilon}-(\partial_th_{\varepsilon,\alpha}+\partial_ts\partial_sh_{\varepsilon,\alpha})\right]+\partial_Z\hat{w}\frac{\partial_tz^\pm_\alpha}{\varepsilon}+\partial_t\hat{w},\\
	\nabla w&=\partial_\rho\hat{w}\left[\frac{\nabla r}{\varepsilon}-\nabla s\partial_sh_{\varepsilon,\alpha}\right]+\partial_Z\hat{w}\frac{\nabla z^\pm_\alpha}{\varepsilon},\\
	\Delta w&=\partial_\rho\hat{w}\left[\frac{\Delta r}{\varepsilon}-(\Delta s\partial_sh_{\varepsilon,\alpha}+|\nabla s|^2\partial_s^2h_{\varepsilon,\alpha})\right]+\partial_Z\hat{w}\frac{\Delta z^\pm_\alpha}{\varepsilon}+\partial_Z^2\hat{w}\frac{|\nabla z^\pm_\alpha|^2}{\varepsilon^2}\\
	&+ 2\partial_\rho\partial_Z\hat{w}\,\frac{\nabla z^\pm_\alpha}{\varepsilon}\cdot\left[\frac{\nabla r}{\varepsilon}-\nabla s\partial_sh_{\varepsilon,\alpha}\right]+\partial_\rho^2\hat{w}\left|\frac{\nabla r}{\varepsilon}-\nabla s\partial_sh_{\varepsilon,\alpha}\right|^2,
	\end{align*}
	where the $w$-terms and the derivatives of $r$ or $s$ are evaluated at $(x,t)$, the $h_{\varepsilon,\alpha}$-expressions at $(s(x,t),t)$ and the $\hat{w}$-terms at $(\rho_{\varepsilon,\alpha}(x,t),Z_{\varepsilon,\alpha}^\pm(x,t),t)$.
\end{Lemma}

\begin{proof}
	This follows from the chain rule.
\end{proof}

Similar as in the case $\alpha=\frac{\pi}{2}$ in \cite{MoserACvACND} we use the following notation for the higher orders:
\begin{Definition}[\textbf{Notation for Contact Point Expansion of (AC$_\alpha$) in 2D}]\upshape\phantomsection{\label{th_asym_ACalpha_cp_def}}\qquad
\begin{enumerate}
		\item We refer to the functions $(\theta_0,u^I_{1,\alpha},w^{C\pm}_\alpha)$ as the \textit{zero-th order} and $(h_{j,\alpha},u^I_{j+1,\alpha},u^{C\pm}_{j,\alpha})$ as the \textit{$j$-th order}, where $j=1,...,M$. 
		\item Let $k\in\{-1,...,M+2\}$. We write $P^C_{k,\alpha}(\rho,Z)$ for the set of polynomials in $(\rho,Z)$ with smooth coefficients in $t\in[0,T]$ depending only on the $h_{j,\alpha}$ for $1\leq j\leq \min\{k,M\}$. In the analogous way we introduce $P^C_{k,\alpha}(\rho)$ and $P^C_{k,\alpha}(Z)$ with $\rho$ and $Z$ instead of $(\rho,Z)$, respectively.
		\item Let $k\in\{-1,...,M+2\}$ and $\beta,\gamma>0$. Let $R^C_{k,(\beta,\gamma),\alpha}$ be the set of smooth functions $R:\overline{\R^2_+}\times[0,T]\rightarrow\R$ that depend only on the $j$-th orders for $0\leq j\leq \min\{k,M\}$ and such that uniformly in $(\rho,Z,t)$:
		\[
		|\partial_\rho^i\partial_Z^l\partial_t^n R(\rho,Z,t)|=\Oc(e^{-(\beta|\rho|+\gamma Z)})\quad\text{ for all }i,l,n\in\N_0.
		\]
		Similarly we define the set $R^C_{k,(\beta),\alpha}$ without the $Z$-dependence.
	\end{enumerate}
\end{Definition}

\subsubsection{Contact Point Expansion: The Bulk Equation}\label{sec_asym_ACalpha_cp_bulk}
We rewrite \eqref{eq_asym_ACalpha_cp4} in $\overline{\Gamma(2\delta)}$ with Lemma \ref{th_asym_ACalpha_cp_trafo} as follows:
\begin{align}\begin{split}\label{eq_asym_ACalpha_cp}
0&=
(\partial_\rho\hat{w}^{C\pm}_\alpha - \hat{\chi}_\alpha\theta_0')
\left[\frac{\partial_tr-\Delta r}{\varepsilon}-(\partial_th_{\varepsilon,\alpha}+|\nabla s|^2\partial_s^2h_{\varepsilon,\alpha}+(\partial_ts-\Delta s)\partial_sh_{\varepsilon,\alpha})\right]\\  
&+\partial_Z\hat{w}^{C\pm}_\alpha\frac{(\partial_t-\Delta)z^\pm_\alpha}{\varepsilon} 
+\partial_t\hat{w}^{C\pm}_\alpha 
-\partial_Z^2\hat{w}^{C\pm}_\alpha\frac{|\nabla z^\pm_\alpha|^2}{\varepsilon^2} \\
&-2\partial_\rho\partial_Z\hat{w}^{C\pm}_\alpha\frac{\nabla z^\pm_\alpha}{\varepsilon}\cdot\left[\frac{\nabla r}{\varepsilon}-\nabla s \partial_sh_{\varepsilon,\alpha}\right]
-(\partial_\rho^2\hat{w}^{C\pm}_\alpha-\hat{\chi}_\alpha\theta_0'')\left|\frac{\nabla r}{\varepsilon}-\nabla s\partial_sh_{\varepsilon,\alpha}\right|^2\\
&+(\partial_t-\Delta)\tilde{u}^{C\pm}_{\varepsilon,\alpha}
+ \tilde{u}^I_{\varepsilon,\alpha}(\partial_t-\Delta)\chi_\alpha
+2\nabla(\tilde{u}^I_{\varepsilon,\alpha})\cdot\nabla(\chi_\alpha)\\
&+\frac{1}{\varepsilon^2}\left[f'(w^{C\pm}_\alpha+\chi_\alpha\tilde{u}^I_{\varepsilon,\alpha}+\tilde{u}^{C\pm}_{\varepsilon,\alpha})-\chi_\alpha f'(u^I_{\varepsilon,\alpha})\right],\end{split}
\end{align}
where we use the conventions for evaluations as in Lemma \ref{th_asym_ACalpha_cp_trafo}. Later we will choose $\hat{w}^{C\pm}_\alpha=v_\alpha$ such that the lowest order in the $\varepsilon$-expansion vanishes, where $v_\alpha$ is from Remark \ref{th_asym_ACalpha_decay_param}. In \eqref{eq_asym_ACalpha_cp} one can observe that the $\theta_0$-contributions are crucial for the asymptotics as $Z\rightarrow\infty$ in the $\varepsilon$-expansion, since we want exponentially decaying terms in the expansion at each order.

In the following we specify how all the terms in \eqref{eq_asym_ACalpha_cp} are expanded into $\varepsilon$-series. For the $f'$-parts: provided that $u_{j,\alpha}^I, u_{j,\alpha}^{C\pm}$ are bounded, Taylor expansions imply on $\overline{\Gamma(2\delta)}$ 
\begin{align}\notag
&f'(w^{C\pm}_\alpha+\chi_\alpha\tilde{u}^I_{\varepsilon,\alpha}+\tilde{u}^{C\pm}_{\varepsilon,\alpha})\\
=&f'(w^{C\pm}_\alpha)
+\sum_{k=1}^{M+2}\frac{1}{k!}f^{(k+1)}(w^{C\pm}_\alpha)\left[\sum_{j=1}^{M+1}\varepsilon^j(\chi_\alpha u_{j,\alpha}^I+u_{j,\alpha}^{C\pm})\right]^k+\Oc(\varepsilon^{M+3}),\label{eq_asym_ACalpha_cp_taylor_f1}
\end{align}
as well as on $\{(x,t)\in\overline{\Gamma(2\delta)}:s|_{(x,t)}\in I\}$
\begin{align}\label{eq_asym_ACalpha_cp_taylor_f2}
f'(u^I_{\varepsilon,\alpha})
=f'(\theta_0)+\sum_{k=1}^{M+2}\frac{1}{k!}f^{(k+1)}(\theta_0)\left[\sum_{j=2}^{M+1}\varepsilon^ju_{j,\alpha}^I\right]^k+\Oc(\varepsilon^{M+3}),
\end{align}
where $u^I_{1,\alpha}=0$ due to \eqref{eq_asym_ACalpha_u01}. Therefore the terms for $f'(w^{C\pm}_\alpha+\chi_\alpha\tilde{u}^I_{\varepsilon,\alpha}+\tilde{u}^{C\pm}_{\varepsilon,\alpha})-\chi_\alpha f'(u^I_{\varepsilon,\alpha})$ in the asymptotic expansion are for $k=2,...,M+2$:
\begin{align*}
\Oc(1)&:\quad f'(w^{C\pm}_\alpha)-\chi_\alpha f'(\theta_0),\\
\Oc(\varepsilon)&:\quad f''(w^{C\pm}_\alpha)u^{C\pm}_{1,\alpha}+\chi_\alpha[f''(w^{C\pm}_\alpha)-f''(\theta_0)]u^I_{1,\alpha}=f''(w^{C\pm}_\alpha)u^{C\pm}_{1,\alpha},\\
\Oc(\varepsilon^k)&:\quad f''(w^{C\pm}_\alpha)u_{k,\alpha}^{C\pm} + [\text{a polynomial in }(\chi_\alpha u_{1,\alpha}^I,...,\chi_\alpha u_{k-1,\alpha}^I, u_{1,\alpha}^{C\pm},...,u_{k-1,\alpha}^{C\pm})\\
& \phantom{:\quad f''(w^{C\pm}_\alpha)u_{k,\alpha}^{C\pm} + [\,}\text{of order }\leq k\text{, where the coefficients are multiples of }\\
& \phantom{:\quad f''(w^{C\pm}_\alpha)u^{C\pm}_{k,\alpha} + [\,} f^{(3)}(w^{C\pm}_\alpha),...,f^{(k+1)}(w^{C\pm}_\alpha) \text{ and every term has a }u^{C\pm}_{j,\alpha}\text{-factor}]\\
&\phantom{:\quad f''(w^{C\pm}_\alpha)u_{k,\alpha}^{C\pm}~}+ [\text{a polynomial in }(u_{1,\alpha}^I,...,u_{k,\alpha}^I)\text{ of order }\leq k,\text{ where the}\\
& \phantom{:\quad f''(w^{C\pm}_\alpha)u_{k,\alpha}^{C\pm} + [\,}\text{coefficients are multiples of }\chi_\alpha^l f^{(l+1)}(w^{C\pm}_\alpha)-\chi_\alpha f^{(l+1)}(\theta_0), \\
& \phantom{:\quad f''(w^{C\pm}_\alpha)u_{k,\alpha}^{C\pm} + [\,} 
l=1,...,k+1, \text{and every term has a }u^I_{j,\alpha}\text{-factor}].
\end{align*}
The remaining explicit terms in $f'(w^{C\pm}_\alpha+\chi_\alpha\tilde{u}^I_{\varepsilon,\alpha}+\tilde{u}^{C\pm}_{\varepsilon,\alpha})-\chi_\alpha f'(u^I_{\varepsilon,\alpha})$ are of order $\Oc(\varepsilon^{M+3})$.

Additionally, we expand terms in \eqref{eq_asym_ACalpha_cp}-\eqref{eq_asym_ACalpha_cp_taylor_f2} depending $(s,t)$ or $(\rho,s,t)$, namely the $h_{j,\alpha}$-terms and the $u_{j,\alpha}^I$-terms, respectively. Such terms also appear because of the product rule (cf.~Lemma \ref{th_asym_ACalpha_cp_trafo} and \cite{AbelsMoser}, Lemma 3.1) for $\chi_\alpha,\tilde{u}^I_{\varepsilon,\alpha},\tilde{u}^{C\pm}_{\varepsilon,\alpha}$. To this end let $g_1:I_\mu\times[0,T]\rightarrow\R$ or $g_1:\R\times I_\mu\times[0,T]\rightarrow\R$ be smooth with bounded derivatives in $s$. A Taylor expansion yields
\begin{align}\label{eq_asym_ACalpha_cp_taylor2}
g_1|_s=g_1|_{s=\pm 1}+\sum_{k=1}^{M+2}\frac{\partial_s^k g_1|_{s=\pm 1}}{k!}(s\mp 1)^k+\Oc(|s\mp 1|^{M+3})
\end{align}
with a uniform remainder. Then because of \eqref{eq_asym_ACalpha_cp_s1} we replace 
\begin{align}\label{eq_asym_ACalpha_cp_s2}
s\mp 1
=\mp \varepsilon\frac{1}{\sin\alpha}\left[Z+(\rho+h_{\varepsilon,\alpha}(s,t))\cos\alpha\right].
\end{align}
In particular $|s\mp 1|=\varepsilon\Oc(1+|\rho|+Z)$, if the $h_{j,\alpha}$ are bounded. On the right hand side in \eqref{eq_asym_ACalpha_cp_s2} we again have the $s$-dependent term $h_{\varepsilon,\alpha}$, but the $\varepsilon$-order has increased by one. If the $h_{j,\alpha}$ are sufficiently regular, then we can use \eqref{eq_asym_ACalpha_cp_taylor2}-\eqref{eq_asym_ACalpha_cp_s2} for the $h_{\varepsilon,\alpha}$-term inductively. The latter is needed only finitely many times and yields an expansion of $g_1$ into an $\varepsilon$-series with coefficients in $(\rho,Z,t)$ up to $\Oc(\varepsilon^{M+2})$. The terms are for $k=1,...,M+2$:
\begin{align*}\Oc(1):&\quad g_1|_{s=\pm 1},\\
\Oc(\varepsilon^k):&  \quad [\text{a polynomial in }(\rho,Z,\partial_s^lh_{j,\alpha}|_{(\pm 1,t)}), l=0,...,k-1, j=1,...,k
\text{ of order }\leq k,\\
&\quad\phantom{[} \text{where the coefficients are multiples of }\partial_s^lg_1|_{s=\pm 1}, l=1,...,k].
\end{align*}
Finally, the remainder in the expansion of $g_1$ is of order $\varepsilon^{M+3}\Oc((1+|\rho|+Z)^{M+3})$. The latter will be multiplied with decaying terms later and becomes $\Oc(\varepsilon^{M+3})$.

Furthermore, we have to expand terms in \eqref{eq_asym_ACalpha_cp} depending on $(x,t)$ that appear after applying Lemma \ref{th_asym_ACalpha_cp_trafo} and \cite{AbelsMoser}, Lemma 3.1, more precisely the derivatives of $r,s$ and $z^\pm_\alpha$. To this end let $g_2:\overline{\Gamma(2\delta)}\rightarrow\R$ be smooth. Then a Taylor expansion yields uniformly in $(r,s,t)\in\overline{S_{\delta,\alpha}}\times[0,T]$
\begin{align}\label{eq_asym_ACalpha_cp_taylor3}
\tilde{g}_2(r,s,t):=g_2(\overline{X}(r,s,t))=\sum_{j+k=0}^{M+2}\frac{\partial_r^j\partial_s^k\tilde{g}_2|_{(0,\pm1,t)}}{j!\,k!}r^j(s\mp 1)^k+\Oc(|(r,s\mp 1)|^{M+3}).
\end{align}
We substitute $r$ by $\varepsilon(\rho+h_{\varepsilon,\alpha}(s,t))$ and $s\mp 1$ by \eqref{eq_asym_ACalpha_cp_s2}. For the appearing $s$-dependent term $h_{\varepsilon,\alpha}$ we use the expansion for $g_1$ above. Hence we obtain an expansion of $g_2$ into $\varepsilon$-series with coefficients in $(\rho,Z,t)$ up to $\Oc(\varepsilon^{M+2})$. The terms in the expansion are for $k=1,...,M+2$:
\begin{align*}
\Oc(1):\quad &g_2|_{\overline{p}^\pm(t)},\\
\Oc(\varepsilon):\quad &\partial_r\tilde{g}_2|_{(0,\pm 1,t)}(\rho+h_{1,\alpha}|_{(\pm1,t)})\mp\partial_s\tilde{g}_2|_{(0,\pm 1,t)}\frac{1}{\sin\alpha}\left[Z+(\rho+h_{1,\alpha}|_{(\pm 1,t)})\cos\alpha\right],\\
\Oc(\varepsilon^k):\quad 
& h_{k,\alpha}|_{(\pm 1,t)} \left[
\partial_r\tilde{g}_2|_{(0,\pm 1,t)}\mp\frac{\cos\alpha}{\sin\alpha}\partial_s\tilde{g}_2|_{(0,\pm 1,t)}\right]\\
\quad+&[\text{a polynomial in }(\rho,Z,\partial_s^l h_{j,\alpha}|_{(\pm 1,t)}), l=0,...,k-1, j=1,...,k-1\text{ of order}\leq k,\\
\quad&\phantom{[}\text{where the coefficients are multiples of }\partial_r^{l_1}\partial_s^{l_2}\tilde{g}_2|_{(0,\pm 1,t)}, l_1,l_2\in\N_0, l_1+l_2\leq k].
\end{align*}
Here in contrast to the case $\alpha=\frac{\pi}{2}$ in \cite{MoserACvACND}, Section 5.1.2 we need the $\Oc(\varepsilon)$ explicitly. The remainder in the expansion for $g_2$ is  $\varepsilon^{M+3}\Oc((1+|\rho|+Z)^{M+3})$. The latter will be multiplied with exponentially decaying terms later and becomes $\Oc(\varepsilon^{M+3})$.

Now we expand \eqref{eq_asym_ACalpha_cp} with the above identities into $\varepsilon$-series with coefficients in $(\rho_{\varepsilon,\alpha},Z_{\varepsilon,\alpha}^\pm,t)$. 

\paragraph{Bulk Equation: $\Oc(\varepsilon^{-2})$}\label{sec_asym_ACalpha_cp_bulk_m2}
The lowest order $\Oc(\frac{1}{\varepsilon^2})$ in  \eqref{eq_asym_ACalpha_cp} vanishes if 
\begin{align*}
0&=-\partial_Z^2\hat{w}^{C\pm}_\alpha|\nabla z^\pm_\alpha|^2|_{\overline{p}^\pm(t)} 
- 2\partial_Z\partial_\rho\hat{w}^{C\pm}_\alpha \nabla z^\pm_\alpha\cdot\nabla r|_{\overline{p}^\pm(t)}
-(\partial_\rho^2\hat{w}^{C\pm}_\alpha - \hat{\chi}_\alpha\theta_0'')|\nabla r|^2|_{\overline{p}^\pm(t)}\\
&+f'(\hat{w}^{C\pm}_\alpha)-\hat{\chi}_\alpha f'(\theta_0).
\end{align*}
Now we use that due to Remark \ref{th_coord2D_rem1} and Theorem \ref{th_coord2D} it holds $|\nabla r|^2|_{\overline{p}^\pm(t)}=|\nabla s|^2|_{\overline{p}^\pm(t)}=1$ as well as $\nabla r\cdot\nabla s|_{\overline{p}^\pm(t)}=0$. Therefore with the definition \eqref{eq_asym_ACalpha_cp_zalpha} we obtain $|\nabla z^\pm_\alpha|^2|_{\overline{p}^\pm(t)}=1$ and $\nabla z^\pm_\alpha\cdot\nabla r|_{\overline{p}^\pm(t)}=-\cos\alpha$. Hence because of $\theta_0''=f'(\theta_0)$ the lowest order becomes
\begin{align}\label{eq_asym_ACalpha_cp_bulk_m2}
[-\partial_Z^2+2\cos\alpha \partial_Z\partial_\rho-\partial_\rho^2]\hat{w}^{C\pm}_\alpha
+f'(\hat{w}^{C\pm}_\alpha)=0\quad \text{ for }(\rho,Z,t)\in\overline{\R^2_+}\times[0,T].
\end{align}

\paragraph{Bulk Equation: $\Oc(\varepsilon^{-1})$}\label{sec_asym_ACalpha_cp_bulk_m1}
The next order $\Oc(\frac{1}{\varepsilon})$ in \eqref{eq_asym_ACalpha_cp} cancels if we require
\begin{align}\label{eq_asym_ACalpha_cp_bulk_m1}
\left[-\partial_Z^2\right.&\left.+2\cos\alpha \partial_Z\partial_\rho-\partial_\rho^2+f''(\hat{w}^{C\pm}_\alpha)\right]\hat{u}^{C\pm}_{1,\alpha}=G^{C\pm}_{1,\alpha},\\\notag
G^{C\pm}_{1,\alpha}(\rho,Z,t):=&-\left[\partial_\rho\hat{w}^{C\pm}_\alpha-\hat{\chi}_\alpha\theta_0'\right](\partial_t r-\Delta r)|_{\overline{p}^\pm(t)}
-\partial_Z\hat{w}^{C\pm}_\alpha(\partial_tz^\pm_\alpha-\Delta z^\pm_\alpha)|_{\overline{p}^\pm(t)}\\\notag
&-2\partial_\rho\partial_Z\hat{w}^{C\pm}_\alpha (\nabla z^\pm_\alpha\cdot\nabla s)|_{\overline{p}^\pm(t)} \partial_sh_{1,\alpha}|_{(\pm 1,t)}\\\notag
&-\left[\partial_\rho^2\hat{w}^{C\pm}_\alpha-\hat{\chi}_\alpha\theta_0''\right]2(\nabla r\cdot\nabla s)|_{\overline{p}^\pm(t)}\partial_sh_{1,\alpha}|_{(\pm 1,t)}\\\notag
&+\partial_Z^2\hat{w}^{C\pm}_\alpha\left[
\partial_r(|\nabla z^\pm_\alpha|^2\circ\overline{X})|_{(0,\pm 1,t)}(\rho+h_{1,\alpha}|_{(\pm 1,t)})\phantom{\frac{1}{\sin\alpha}}\right.\\\notag
&\quad\qquad\qquad\left.\mp\partial_s(|\nabla z^\pm_\alpha|^2\circ\overline{X})|_{(0,\pm 1,t)}\frac{1}{\sin\alpha}[Z+(\rho+h_{1,\alpha}|_{(\pm 1,t)})\cos\alpha]\right]\\\notag
&+2\partial_\rho\partial_Z\hat{w}^{C\pm}_\alpha\left[
\partial_r((\nabla z^\pm_\alpha\cdot\nabla r)\circ\overline{X})|_{(0,\pm 1,t)}(\rho+h_{1,\alpha}|_{(\pm 1,t)})\phantom{\frac{1}{\sin\alpha}}\right.\\\notag
&\qquad\qquad\left.\mp\partial_s((\nabla z^\pm_\alpha\cdot\nabla r)\circ\overline{X})|_{(0,\pm 1,t)}\frac{1}{\sin\alpha}[Z+(\rho+h_{1,\alpha}|_{(\pm 1,t)})\cos\alpha]\right]\\\notag
&+\left[\partial_\rho^2\hat{w}^{C\pm}_\alpha-\hat{\chi}_\alpha\theta_0''\right]\left[
\partial_r(|\nabla r|^2\circ\overline{X})|_{(0,\pm 1,t)}(\rho+h_{1,\alpha}|_{(\pm 1,t)})\phantom{\frac{1}{\sin\alpha}}\right.\notag
\end{align}\begin{align*}
&\qquad\qquad\qquad\left.\mp\partial_s(|\nabla r|^2\circ\overline{X})|_{(0,\pm 1,t)}\frac{1}{\sin\alpha}[Z+(\rho+h_{1,\alpha}|_{(\pm 1,t)})\cos\alpha]\right].
\end{align*}
Because of Remark \ref{th_coord2D_rem1} and Theorem \ref{th_coord2D} it follows that  $(\partial_t r-\Delta r)|_{\overline{p}^\pm(t)}=(\nabla r\cdot\nabla s)|_{\overline{p}^\pm(t)}=0$, $(\nabla z^\pm_\alpha\cdot\nabla s)|_{\overline{p}^\pm(t)}=\mp\sin\alpha$ and $\partial_r(|\nabla r|^2\circ\overline{X})|_{(0,\pm 1,t)}=\partial_s(|\nabla r|^2\circ\overline{X})|_{(0,\pm 1,t)}=0$. Therefore 
\begin{align*}
G^{C\pm}_{1,\alpha}(\rho,Z,t)=&
-\partial_Z\hat{w}^{C\pm}_\alpha(\partial_tz^\pm_\alpha-\Delta z^\pm_\alpha)|_{\overline{p}^\pm(t)}\pm 2\sin\alpha\, \partial_\rho\partial_Z\hat{w}^{C\pm}_\alpha \partial_sh_{1,\alpha}|_{(\pm 1,t)}\\\notag
&+\partial_Z^2\hat{w}^{C\pm}_\alpha\left[
\partial_r(|\nabla z^\pm_\alpha|^2\circ\overline{X})|_{(0,\pm 1,t)}(\rho+h_{1,\alpha}|_{(\pm 1,t)})\phantom{\frac{1}{\sin\alpha}}\right.\\\notag
&\quad\qquad\qquad\left.\mp\partial_s(|\nabla z^\pm_\alpha|^2\circ\overline{X})|_{(0,\pm 1,t)}\frac{1}{\sin\alpha}[Z+(\rho+h_{1,\alpha}|_{(\pm 1,t)})\cos\alpha]\right]\\\notag
&+2\partial_\rho\partial_Z\hat{w}^{C\pm}_\alpha\left[
\partial_r((\nabla z^\pm_\alpha\cdot\nabla r)\circ\overline{X})|_{(0,\pm 1,t)}(\rho+h_{1,\alpha}|_{(\pm 1,t)})\phantom{\frac{1}{\sin\alpha}}\right.\\\notag
&\qquad\qquad\left.\mp\partial_s((\nabla z^\pm_\alpha\cdot\nabla r)\circ\overline{X})|_{(0,\pm 1,t)}\frac{1}{\sin\alpha}[Z+(\rho+h_{1,\alpha}|_{(\pm 1,t)})\cos\alpha]\right].
\end{align*}
In particular $G^{C\pm}_{1,\alpha}$ is independent of $\hat{\chi}_\alpha$. This is important in order to choose $H_0$ and hence $\hat{\chi}_\alpha$ independently, see Remark \ref{th_asym_ACalpha_cp_cutoff}, 1. For later use, we collect the $h_{1,\alpha}$-terms and write
\begin{align*}
&G^{C\pm}_{1,\alpha}(\rho,Z,t)=
\pm 2\sin\alpha\,\partial_\rho\partial_Z\hat{w}^{C\pm}_\alpha \partial_sh_{1,\alpha}|_{(\pm 1,t)}\\
&+2\partial_\rho\partial_Z\hat{w}^{C\pm}_\alpha\left[\partial_r((\nabla z^\pm_\alpha\cdot\nabla r)\circ\overline{X})|_{(0,\pm 1,t)}\mp\frac{\cos\alpha}{\sin\alpha}\partial_s((\nabla z^\pm_\alpha\cdot\nabla r)\circ\overline{X})|_{(0,\pm 1,t)}\right]h_{1,\alpha}|_{(\pm 1,t)}\\
&+\partial_Z^2\hat{w}^{C\pm}_\alpha\left[
\partial_r(|\nabla z^\pm_\alpha|^2\circ\overline{X})|_{(0,\pm 1,t)}\mp\frac{\cos\alpha}{\sin\alpha}\partial_s(|\nabla z^\pm_\alpha|^2\circ\overline{X})|_{(0,\pm 1,t)}\right]h_{1,\alpha}|_{(\pm 1,t)}+\tilde{G}^{C\pm}_{0,\alpha},
\end{align*}
where $\tilde{G}^{C\pm}_{0,\alpha}\in R^C_{0,(\beta,\gamma),\alpha}$ for all $\beta\in[0,\beta_0)$, $\gamma\in[\frac{\gamma_0}{2},\gamma_0)$ provided that $\hat{w}^{C\pm}_\alpha-\theta_0\in R^C_{0,(\beta,\gamma),\alpha}$ for all these $\beta$, $\gamma$. The latter corresponds to the matching condition \eqref{eq_asym_ACalpha_matching1}.

\paragraph{Bulk Equation: $\Oc(\varepsilon^{k-1})$}\label{sec_asym_ACalpha_cp_bulk_km1}
For $k=1,...,M-1$ we consider $\Oc(\varepsilon^{k-1})$ in  \eqref{eq_asym_ACalpha_cp} and derive an equation for $\hat{u}^{C\pm}_{k+1,\alpha}$. Therefore suppose that the $j$-th order is constructed for all $j=0,...,k$, that it is smooth and that $H_0$ in $\hat{\chi}_\alpha$ is known. Additionally, let $\hat{u}^I_{j+1,\alpha}\in R^I_{j,(\beta_1),\alpha}$ for all $\beta_1\in(0,\min\{\sqrt{f''(\pm1)}\})$ and $j=0,...,k$. Note that in contrast to the case $\alpha=\frac{\pi}{2}$ in \cite{MoserACvACND}, Section 5.1.2.1.2 the decay is used at this point. Namely, with the inequality \eqref{eq_asym_ACalpha_decay_param} for the decay parameters $\beta_0,\gamma_0$ and the decay for the $\hat{u}^I_{i,\alpha}$ it will be possible to control the contributions of $\partial_\rho^l\hat{w}^{C\pm}_\alpha-\hat{\chi}_\alpha\partial_\rho^l\theta_0$, $l=1,2$ and the new types of terms in the expansion of the $f'$-parts, cf.~\eqref{eq_asym_ACalpha_cp_taylor_f1}-\eqref{eq_asym_ACalpha_cp_taylor_f2}. Finally, we assume that $\hat{w}^{C\pm}_\alpha-\theta_0\in R^C_{0,(\beta,\gamma),\alpha}$ as well as $\hat{u}^{C\pm}_{j,\alpha}\in R^C_{j,(\beta,\gamma),\alpha}$ for all $j=1,...,k$ and all $\beta\in[0,\beta_0)$, $\gamma\in[\frac{\gamma_0}{2},\gamma_0)$. The latter corresponds to the matching conditions \eqref{eq_asym_ACalpha_matching1}-\eqref{eq_asym_ACalpha_matching2}. 

Then with the notation from Definition \ref{th_asym_ACalpha_cp_def} it follows for all those $(\beta,\gamma)$:
\begin{alignat*}{2}
\text{For }j=1,...,k+1:&\quad[\Oc(\varepsilon^j)\text{ in }[\eqref{eq_asym_ACalpha_cp_taylor_f1}-\chi_\alpha\cdot\eqref{eq_asym_ACalpha_cp_taylor_f2}]]&\quad &\in\quad  f''(\hat{w}^{C\pm}_\alpha)\hat{u}^{C\pm}_{j,\alpha} + R^C_{j-1,(\beta,\gamma),\alpha},\\
\text{For }i,j=1,...,k:&\quad[\Oc(\varepsilon^j)\text{ in }\eqref{eq_asym_ACalpha_cp_taylor2}\text{ for }g_1=g_1(h_i)]&\quad &\in\quad  P^C_{\max\{i,j\},\alpha}(\rho,Z).
\end{alignat*}
Moreover, for $j=1,...,k+1$ we obtain
\[
[\Oc(\varepsilon^j)\text{ in }\eqref{eq_asym_ACalpha_cp_taylor3}]\quad \in\quad  h_{j,\alpha}|_{(\pm 1,t)} \left[
\partial_r\tilde{g}_2|_{(0,\pm 1,t)}\mp\frac{\cos\alpha}{\sin\alpha}\partial_s\tilde{g}_2|_{(0,\pm 1,t)}\right]+ P^C_{j-1,\alpha}(\rho,Z).
\]
All those identities can be verified with the remarks accompanying \eqref{eq_asym_ACalpha_cp_taylor_f1}-\eqref{eq_asym_ACalpha_cp_taylor3}. The only contributions that are not straight-forward are the (finitely many) terms appearing in the expansion of the $f'$-parts that are of type $\chi_\alpha^l f^{(l+1)}(w^{C\pm}_\alpha)-\chi_\alpha f^{(l+1)}(\theta_0)$, $l\in\{1,...,j+1\}$ times a term in $\hat{R}^I_{j-1,(\beta_1),\alpha}$ for all $\beta_1\in(0,\min\{f''(\pm 1)\})$ times some polynomial in $P^C_{j-1,\alpha}(\rho,Z)$. The latter appear due to the order $\Oc(\varepsilon^i)$ in the expansion of $u^I_{n,\alpha}$, where $i\in\{0,...,j-1\}$ 
and $n\in\{1,...,j\}$. We have to show that such terms are contained in $R^C_{j-1,(\beta,\gamma),\alpha}$ for all $\beta\in[0,\beta_0)$, $\gamma\in[\frac{\gamma_0}{2},\gamma_0)$. On the set $\{\hat{\chi}_\alpha=0\}$ there is nothing to prove. Moreover, on $\{\hat{\chi}_\alpha=1\}$ we can use uniform continuity for $f'$-derivatives on compact sets and that $\hat{w}_\alpha^{C\pm}-\theta_0\in R^C_{0,(\beta,\gamma),\alpha}$ for all $(\beta,\gamma)$ as above to obtain the desired estimate. Finally, for the decay on the set $\Xi:=\{\hat{\chi}_\alpha\in(0,1)\}$ we use $Z\leq|\rho|$ for all $(\rho,Z)\in\Xi$ with $|\rho|+Z\geq \overline{R}$, where $\overline{R}$ is large depending on $\alpha,H_0$. See also Figure \ref{fig_z_alpha_xi}. Then it follows that $\beta|\rho|+\gamma Z \leq (\beta+\gamma)|\rho|\leq \beta_1|\rho|$ for all those $(\rho,Z)$ and all $\beta+\gamma\leq\beta_1$. Due to the inequality $\beta_0+\gamma_0\leq \min\{f''(\pm 1)\}$ we obtain the claimed inclusion.

Now we compute $\Oc(\varepsilon^{k-1})$ for $k=1,...,M-1$ in \eqref{eq_asym_ACalpha_cp}. Let $(\beta,\gamma)$ be as above and arbitrary. The $f'$-part yields a term in $f''(\hat{w}^{C\pm}_\alpha)\hat{u}^{C\pm}_{k+1,\alpha}+R^C_{k,(\beta,\gamma),\alpha}$. Moreover, note that
\begin{align}\label{eq_asym_ACalpha_cp_bulk_VmTh0}
\partial_\rho^l\hat{w}^{C\pm}_\alpha-\hat{\chi}_\alpha\theta_0^{(l)}\in \partial_\rho^l\hat{w}^{C\pm}_\alpha-\theta_0^{(l)} + R^C_{0,(\beta,\gamma),\alpha}\subseteq R^C_{0,(\beta,\gamma),\alpha}
\end{align}
for all $l\in\N$ and all $\beta, \gamma$ as above. The first inclusion can be shown with the decay of $\theta_0^{(l)}$ from Theorem \ref{th_theta_0} and analogous arguments as before since for $(\rho,Z)\in \{\hat{\chi}_\alpha\in[0,1)\}$ it holds $Z\leq |\rho|$ if $|\rho|+Z$ is large. The second inclusion follows from $\hat{w}^{C\pm}_\alpha-\theta_0\in R^C_{0,(\beta,\gamma),\alpha}$. Therefore the contribution at order $\Oc(\varepsilon^{k-1})$ from the first line in \eqref{eq_asym_ACalpha_cp} is contained in
\[
(\partial_\rho\hat{w}^{C\pm}_\alpha-\hat{\chi}_\alpha\theta_0')\left[P^C_{k,\alpha}(\rho,Z)+\sum_{j=0}^{k-1} P^C_{j,\alpha}(\rho,Z) P^C_{k-1,\alpha}(\rho,Z)\right] \subseteq  R^C_{k,(\beta,\gamma),\alpha}.
\]
Analogously it follows that the $\partial_Z\hat{w}^{C\pm}_\alpha$-part yields an element of $R^C_{k,(\beta,\gamma),\alpha}$. Moreover, the term $\partial_t\hat{w}^{C\pm}_\alpha\in R^C_{0,(\beta,\gamma),\alpha}$ only contributes to $\Oc(1)$. Furthermore, the $\partial_Z^2\hat{w}^{C\pm}_\alpha$-part gives an element of 
\begin{align*}
&-\partial_Z^2\hat{w}^{C\pm}_\alpha\left[
h_{k+1,\alpha}|_{(\pm 1,t)}\left(\!\partial_r(|\nabla z^\pm_\alpha|^2\circ\overline{X})|_{(0,\pm 1,t)}\mp\frac{\cos\alpha}{\sin\alpha}\partial_s(|\nabla z^\pm_\alpha|^2\circ\overline{X})|_{(0,\pm 1,t)}\!\right)\!\right]\\
&+\partial_Z^2\hat{w}^{C\pm}_\alpha P^C_{k,\alpha}(\rho,Z),
\end{align*}
where the last term is contained in $R^C_{k,(\beta,\gamma),\alpha}$. Similarly, the $\partial_\rho\partial_Z\hat{w}^{C\pm}_\alpha$-part yields a term in
\begin{align*}
R^C_{k,(\beta,\gamma),\alpha}& -2\partial_\rho\partial_Z\hat{w}^{C\pm}_\alpha \left[\pm \sin\alpha\,\partial_sh_{k+1,\alpha}|_{(\pm 1,t)}\phantom{\frac{\cos}{\sin}}\right.\\
+&h_{k+1,\alpha}|_{(\pm 1,t)}\left.\left(\partial_r((\nabla z^\pm_\alpha \cdot\nabla r)\circ\overline{X})|_{(0,\pm 1,t)}\mp\frac{\cos\alpha}{\sin\alpha}\partial_s((\nabla z^\pm_\alpha\cdot\nabla r)\circ\overline{X})|_{(0,\pm 1,t)}\right)\right].
\end{align*}
Due to \eqref{eq_asym_ACalpha_cp_bulk_VmTh0} and since $\nabla r\cdot \nabla s|_{\overline{p}^\pm(t)}=\partial_r(|\nabla r|^2\circ\overline{X})|_{(0,\pm 1,t)}=\partial_s(|\nabla r|^2\circ\overline{X})|_{(0,\pm 1,t)}=0$, we get in an analogous way that the contribution at order $\Oc(\varepsilon^{k-1})$ from the $(\partial_\rho^2\hat{w}^{C\pm}_\alpha-\hat{\chi}_\alpha\theta_0'')$-term in \eqref{eq_asym_ACalpha_cp} is an element of $R^C_{k,(\beta,\gamma),\alpha}$. Moreover, the term $(\partial_t-\Delta)\tilde{u}^{C\pm}_{\varepsilon,\alpha}$ yields a contribution in
\[
\left[-\partial_Z^2+2\cos\alpha\,\partial_\rho\partial_Z-\partial_\rho^2\right]\hat{u}^{C\pm}_{k+1,\alpha} + R^C_{k,(\beta,\gamma),\alpha}.
\]
For $\tilde{u}^I_{\varepsilon,\alpha}(\partial_t-\Delta)\chi_\alpha$ we use $u^I_{1,\alpha}=0$, the decay of $\hat{u}^I_{j,\alpha}$ and that $Z\leq|\rho|$ on $\{\hat{\chi}_\alpha\in(0,1)\}$ if $|\rho|+Z$ is large. With the latter we obtain as above the decay of the products of the inner expansion terms with derivatives of $\hat{\chi}_\alpha$. Therefore we get a term in $R^C_{k,(\beta,\gamma),\alpha}$, where we note that $\hat{u}^I_{j,\alpha}$ counts to order $j-1$. Finally, with the same ideas we obtain that $2\nabla(\tilde{u}^I_{\varepsilon,\alpha})\cdot\nabla(\chi_\alpha)$ also yields a contribution in $R^C_{k,(\beta,\gamma),\alpha}$.

Altogether the $\Oc(\varepsilon^{k-1})$-order in the expansion for the bulk equation \eqref{eq_asym_ACalpha_cp} is zero if 
\begin{align}\label{eq_asym_ACalpha_cp_bulk_km1}
&\left[-\partial_Z^2+2\cos\alpha \partial_Z\partial_\rho-\partial_\rho^2+f''(\hat{w}^{C\pm}_\alpha)\right]\hat{u}^{C\pm}_{k+1,\alpha}=G^{C\pm}_{k+1,\alpha},\\\notag
G^{C\pm}_{k+1,\alpha}&:=2\partial_\rho\partial_Z\hat{w}^{C\pm}_\alpha \left[\pm \sin\alpha\,\partial_sh_{k+1,\alpha}|_{(\pm 1,t)}\phantom{\frac{\cos}{\sin}}\right.\\\notag
+&h_{k+1,\alpha}|_{(\pm 1,t)}\left.\left(\partial_r((\nabla z^\pm_\alpha \cdot\nabla r)\circ\overline{X})|_{(0,\pm 1,t)}\mp\frac{\cos\alpha}{\sin\alpha}\partial_s((\nabla z^\pm_\alpha\cdot\nabla r)\circ\overline{X})|_{(0,\pm 1,t)}\right)\right]\\\notag
+\partial_Z^2\hat{w}^{C\pm}_\alpha\,& h_{k+1,\alpha}|_{(\pm 1,t)}\left[
\partial_r(|\nabla z^\pm_\alpha|^2\circ\overline{X})|_{(0,\pm 1,t)}\mp\frac{\cos\alpha}{\sin\alpha}\partial_s(|\nabla z^\pm_\alpha|^2\circ\overline{X})|_{(0,\pm 1,t)}\right]+\tilde{G}^{C\pm}_{k,\alpha},
\end{align}
where $\tilde{G}^{C\pm}_{k,\alpha}\in R^C_{k,(\beta,\gamma),\alpha}$ for all $\beta\in[0,\beta_0)$, $\gamma\in[\frac{\gamma_0}{2},\gamma_0)$.

From the expansion of the Robin boundary condition \eqref{eq_ACalpha2} we will obtain boundary conditions for the equations \eqref{eq_asym_ACalpha_cp_bulk_m2}, \eqref{eq_asym_ACalpha_cp_bulk_m1} and \eqref{eq_asym_ACalpha_cp_bulk_km1}. This is done in the next section.

\subsubsection{Contact Point Expansion: The Robin Boundary Condition}\label{sec_asym_ACalpha_cp_robin}
We expand the non-linear Robin-boundary condition \eqref{eq_ACalpha2} for $u_{\varepsilon,\alpha}=\chi_\alpha u^I_{\varepsilon,\alpha}+u^{C\pm}_{\varepsilon,\alpha}$ on $\partial Q_T\cap\overline{\Gamma(2\delta)}$. Since $\chi_\alpha$ is zero in an $\varepsilon$-dependent neighbourhood of $\partial Q_T$, the latter equals
\begin{align}\label{eq_asym_ACalpha_cp_robin1}
N_{\partial\Omega}\cdot\nabla(w^{C\pm}_\alpha+\tilde{u}^{C\pm}_{\varepsilon,\alpha})|_{\partial Q_T}+\frac{1}{\varepsilon}\sigma_\alpha'(w^{C\pm}_\alpha+\tilde{u}^{C\pm}_{\varepsilon,\alpha})|_{\partial Q_T}=0\quad\text{ on }\partial Q_T\cap\overline{\Gamma(2\delta)}.
\end{align} 
Here on $\partial Q_T\cap\overline{\Gamma(2\delta)}$ it holds $z^\pm_\alpha=Z_{\varepsilon,\alpha}^\pm=0$ and $s^\pm=\frac{\cos\alpha}{\sin\alpha}r$, cf.~\eqref{eq_asym_ACalpha_cp_zalpha}. Therefore we set
\begin{align}\label{eq_asym_ACalpha_cp_sbar}
\overline{s}^\pm(r):=\pm 1\mp\frac{\cos\alpha}{\sin\alpha}r,\quad
\overline{X}^\pm_1|_{(r,t)}:=\overline{X}|_{(r,\overline{s}^\pm(r),t)}
\quad\text{ for }(r,t)\in[-2\delta,2\delta]\times[0,T].
\end{align}
Then due to Lemma \ref{th_asym_ACalpha_cp_trafo} the equation \eqref{eq_asym_ACalpha_cp_robin1} is equivalent to
\begin{align}\notag
0=N_{\partial\Omega}|_{\overline{X}^\pm_1(r,t)}&\cdot\left[
(\partial_\rho\hat{w}^{C\pm}_\alpha+\partial_\rho\hat{\tilde{u}}^{C\pm}_{\varepsilon,\alpha})|_{Z=0}
\left(\frac{\nabla r|_{\overline{X}^\pm_1(r,t)}}{\varepsilon}-\nabla s|_{\overline{X}^\pm_1(r,t)}\partial_sh_{\varepsilon,\alpha}|_{(\overline{s}^\pm(r),t)}\right)\right.\\
+&\left.(\partial_Z\hat{w}^{C\pm}_\alpha+\partial_Z\hat{\tilde{u}}^{C\pm}_{\varepsilon,\alpha})|_{Z=0}\frac{\nabla z^\pm_\alpha|_{\overline{X}^\pm_1(r,t)}}{\varepsilon}
\right]+\frac{1}{\varepsilon}\sigma'_\alpha(\hat{w}^{C\pm}_\alpha+\hat{\tilde{u}}^{C\pm}_{\varepsilon,\alpha})|_{Z=0}\label{eq_asym_ACalpha_cp_robin2}
\end{align}
on $\partial Q_T\cap\overline{\Gamma(2\delta)}$, where $r=r(x,t)$ and $\rho=\rho_{\varepsilon,\alpha}(x,t)$. 

In the following we determine how all the terms are expanded into $\varepsilon$-series up to $\Oc(\varepsilon^{M-1})$ with coefficients in $(\rho,t)$. For the $h_{\varepsilon,\alpha}$-terms let $g_1:I_\mu\times[0,T]\rightarrow\R$ be smooth. We use the rigorous Taylor expansion \eqref{eq_asym_ACalpha_cp_taylor2} and we replace $s\mp 1$ by \eqref{eq_asym_ACalpha_cp_s2} with $Z=0$. Then the remarks and the assertions for the remainder terms below \eqref{eq_asym_ACalpha_cp_taylor2} are still valid when we formally set $Z=0$. Concerning terms evaluated at $\overline{X}^\pm_1$, we consider $g_2:\partial Q_T\cap\overline{\Gamma(2\delta)}\rightarrow\R$ smooth. A Taylor expansion yields for $(r,t)\in[-2\delta,2\delta]\times[0,T]$:
\begin{align}\label{eq_asym_ACalpha_cp_taylor4}
\tilde{g}^\pm_2(r,t):=g_2(\overline{X}^\pm_1(r,t))=\sum_{k=0}^{M+2}\frac{\partial_r^k\tilde{g}^\pm_2|_{(0,t)}}{k!}r^k+\Oc(|r|^{M+3}).
\end{align}
Then we use $r=\varepsilon(\rho_{\varepsilon,\alpha}+h_{\varepsilon,\alpha}|_{(s,t)})$ and expand $h_{\varepsilon,\alpha}$ as specified above. To this end the height functions need to be smooth enough. Similar to the expansion of the $(x,t)$-dependent terms in the bulk equation, cf.~\eqref{eq_asym_ACalpha_cp_taylor3}, the terms in the $\varepsilon$-expansion of \eqref{eq_asym_ACalpha_cp_taylor4} are for $k=2,...,M+2$:
\begin{align*}
\Oc(1):\quad &g_2|_{\overline{p}^\pm(t)},\\
\Oc(\varepsilon):\quad &(\rho+h_{1,\alpha}|_{(\pm 1,t)})\partial_r\tilde{g}^\pm_2|_{(0,t)},\\
\Oc(\varepsilon^k):\quad &h_{k,\alpha}|_{(\pm 1,t)}\partial_r\tilde{g}^\pm_2|_{(0,t)} +[\text{some polynomial in }(\rho,\partial_s^l h_{j,\alpha}|_{(\pm 1,t)}), l=0,...,k-1,\\ 
&\phantom{h_{k,\alpha}|_{(\pm 1,t)}\partial_r\tilde{g}^\pm_2|_{(0,t)} +[} j=1,...,k-1\text{ of order }\leq k,\text{ where}\\
&\phantom{h_{k,\alpha}|_{(\pm 1,t)}\partial_r\tilde{g}^\pm_2|_{(0,t)} +[}\text{the coefficients are multiples of }(\partial_r^2\tilde{g}^\pm_2,...,\partial_r^k\tilde{g}^\pm_2)|_{(0,t)}].
\end{align*}
The other explicit terms in \eqref{eq_asym_ACalpha_cp_taylor4} are bounded by $\varepsilon^{M+3}$ times some polynomial in $|\rho|$ if the $h_{j,\alpha}$ are smooth. Later, these terms and the $\Oc(|r|^{M+3})$-remainder in \eqref{eq_asym_ACalpha_cp_taylor4} for each choice of $g_2$ will be multiplied with exponentially decaying terms in $|\rho|$. Then these terms are $\Oc(\varepsilon^{M+3})$. Finally, the $\sigma_\alpha$-term is replaced via
\[
\sigma_\alpha'(\hat{w}^{C\pm}_\alpha+\hat{\tilde{u}}^{C\pm}_{\varepsilon,\alpha})|_{Z=0}
=\sigma_\alpha'(\hat{w}^{C\pm}_\alpha)|_{Z=0}
+\sum_{k=1}^{M+2}\frac{1}{k!}\sigma_\alpha^{(k+1)}(\hat{w}^{C\pm}_\alpha)|_{Z=0} (\hat{\tilde{u}}^{C\pm}_{\varepsilon,\alpha}|_{Z=0})^k
+\Oc(\varepsilon^{M+3}).
\]
The terms in the $\varepsilon$-expansion are for $k=2,...,M+2$:
\begin{align*}
\Oc(1):\quad & \sigma_\alpha'(\hat{w}^{C\pm}_\alpha|_{Z=0}),\\
\Oc(\varepsilon):\quad & \sigma_\alpha''(\hat{w}^{C\pm}_\alpha)\hat{u}^{C\pm}_{1,\alpha}|_{Z=0},\\
\Oc(\varepsilon^k):\quad & \sigma_\alpha''(\hat{w}^{C\pm}_\alpha)\hat{u}^{C\pm}_{k,\alpha}|_{Z=0} 
+ [\text{a polynomial in }(\hat{u}^{C\pm}_{1,\alpha}|_{Z=0},...,\hat{u}^{C\pm}_{k-1,\alpha}|_{Z=0})\text{ of order }\leq k,\\
& \phantom{\sigma_\alpha''(\hat{w}^{C\pm}_\alpha)\hat{u}^{C\pm}_{k,\alpha}|_{Z=0}+ [}
\text{where the coefficients are multiples of } \sigma_\alpha^{(3)}(\hat{w}^{C\pm}_\alpha)|_{Z=0},...,\\
& \phantom{\sigma_\alpha''(\hat{w}^{C\pm}_\alpha)\hat{u}^{C\pm}_{k,\alpha}|_{Z=0}+ [}
\sigma_\alpha^{(k+1)}(\hat{w}^{C\pm}_\alpha)|_{Z=0}\text{ and every term contains a }\hat{u}^{C\pm}_{j,\alpha}\text{-factor}].
\end{align*} 
The other explicit terms are of order  $\Oc(\varepsilon^{M+3})$.

Now we can expand \eqref{eq_asym_ACalpha_cp_robin2} into $\varepsilon$-series with coefficients in $(\rho_{\varepsilon,\alpha},t)$ up to $\Oc(\varepsilon^{M-1})$. 

\paragraph{Robin Boundary Condition: $\Oc(\varepsilon^{-1})$}\label{sec_asym_ACalpha_cp_robin_m1}
At the lowest order $\Oc(\frac{1}{\varepsilon})$ we have 
\[
(N_{\partial\Omega}\cdot\nabla r)|_{\overline{p}^\pm(t)} \partial_\rho\hat{w}^{C\pm}_\alpha|_{Z=0}
+(N_{\partial\Omega}\cdot\nabla z^\pm_\alpha)|_{\overline{p}^\pm(t)} \partial_Z\hat{w}^{C\pm}_\alpha|_{Z=0} + \sigma_\alpha'(\hat{w}^{C\pm}_\alpha|_{Z=0})=0.
\]
By construction it holds $N_{\partial\Omega}|_{\overline{p}^\pm(t)}=-\nabla z^\pm_\alpha|_{\overline{p}^\pm(t)}$, where $\nabla z^\pm_\alpha\cdot\nabla r|_{\overline{p}^\pm(t)}=-\cos\alpha$ and $|\nabla z^\pm_\alpha|^2|_{\overline{p}^\pm(t)}=1$, cf.~Section \ref{sec_asym_ACalpha_cp_bulk_m2}. Therefore the order $\Oc(\frac{1}{\varepsilon})$ vanishes if we require
\begin{align}
[-\partial_Z+\cos\alpha\,\partial_\rho]\hat{w}^{C\pm}_\alpha|_{Z=0}+\sigma_\alpha'(\hat{w}^{C\pm}_\alpha|_{Z=0})=0.
\end{align}
The latter equation together with \eqref{eq_asym_ACalpha_cp_bulk_m2} is solved by $\hat{w}^{C\pm}_\alpha:=v_\alpha$ for $\alpha\in\frac{\pi}{2}+[-\alpha_0,\alpha_0]$, where $v_\alpha$ and $\alpha_0$ are as in Remark \ref{th_asym_ACalpha_decay_param}. In particular $v_\alpha$ is $t$-independent and $v_\alpha-\theta_0\in R^C_{0,(\beta,\gamma),\alpha}$ for all $\beta\in[0,\beta_0]$, $\gamma\in[\frac{\gamma_0}{2},\gamma_0]$.

\paragraph{Robin Boundary Condition: $\Oc(\varepsilon^0)$}\label{sec_asym_ACalpha_cp_robin_0}
The next order $\Oc(1)$ vanishes if
\begin{align}\label{eq_asym_ACalpha_cp_bc1}
&[-\partial_Z+\cos\alpha\,\partial_\rho+\sigma_\alpha''(v_\alpha|_{Z=0})]\hat{u}^{C\pm}_1|_{Z=0}(\rho,t)=g^{C\pm}_{1,\alpha}(\rho,t),\\ \notag
g^{C\pm}_{1,\alpha}(\rho,t):=&
-\partial_\rho v_\alpha|_{Z=0}\left[h_{1,\alpha}|_{(\pm 1,t)} \partial_r((N_{\partial\Omega}\cdot\nabla r)\circ\overline{X}^\pm_1)|_{(0,t)}\mp\sin\alpha\,\partial_sh_{1,\alpha}|_{(\pm 1,t)}\right]\\\notag
&-\partial_Z v_\alpha|_{Z=0}h_{1,\alpha}|_{(\pm 1,t)} \partial_r((N_{\partial\Omega}\cdot\nabla z^\pm_\alpha)\circ\overline{X}^\pm_1)|_{(0,t)}+\tilde{g}^{C\pm}_{0,\alpha}(\rho,t),
\end{align}
where $\tilde{g}^{C\pm}_{0,\alpha}\in R^C_{0,(\beta),\alpha}$ is given by
\[
\tilde{g}^{C\pm}_{0,\alpha}(\rho,t):=
-\rho\,\partial_\rho v_\alpha|_{Z=0} \partial_r((N_{\partial\Omega}\cdot\nabla r)\circ\overline{X}^\pm_1)|_{(0,t)}
-\rho\,\partial_Z v_\alpha|_{Z=0} \partial_r((N_{\partial\Omega}\cdot\nabla z^\pm_\alpha)\circ\overline{X}^\pm_1)|_{(0,t)}.
\] 
We solve this equation together with \eqref{eq_asym_ACalpha_cp_bulk_m1}. If $h_{1,\alpha}$ is smooth and determined only from the $0$-th order, then $G^{C\pm}_{1,\alpha}\in R^C_{0,(\beta,\gamma),\alpha}$ and $g^{C\pm}_{1,\alpha}\in R^C_{0,(\beta),\alpha}$ for all $\beta\in[0,\beta_0)$, $\gamma\in[\frac{\gamma_0}{2},\gamma_0)$. Note that both are independent of $\chi_\alpha$. Therefore due to Remark \ref{th_asym_ACalpha_decay_param} and Theorem \ref{th_hp_alpha_sol_reg3} there is a unique smooth solution $\hat{u}^{C\pm}_1$ to \eqref{eq_asym_ACalpha_cp_bulk_m1} and \eqref{eq_asym_ACalpha_cp_bc1} with the same decay as $G^{C\pm}_{1,\alpha}$ if and only if the compatibility condition \eqref{eq_hp_alpha_lin_comp} holds, i.e.
\[
\int_{\R^2_+} G^{C\pm}_{1,\alpha} \partial_\rho v_\alpha\,d(\rho,Z)
+ \int_\R g^{C\pm}_{1,\alpha} \partial_\rho v_\alpha|_{Z=0}\,d\rho=0.
\] 
The latter is equivalent to the following linear boundary condition for $h_{1,\alpha}$:
\begin{align}\label{eq_asym_ACalpha_cp_h1}
b_{1,\alpha}^\pm(t)\partial_sh_{1,\alpha}|_{(\pm 1,t)}
+b_{0,\alpha}^\pm(t)h_{1,\alpha}|_{(\pm 1,t)}
=f_{0,\alpha}^\pm(t)\quad\text{ for }t\in[0,T],
\end{align}
where 
\begin{align*}
b_{1,\alpha}^\pm(t) :=&
\pm \sin\alpha\left[
2\int_{\R^2_+}\partial_\rho\partial_Zv_\alpha \partial_\rho v_\alpha\,d(\rho,Z) 
+ \int_\R (\partial_\rho v_\alpha)^2|_{Z=0}\,dZ
\right],\\
b_{0,\alpha}^\pm(t) :=& \int_{\R^2_+}\partial_Z^2v_\alpha\partial_\rho v_\alpha\,d(\rho,Z)\left[
\partial_r(|\nabla z^\pm_\alpha|^2\circ\overline{X})|_{(0,\pm 1,t)}\mp\frac{\cos\alpha}{\sin\alpha}\partial_s(|\nabla z^\pm_\alpha|^2\circ\overline{X})|_{(0,\pm 1,t)}\right]\\
+2\int_{\R^2_+}&\partial_\rho\partial_Zv_\alpha \partial_\rho v_\alpha
\left[\partial_r((\nabla z^\pm_\alpha\cdot\nabla r)\circ\overline{X})|_{(0,\pm 1,t)}\mp\frac{\cos\alpha}{\sin\alpha}\partial_s((\nabla z^\pm_\alpha\cdot\nabla r)\circ\overline{X})|_{(0,\pm 1,t)}\right]\\ \textstyle
-\!\int_\R(\partial_\rho v&_\alpha)^2|_{Z=0}\textstyle \partial_r((N_{\partial\Omega}\cdot\nabla r)\circ\overline{X}^\pm_1)|_{(0,t)}
-\!\int_\R\partial_Z v_\alpha\partial_\rho v_\alpha|_{Z=0} \partial_r((N_{\partial\Omega}\cdot\nabla z^\pm_\alpha)\circ\overline{X}^\pm_1)|_{(0,t)},\\
f_{0,\alpha}^\pm(t) := &
-\int_{\R^2_+}\tilde{G}^{C\pm}_{0,\alpha}\partial_\rho v_\alpha\,d(\rho,Z)-\int_\R\tilde{g}^{C\pm}_{0,\alpha}\partial_\rho v_\alpha|_{Z=0}\,d\rho
\end{align*}
are smooth in $t\in[0,T]$ and independent of $\chi_\alpha$, where $\tilde{G}^{C\pm}_{0,\alpha}$ is as in Section \ref{sec_asym_ACalpha_cp_bulk_m1}. Together with the linear parabolic equation \eqref{eq_asym_ACalpha_in_hk} for $k=0$ from the inner expansion in Section \ref{sec_asym_ACalpha_in}, we obtain a time-dependent linear parabolic boundary value problem for $h_{1,\alpha}$. Here analogously to the $\frac{\pi}{2}$-case in \cite{AbelsMoser} the initial value $h_{1,\alpha}|_{t=0}$ is not specified. 

\begin{Remark}\label{th_asym_ACalpha_feven_rem2}\upshape
	If $f$ is even, then due to Remark \ref{th_asym_ACalpha_feven_rem1} it holds $f_{0,\alpha}=0$ for the right hand side in \eqref{eq_asym_ACalpha_in_hk} for $k=0$. However, in general it holds $f^{\pm}_{0,\alpha}\neq 0$. Therefore in contrast to the $\frac{\pi}{2}$-case in \cite{AbelsMoser}, Section 3 a non-trivial $h_{1,\alpha}$ is needed except from special cases. This is due to the many terms appearing in $f^\pm_{0,\alpha}$. Also note that for $v_\alpha$ there are no symmetry properties available.
\end{Remark}

We solve the equations for $h_{1,\alpha}$ with Lunardi, Sinestrari, von Wahl \cite{LunardiSvW}, Chapter 9 in the analogous way as in \cite{AbelsMoser}, Section 3.2.2. To this end we only need that all the coefficients are smooth and that $|\nabla s|^2|_{\overline{X}_0}$, $|b^\pm_{1,\alpha}|$ are bounded from below by a positive constant. Everything is known except the estimate for $b^\pm_{1,\alpha}$. However, the latter follows from the additional estimates in Remark \ref{th_hp_alpha_solN3_rem}. Hence we obtain a smooth solution $h_{1,\alpha}:I\times[0,T]\rightarrow\R$ to \eqref{eq_asym_ACalpha_in_hk} with $k=0$ together with \eqref{eq_asym_ACalpha_cp_h1}. We can extend $h_{1,\alpha}$ to a smooth function on $I_\mu\times[0,T]$ for example with the Stein Extension Theorem, see Leoni \cite{Leoni}, Theorem 13.17. Moreover, $h_{1,\alpha}$ is independent of $\chi_\alpha$. This enables us to define $H_0$ and $\chi_\alpha$ according to Remark \ref{th_asym_ACalpha_cp_cutoff}, 1. Furthermore, due to Section \ref{sec_asym_ACalpha_in} we obtain $\hat{u}^I_{2,\alpha}$ (solving \eqref{eq_asym_ACalpha_in_uk} for $k=0$) with $\hat{u}^I_{2,\alpha}\in R^I_{1,(\beta_1),\alpha}$ for all $\beta_1\in(0,\min\{\sqrt{f''(\pm1)}\})$. Therefore the first inner order is computed. Finally, Theorem \ref{th_hp_alpha_sol_reg3} yields a unique smooth solution $\hat{u}^{C\pm}_{1,\alpha}$ to  \eqref{eq_asym_ACalpha_cp_bulk_m1} and \eqref{eq_asym_ACalpha_cp_bc1} such that $\hat{u}^{C\pm}_1\in R^C_{1,(\beta,\gamma),\alpha}$ for all $\beta\in[0,\beta_0)$, $\gamma\in[\frac{\gamma_0}{2},\gamma_0)$. Hence the first order is determined.

\paragraph{Robin Boundary Condition: $\Oc(\varepsilon^k)$ and Induction}\label{sec_asym_ACalpha_cp_robin_k}
For $k=1,...,M-1$ we consider $\Oc(\varepsilon^k)$ in \eqref{eq_asym_ACalpha_cp_robin2} and derive equations for the ($k+1$)-th order. Therefore we assume the following induction hypothesis: suppose that the $j$-th order is constructed for all $j=0,...,k$, that it is smooth and that $H_0$ and $\chi_\alpha$ is known. Moreover, suppose that $\hat{u}^I_{j+1,\alpha}\in R^I_{j,(\beta_1),\alpha}$ for every $\beta_1\in(0,\min\{\sqrt{f''(\pm1)}\})$ and $j=0,...,k$. Finally, let $\hat{u}^{C\pm}_{j,\alpha}\in R^C_{j,(\beta,\gamma),\alpha}$ for all $\beta\in[0,\beta_0)$, $\gamma\in[\frac{\gamma_0}{2},\gamma_0)$ and $j=0,...,k$. The assumption holds for $k=1$ due to Section \ref{sec_asym_ACalpha_cp_robin_0}. 

Then with the notation in Definition \ref{th_asym_ACalpha_cp_def} we have 
\begin{alignat*}{2}
\text{For }j=1,...,k+1:&\quad[\Oc(\varepsilon^j)\text{ in }\sigma_\alpha'(\hat{w}^{C\pm}_\alpha+\hat{\tilde{u}}^{C\pm}_{\varepsilon,\alpha})|_{Z=0}]\quad& \in\quad & \sigma_\alpha''(v_\alpha)\hat{u}^C_{k+1,\alpha}|_{Z=0}+R^C_{k,(\beta),\alpha},\\
\text{For }i,j=1,...,k:&\quad[\Oc(\varepsilon^j)\text{ for }g_1=g_1(h_i)|_{(\overline{s}^\pm(r),t)}]\quad& \in\quad & P^C_{\max\{i,j\},\alpha}(\rho).
\end{alignat*}
Moreover, for $j=1,...,k+1$ we obtain
\begin{align*}
[\Oc(\varepsilon^j)\text{ in }\eqref{eq_asym_ACalpha_cp_taylor4}]\quad\in 
 h_{j,\alpha}|_{(\pm 1,t)}\partial_r\tilde{g}^\pm_2|_{(0,t)} + P^C_{j-1,\alpha}(\rho)\quad 
[\subseteq P^C_{j,\alpha}(\rho)\text{, if }j\leq k].
\end{align*}

With this we can compute the order $\Oc(\varepsilon^k)$ in \eqref{eq_asym_ACalpha_cp_robin2}. Therefore let $\beta\in[0,\beta_0)$ be arbitrary. The contribution of $\frac{1}{\varepsilon}(\partial_\rho v_\alpha
+\partial_\rho\hat{\tilde{u}}^{C\pm}_{\varepsilon,\alpha})|_{Z=0}
(N_{\partial\Omega}\cdot\nabla r)|_{\overline{X}^\pm_1(r,t)}$
yields a term in
\begin{align*}
&\partial_\rho\hat{u}^{C\pm}_{k+1,\alpha}|_{Z=0}
N_{\partial\Omega}\cdot\nabla r|_{\overline{p}^\pm(t)} 
+\partial_\rho v_\alpha|_{Z=0} h_{k+1,\alpha}|_{(\pm 1,t)}
\partial_r((N_{\partial\Omega}\cdot\nabla r)\circ\overline{X}^\pm_1)|_{(0,t)}\\
&+\sum_{j=1}^k P^C_{j,\alpha}(\rho) \partial_\rho\hat{u}^{C\pm}_{k+1-j}|_{Z=0},
\end{align*}
where $N_{\partial\Omega}\cdot\nabla r|_{\overline{p}^\pm(t)} = \cos\alpha$ and the last sum is contained in $R^C_{k,(\beta),\alpha}$. Moreover, from the term
$-(\partial_\rho v_\alpha
+\partial_\rho\hat{\tilde{u}}^{C\pm}_{\varepsilon,\alpha})|_{Z=0}
(N_{\partial\Omega}\cdot\nabla s)|_{\overline{X}^\pm_1(r,t)}\partial_sh_{\varepsilon,\alpha}|_{(\overline{s}^\pm(r),t)}$ we obtain an element of
\begin{align*}
(\partial_sh_{k+1,\alpha}|_{(\pm 1,t)}+ P^C_{k,\alpha}(\rho))  
N_{\partial\Omega}\cdot\nabla s|_{\overline{p}^\pm(t)} \partial_\rho v_\alpha|_{Z=0}
+ P^C_{k,\alpha}(\rho) \sum_{j=0}^k P^C_{j,\alpha}(\rho) \partial_\rho\hat{u}^{C\pm}_{k-j,\alpha}\\
\subseteq \pm\sin\alpha\,\partial_\rho v_\alpha|_{Z=0} \partial_sh_{k+1,\alpha}|_{(\pm 1,t)} + R^C_{k,(\beta),\alpha}.
\end{align*}
Analogously as before, $\frac{1}{\varepsilon}(\partial_Z v_\alpha+\partial_Z\hat{\tilde{u}}^{C\pm}_{\varepsilon,\alpha})|_{Z=0}
(N_{\partial\Omega}\cdot\nabla z^\pm_\alpha)|_{\overline{X}^\pm_1(r,t)}$ contributes a term in
\[
-\partial_Z\hat{u}^{C\pm}_{k+1,\alpha}|_{Z=0} 
+\partial_Zv_\alpha|_{Z=0} h_{k+1,\alpha}|_{(\pm 1,t)} 
\partial_r((N_{\partial\Omega}\cdot\nabla z^\pm_\alpha)\circ\overline{X}^\pm_1)|_{(0,t)}
+ R^C_{k,(\beta),\alpha}.
\]
Finally, the term
$\frac{1}{\varepsilon}\sigma'_\alpha(v_\alpha+\hat{\tilde{u}}^{C\pm}_{\varepsilon,\alpha})|_{Z=0}$ gives an element in $\sigma_\alpha''(v_\alpha)\hat{u}^{C\pm}_{k+1,\alpha}|_{Z=0}+R^C_{k,(\beta),\alpha}$.

Altogether the $\Oc(\varepsilon^k)$-order in the expansion of \eqref{eq_asym_ACalpha_cp_robin2} vanishes if
\begin{align}\label{eq_asym_ACalpha_cp_bck}
&[-\partial_Z+\cos\alpha\,\partial_\rho+\sigma_\alpha''(v_\alpha|_{Z=0})]\hat{u}_{k+1}^{C\pm}|_{Z=0}(\rho,t)=g^{C\pm}_{k+1,\alpha}(\rho,t),\\ \notag
g^{C\pm}_{k+1,\alpha}(\rho,t):=&
-\partial_\rho v_\alpha|_{Z=0}\left[h_{k+1,\alpha}|_{(\pm 1,t)} \partial_r((N_{\partial\Omega}\cdot\nabla r)\circ\overline{X}^\pm_1)|_{(0,t)}
\mp\sin\alpha\,\partial_sh_{k+1,\alpha}|_{(\pm 1,t)}\right]\\\notag
&-\partial_Z v_\alpha|_{Z=0}h_{k+1,\alpha}|_{(\pm 1,t)} \partial_r((N_{\partial\Omega}\cdot\nabla z^\pm_\alpha)\circ\overline{X}^\pm_1)|_{(0,t)}+\tilde{g}^{C\pm}_{k,\alpha}(\rho,t),
\end{align}
where $\tilde{g}^{C\pm}_{k,\alpha}\in R^C_{k,(\beta),\alpha}$. We solve the latter equation \eqref{eq_asym_ACalpha_cp_bck} together with \eqref{eq_asym_ACalpha_cp_bulk_km1}.

The compatibility condition \eqref{eq_hp_alpha_lin_comp} yields the following linear boundary condition for $h_{k+1,\alpha}$:
\begin{align}\label{eq_asym_ACalpha_cp_hk}
b_{1,\alpha}^\pm(t)\partial_sh_{k+1,\alpha}|_{(\pm 1,t)}+ b_{0,\alpha}^\pm(t)h_{k+1,\alpha}|_{(\pm 1,t)}=f^\pm_{k,\alpha}(t)\quad\text{ for }t\in[0,T],
\end{align}
where $b^\pm_{1,\alpha},b^\pm_{0,\alpha}$ are the same as the ones after \eqref{eq_asym_ACalpha_cp_h1} and 
\[
f_{k,\alpha}^\pm(t) := 
-\int_{\R^2_+}\tilde{G}^{C\pm}_{k,\alpha}\partial_\rho v_\alpha\,d(\rho,Z)-\int_\R\tilde{g}^{C\pm}_{k,\alpha}\partial_\rho v_\alpha|_{Z=0}\,d\rho
\] 
is smooth in $t\in[0,T]$, where $\tilde{G}^{C\pm}_{k,\alpha}$ is as in Section \ref{sec_asym_ACalpha_cp_bulk_km1}.

The arguments in the last Section \ref{sec_asym_ACalpha_cp_robin_0} yield a smooth solution $h_{k+1,\alpha}:I\times[0,T]\rightarrow\R$ of \eqref{eq_asym_ACalpha_in_hk} from Section \ref{sec_asym_ACalpha_in} together with \eqref{eq_asym_ACalpha_cp_hk}. Again, the latter can be extended to a smooth function on $I_\mu\times[0,T]$. Moreover, Section \ref{sec_asym_ACalpha_in} determines $\hat{u}^I_{k+2,\alpha}$ (solving \eqref{eq_asym_ACalpha_in_uk}) with $\hat{u}^I_{k+2,\alpha}\in R^I_{k+1,(\beta_1),\alpha}$ for all $\beta_1\in(0,\min\{\sqrt{f''(\pm 1)}\})$. Therefore the $(k+1)$-th inner order is computed. Moreover, it holds $G^{C\pm}_{k,\alpha}\in R^C_{k+1,(\beta,\gamma),\alpha}$ as well as $g^{C\pm}_{k+1}\in R^C_{k+1,(\beta),\alpha}$ for all $\beta\in[0,\beta_0),\gamma\in[\frac{\gamma_0}{2},\gamma_0)$ and they are independent of $\hat{u}^{C\pm}_{k+1,\alpha}$. Finally, Theorem \ref{th_hp_alpha_sol_reg3} yields a unique smooth solution $\hat{u}^{C\pm}_{k+1,\alpha}$ to  \eqref{eq_asym_ACalpha_cp_bulk_km1} and \eqref{eq_asym_ACalpha_cp_bck} such that $\hat{u}^{C\pm}_{k+1,\alpha}\in R^C_{k+1,(\beta,\gamma),\alpha}$ for all $\beta\in[0,\beta_0)$, $\gamma\in[\frac{\gamma_0}{2},\gamma_0)$. Hence the $(k+1)$-th order is determined.

Finally, the $j$-th order is determined inductively for all $j=0,...,M$, the $h_{j,\alpha}$ are smooth and $\hat{u}^I_{j+1,\alpha}\in R^I_{j,(\beta_1),\alpha}$ for all $\beta_1\in(0,\min\{\sqrt{f''(\pm1)}\})$ as well as $\hat{u}^{C\pm}_{j,\alpha}\in R^C_{j,(\beta,\gamma),\alpha}$ for all $\beta\in[0,\beta_0)$, $\gamma\in[\frac{\gamma_0}{2},\gamma_0)$.

\subsection{The Approximate Solution for (AC$_\alpha$) in 2D}\label{sec_asym_ACalpha_uA}
Let $\sigma_\alpha$ be as in Definition \ref{th_ACalpha_sigma_def} and $\alpha_0>0$ be as in Remark \ref{th_asym_ACalpha_decay_param}. Moreover, let $\Gamma:=(\Gamma_t)_{t\in[0,T]}$ be as in Section \ref{sec_coord_surface_requ} with contact angle $\alpha\in\frac{\pi}{2}+[-\alpha_0,\alpha_0]$ and a solution to \eqref{MCF} in $\Omega$. Additionally, let $\delta>0$ be such that the assertions of Theorem \ref{th_coord2D} hold for $\delta$ replaced by $2\delta$ and consider $r,s$ from the theorem. Let $M\in\N$, $M\geq 2$ be as in the inception of Section \ref{sec_asym_ACalpha}. Moreover, let $\delta_0\in(0,\delta]$ be small such that $-2\delta_0+\mu_0\sin\alpha>0$ and 
\begin{align}\label{eq_asym_ACalpha_uA_delta0}
s^\pm=\frac{1}{\sin\alpha}[z_\alpha^\pm+r\cos\alpha]\in(\tfrac{5}{4},\tfrac{7}{4})\mu_0\quad\text{ for } z_\alpha^\pm\in[\tfrac{11}{8},\tfrac{3}{2}]\mu_0\sin\alpha\text{ and }|r|\leq\delta_0,
\end{align}
where $\mu_0$ is from Theorem \ref{th_coord2D}. Note that \eqref{eq_asym_ACalpha_uA_delta0} is only needed in order to have a suitable partition of $\Gamma(\delta_0)$ for the spectral estimate later, cf.~\eqref{eq_SE_ACalpha_sets} below. Finally, let $\eta, \tilde{\eta}:\R\rightarrow[0,1]$ be smooth with 
$\eta(r)=1$ for $|r|\leq 1$, $\eta(r)=0$ for $|r|\geq 2$ and $\tilde{\eta}(r)=1$ for $r\leq 1$, $\tilde{\eta}(r)=0$ for $r\geq 2$. Then we set 
\begin{align}\label{eq_asym_ACalpha_uB}
u^{B\pm}_{\varepsilon,\alpha}:=\chi_\alpha u^I_{\varepsilon,\alpha}+u^{C\pm}_{\varepsilon,\alpha}
=v_\alpha+\chi_\alpha\tilde{u}^I_{\varepsilon,\alpha}+\tilde{u}^{C\pm}_{\varepsilon,\alpha}
\end{align}
for $\varepsilon>0$, where $\chi_\alpha$ and $v_\alpha$ are evaluated at $(\rho_{\varepsilon,\alpha},Z_{\varepsilon,\alpha}^\pm)$. The appearing functions were constructed in Sections \ref{sec_asym_ACalpha_in}-\ref{sec_asym_ACalpha_cp}.  
Then we define
\begin{align*}
u^A_{\varepsilon,\alpha}:=
\begin{cases}
\eta(\frac{r}{\delta_0})\left[\tilde{\eta}(\frac{s^\pm}{\mu_0})u^{B\pm}_{\varepsilon,\alpha}+(1-\tilde{\eta}(\frac{s^\pm}{\delta_0}))u^I_{\varepsilon,\alpha}\right]+(1-\eta(\frac{r}{\delta_0}))\textup{sign}(r)\quad&\text{ in }\overline{\Gamma^\pm(2\delta,1)},\\
\pm 1\quad&\text{ in }Q_T^\pm\setminus\Gamma(2\delta),
\end{cases}
\end{align*}
where $s^\pm=\pm 1\mp s$ and the sets were defined in Remark \ref{th_coord2D_rem2},~1. In the following lemma we show that this gives a suitable approximate solution for \eqref{eq_ACalpha1}-\eqref{eq_ACalpha3}.
\begin{Lemma}\label{th_asym_ACalpha_uA}
	The mapping $u^A_{\varepsilon,\alpha}$ is smooth, uniformly bounded in $x,t,\varepsilon$ and the remainders
	$r^A_{\varepsilon,\alpha} := 
	(\partial_t-\Delta)u^A_{\varepsilon,\alpha} +\frac{1}{\varepsilon^2}f'(u^A_{\varepsilon,\alpha})$ and $s^A_{\varepsilon,\alpha}:=\partial_{N_{\partial\Omega}}u^A_{\varepsilon,\alpha}+\frac{1}{\varepsilon}\sigma_\alpha'(u^A_{\varepsilon,\alpha})$ in \eqref{eq_ACalpha1}-\eqref{eq_ACalpha2} satisfy
	\begin{alignat*}{2}
	|r^A_{\varepsilon,\alpha}|&
	\leq 
	C(\varepsilon^{M-1} e^{-c(|\rho_{\varepsilon,\alpha}|+Z_{\varepsilon,\alpha}^\pm)}
	+\varepsilon^M e^{-c|\rho_{\varepsilon,\alpha}|}+\varepsilon^{M+1})
	&\quad &\text{ in }\Gamma^\pm(2\delta,1),\\
	r^A_{\varepsilon,\alpha} &
	=0&\quad&\text{ in }Q_T\setminus\Gamma(2\delta),\\
	|s^A_{\varepsilon,\alpha}|&
	\leq C\varepsilon^M e^{-c|\rho_{\varepsilon,\alpha}|}&\quad&\text{ on }\partial Q_T\cap\Gamma(2\delta),\\
	s^A_{\varepsilon,\alpha} &
	=0&\quad&\text{ on }\partial Q_T\setminus\Gamma(2\delta)
	\end{alignat*}
	for $\varepsilon>0$ small and some constants $c,C>0$.
\end{Lemma}

\begin{Remark}\upshape
	The analogous statements as in \cite{MoserACvACND}, Remark 5.11 are true.
\end{Remark}

\begin{proof}
	The proof is analogous to the one of \cite{MoserACvACND}, Lemma 5.10, where the case $\alpha=\frac{\pi}{2}$ is shown. One uses Lemma \ref{th_asym_ACalpha_in} and verifies that the Taylor expansions and remainder estimates stated in Section \ref{sec_asym_ACalpha_cp} hold rigorously. The main point left to show in the case $\alpha\neq\frac{\pi}{2}$ is the suitable convergence with respect to $\varepsilon\rightarrow 0$ (i.e.~rates of type $e^{-c/\varepsilon}$ for the function and all derivatives) in the transition regions for the functions we glued together in the definition of $u^A_{\varepsilon,\alpha}$. With the latter, one can then estimate the mixed terms in $r^A_{\varepsilon,\alpha}$ and $s^A_{\varepsilon,\alpha}$ appearing due to the cutoff functions similar to the case $\alpha=\frac{\pi}{2}$. Because of \eqref{eq_asym_ACalpha_uB} and the asymptotics of the appearing functions it is enough to prove 
	\[
	z^\pm_\alpha|_{(x,t)}\geq c>0\quad\text{ and }\quad\chi_\alpha(\rho_{\varepsilon,\alpha},Z_{\varepsilon,\alpha}^\pm)|_{(x,t)}\equiv 1
	\] 
	for all $(x,t)\in\overline{\Gamma^\pm(2\delta,1)}$ with $|r(x,t)|\leq 2\delta^0$ and $s^\pm(x,t)\geq \mu_0$ as well as $\varepsilon>0$ small. However, by the assumption on $\delta_0$ it holds 
	\[
	z^\pm_\alpha|_{(x,t)}=|z^\pm_\alpha|_{(x,t)}|\geq-|r|_{(x,t)}\\\cos\alpha|+s^\pm|_{(x,t)}\sin\alpha 
	\geq -2\delta_0 +\mu_0\sin\alpha>0
	\]
	for all those $(x,t)$. Moreover, recall the definition \eqref{eq_asym_ACalpha_cp_cutoff2} of $\hat{\chi}_\alpha$. Therefore it is left to show that
	\[
	\frac{1}{\sin\alpha}[Z_{\varepsilon,\alpha}^\pm|_{(x,t)}+\rho_{\varepsilon,\alpha}|_{(x,t)}\cos\alpha-H_0]\geq 1
	\]
	for all the $(x,t)$ as above and $\varepsilon>0$ small, where $H_0=2\|h_{1,\alpha}\|_{\infty}$, see Remark \ref{th_asym_ACalpha_cp_cutoff}, 1.~and the end of Section \ref{sec_asym_ACalpha_cp_robin_0}. The estimate follows from
	\[
	\frac{1}{\sin\alpha}[Z_{\varepsilon,\alpha}^\pm|_{(x,t)}+\rho_{\varepsilon,\alpha}|_{(x,t)}\cos\alpha-H_0]
	=\frac{s^\pm|_{(x,t)}}{\varepsilon} -\frac{1}{\sin\alpha}[h_{\varepsilon,\alpha}|_{(s(x,t),t)}\cos\alpha+H_0]
	\]
	for all the $(x,t)$ as before and $\varepsilon>0$. The second term is estimated by $4\|h_{1,\alpha}\|_\infty/\sin\alpha$ for $\varepsilon>0$ small. Hence the lemma is proven.
\end{proof}

\section{Spectral Estimate for (AC$_\alpha$) in 2D}\label{sec_SE_ACalpha}
In the second step for the method of de Mottoni and Schatzman \cite{deMS} one estimates the difference of the exact and approximate solutions. Therefore typically a Gronwall-type argument is used combined with the idea of linearization at the approximate solution, because the structure of the latter is known. An important ingredient is the spectral estimate for a linear operator related to the diffuse interface model and the approximate solution, i.e.~an estimate for the corresponding bilinear form.

In this section we prove such a spectral estimate for the Allen-Cahn equation with nonlinear Robin boundary condition \eqref{eq_ACalpha1}-\eqref{eq_ACalpha3} for $\alpha$ close to $\frac{\pi}{2}$ in the setting of the introduction, i.e.~the case of boundary contact in 2D. The operator is given by
\[
\Lc_{\varepsilon,t}:=-\Delta + \frac{1}{\varepsilon^2} f''(u^A_{\varepsilon,\alpha}(.,t))\quad\text{ in }\Omega
\]
together with the linear Robin boundary condition
\[
\Nc_{\varepsilon,t}u:=
N_{\partial\Omega}\cdot\nabla u + \frac{1}{\varepsilon} \sigma_\alpha''(u^A_{\varepsilon,\alpha}(.,t))u=0\quad\text{ on }\partial\Omega, 
\]
where $f$ is as in \eqref{eq_AC_fvor1}, $\sigma_\alpha$ is from Definition \ref{th_ACalpha_sigma_def} and $u^A_\varepsilon$ is from Section \ref{sec_asym_ACalpha_uA}. We will prove a spectral estimate of the following form (cf.~ also Theorem \ref{th_SE_ACalpha} below): if $\alpha$ is close enough to $\frac{\pi}{2}$ (independent of $\Omega,\Gamma$ etc.), then there are $c_0,C,\varepsilon_0>0$ such that for all $u\in H^1(\Omega)$,  $\varepsilon\in(0,\varepsilon_0]$:
\begin{align}
\begin{split}\label{eq_SE1}
\int_\Omega|\nabla\psi|^2
&+\frac{1}{\varepsilon^2}f''(u^A_{\varepsilon,\alpha}|_{(.,t)})\psi^2\,dx
+\int_{\partial\Omega}\frac{1}{\varepsilon}\sigma_\alpha''(u^A_{\varepsilon,\alpha}|_{(.,t)})(\tr\,\psi)^2\,d\Hc^1\\
&\geq -C\|\psi\|_{L^2(\Omega)}^2+\|\nabla\psi\|_{L^2(\Omega\setminus\Gamma_t(\delta_0))}^2+c_0\varepsilon\|\nabla_\tau\psi\|_{L^2(\Gamma_t(\delta_0))}^2,\end{split}
\end{align}
where $\nabla_\tau$ is the tangential gradient introduced in Remark \ref{th_coord2D_rem2},~2. The estimate (also without the two additional last terms in \eqref{eq_SE1}) yields that the spectrum of $\Lc_{\varepsilon,t}$ is bounded from below by $-C$, where $\Lc_{\varepsilon,t}$ is interpreted as an unbounded operator on $L^2(\Omega)$ with domain $\{u\in H^2(\Omega):\Nc_{\varepsilon,t}u=0 \}$. Finally, note that also a weaker spectral estimate (e.g.~without the last two additional terms) would be enough to obtain a convergence result, see Remark \ref{th_DC_ACalpha_rem} below.

In the following we explain our approach for the proof of the spectral estimate \eqref{eq_SE1}. Using a reduction strategy in analogy to Chen \cite{ChenSpectrums} might work here, but even in the case $\alpha=\frac{\pi}{2}$ this is not clear, cf.~\cite{MoserACvACND}, Remark 6.1. In \cite{AbelsMoser}, Section 4 the case $\alpha=\frac{\pi}{2}$ is carried out in a different way as in \cite{ChenSpectrums} (see below for a summary). Since we have to take $\alpha$ close to $\frac{\pi}{2}$ anyway because of Section \ref{sec_hp_alpha_sol} and the asymptotic expansion in Section \ref{sec_asym_ACalpha}, we chose to adapt the proof in \cite{AbelsMoser} and to take $\alpha$ close to $\frac{\pi}{2}$ where needed such that the arguments still work. Here one has to pay attention on the restrictions for $\alpha$ because a dependency on $\Gamma$ or $\Omega$ is of course not allowed.

At this point let us briefly explain the idea in \cite{AbelsMoser}, Section 4 for the proof of the spectral estimate in the case $\alpha=\frac{\pi}{2}$. First the spectral estimate is localized with a suitable partition of unity. The estimate away from a neighbourhood of the contact points is shown similar to \cite{ChenSpectrums}, Theorem 2.3 with properties of perturbed 1D-Allen-Cahn-type operators on large intervals approximating $\R$. The spectral estimate close to the contact points is obtained by a splitting technique. Therefore an approximate first eigenfunction for the Allen-Cahn operator is constructed, where the model problem on the half space mentioned in Section \ref{sec_hp_alpha_lin} for $\alpha=\frac{\pi}{2}$ is used. Then the space of $H^1$-functions is decomposed $L^2$-orthogonally with respect to the explicit space consisting of tangential alterations of this approximate eigenfunction. Finally, the bilinear form corresponding to the spectral term is estimated on each arising subspace. Here again the features of perturbed 1D-operators are important.

In the following we adapt this proof to the case of $\alpha$ close to $\frac{\pi}{2}$. In \cite{MoserACvACND}, Section 6.1 (or see also \cite{MoserDiss}, Section 6.1) the 1D-preliminaries are proven in an abstract setting independent of the particular geometry and also applicable for the case $\alpha\neq\frac{\pi}{2}$, including scaling transformations and remainder estimates as well as spectral estimates for perturbed 1D Allen-Cahn-type operators. 

\subsection{Assumptions and Spectral Estimate}
\label{sec_SE_AC}
We state the assumptions for this section. For convenience similar notation as in Section \ref{sec_asym_ACalpha} is used. We consider $\beta_0,\gamma_0,\alpha_0>0$, $\alpha\in\frac{\pi}{2}+[-\alpha_0,\alpha_0]$ and $v_\alpha$ as in Remark \ref{th_asym_ACalpha_decay_param}.
Let $\Omega\subset\R^2$ and $\Gamma=(\Gamma_t)_{t\in[0,T]}$ for $T>0$ be as in Section \ref{sec_coord2D} with contact angle $\alpha$ (\eqref{MCF} not needed). Moreover, let $\delta_1>0$ be such that Theorem \ref{th_coord2D} holds for $2\delta_1$ instead of $\delta$. We use the notation from Theorem \ref{th_coord2D} and Remark \ref{th_coord2D_rem2}. Furthermore, we consider height functions $h_{1,\alpha}$ and $h_{2,\alpha}=h_{2,\alpha}(\varepsilon)$ for $\varepsilon>0$ small. We require (with a slight abuse of notation)
\[
h_{j,\alpha}\in B([0,T],C^0(I_\mu)\cap C^2(\hat{I}_\mu))\quad\text{ for }j=1,2,
\]
where $I_\mu:=[-1-\mu,1+\mu]$ for some $\mu>0$ large and $\hat{I}_\mu:=I_\mu\setminus(-1+2\mu_0,1-2\mu_0)$. Moreover, we assume that there exists $\overline{C}_0>0$ such that $\|h_{j,\alpha}\|_{B([0,T],C^0(I_\mu)\cap C^2(\hat{I}_\mu))}\leq \overline{C}_0$ for $j=1,2$. Then for $\varepsilon>0$ small we set $h_{\varepsilon,\alpha}:=h_{1,\alpha}+\varepsilon h_{2,\alpha}$ and define the stretched variables
\[
\rho_{\varepsilon,\alpha}:=\frac{r-\varepsilon h_{\varepsilon,\alpha}(s,t)}{\varepsilon},\quad
Z_{\varepsilon,\alpha}^\pm:=\frac{z_\alpha^\pm}{\varepsilon}\quad\text{ in }\overline{\Gamma(2\delta_1)},
\]
where $z_\alpha^\pm=-r\cos\alpha + (1\mp s)\sin\alpha$ is as in \eqref{eq_asym_ACalpha_cp_zalpha}. Moreover, let $\delta_0\in(0,\delta_1]$ be small such that \eqref{eq_asym_ACalpha_uA_delta0} holds. We set $\hat{\mu}_0:=\frac{11}{8}\mu_0\sin\alpha$ and $\tilde{\mu}_0:=\frac{3}{2}\mu_0\sin\alpha$ as well as for $t\in[0,T]$: 
\[
\Omega^{C\pm}_t:=\{x\in\Gamma_t(\delta_0):z_\alpha^\pm(x,t)\in(0,\tilde{\mu}_0)\}.
\] 
Let $\hat{u}^{C\pm}_{1,\alpha}:\overline{\R^2_+}\times[0,T]\rightarrow\R:(\rho,Z,t)\mapsto\hat{u}^{C\pm}_{1,\alpha}(\rho,Z,t)$ be in $B([0,T],H^2_{(0,\frac{\gamma_0}{2})}(\R^2_+))$. Then we set 
\[
u^{C\pm}_{1,\alpha}(x,t):=\hat{u}^{C\pm}_{1,\alpha}(\rho_{\varepsilon,\alpha}(x,t),Z_{\varepsilon,\alpha}^\pm(x,t),t)\quad\text{ for }(x,t)\in\Omega_t^{C\pm}.
\] 
For $\varepsilon>0$ small we consider
\begin{align*}
u^A_{\varepsilon,\alpha}=
\begin{cases}
\theta_0(\rho_{\varepsilon,\alpha})+\Oc(\varepsilon^2)&\quad\text{ in }\Gamma(\delta_0,\mu_0),\\
v_\alpha(\rho_{\varepsilon,\alpha},Z_{\varepsilon,\alpha}^\pm)+\varepsilon u^{C\pm}_{1,\alpha}+\Oc(\varepsilon^2)&\quad\text{ in }\Omega^{C\pm}_t,\\
\pm 1+\Oc(\varepsilon)&\quad\text{ in }Q_T^\pm\setminus\Gamma(\delta_0),
\end{cases}
\end{align*}
where $\theta_0$ is as in Theorem \ref{th_theta_0}, $v_\alpha$ is as in Remark \ref{th_asym_ACalpha_decay_param} and $\Oc(\varepsilon^k)$ are continuous\footnote{~For evaluation on the boundary.} functions bounded by $C\varepsilon^k$. For convenience\footnote{~It would be tedious to include such terms in the asymptotic expansion for the approximate eigenfunction below and probably a cut-off structure similar as in Section \ref{sec_asym_ACalpha_cp} is needed.} we do not consider further $\varepsilon$-order terms similar to \cite{AbelsMoser}, Remark 4.3, (i). In this situation we prove
\begin{Theorem}[\textbf{Spectral Estimate for (AC$_\alpha$) in 2D}]\label{th_SE_ACalpha}
	There is an $\overline{\alpha}_0\in(0,\alpha_0]$ independent of $\Omega,\Gamma$ such that, if $\alpha\in\frac{\pi}{2}+[-\overline{\alpha}_0,\overline{\alpha}_0]$, then there exist $\varepsilon_0,C,c_0>0$ independent of the $h_{j,\alpha}$ for fixed $\overline{C}_0$ such that for all $\varepsilon\in(0,\varepsilon_0], t\in[0,T]$ and $\psi\in H^1(\Omega)$ we have
	\begin{align}\begin{split}\label{eq_SEalpha}
	\int_\Omega|\nabla\psi|^2
	&+\frac{1}{\varepsilon^2}f''(u^A_{\varepsilon,\alpha}|_{(.,t)})\psi^2\,dx
	+\int_{\partial\Omega}\frac{1}{\varepsilon}\sigma_\alpha''(u^A_{\varepsilon,\alpha}|_{(.,t)})(\tr\,\psi)^2\,d\Hc^1\\
	&\geq -C\|\psi\|_{L^2(\Omega)}^2+\|\nabla\psi\|_{L^2(\Omega\setminus\Gamma_t(\delta_0))}^2+c_0\varepsilon\|\nabla_\tau\psi\|_{L^2(\Gamma_t(\delta_0))}^2.\end{split}
	\end{align}
\end{Theorem}

\begin{Remark}\phantomsection{\label{th_SE_ACalpha_rem}}\upshape
	\begin{enumerate}
	\item Compared to the case $\alpha=\frac{\pi}{2}$ in Theorem 4.1 in \cite{AbelsMoser} we can only show the weaker estimate \eqref{eq_SEalpha} with an $\varepsilon$-factor in front of the $\nabla_\tau$-term. The reason is that the arguments in \cite{AbelsMoser}, Section 4 can be adapted except for the estimate of $\nabla_\tau\phi$ in the proof of \cite{AbelsMoser}, Theorem 4.11. In the present situation there will be a term roughly of the form
	\[
	\frac{1}{\varepsilon^3}\int_0^{\tilde{\mu}_0}a(z)^2\int_{-\delta_0}^{\delta_0}(\partial_Z\partial_\rho v_\alpha)^2|_{(\frac{r}{\varepsilon},\frac{z}{\varepsilon})}\,dr\,dz.
	\]
	We can control the latter only with an $\Oc(\varepsilon^{-1})$-term, e.g.~$\frac{1}{\varepsilon}C\|a\|_{H^1(0,\tilde{\mu}_0)}^2$, cf.~\eqref{eq_SE_ACalpha_cp2} and the proof of Theorem \ref{th_SE_ACalpha_cp2}. Therefore we need the additional $\varepsilon$. Nevertheless, the estimate still gives some control on the $\nabla_\tau$-term.
	\item Note that in this section $z$ corresponds to $z_\alpha^\pm$, whereas in \cite{AbelsMoser} the $z$ rather served as a rescaled (integral) variable over $\R$. Here we use $\rho$ for the latter instead.
	\end{enumerate}
\end{Remark}

The main task is the proof of the following spectral estimate close to the contact points:
\begin{Theorem}\label{th_SE_ACalpha_cp} 
	There is an $\overline{\alpha}_0\in(0,\alpha_0]$ independent of $\Omega,\Gamma$ such that, if $\alpha\in\frac{\pi}{2}+[-\overline{\alpha}_0,\overline{\alpha}_0]$, then there are $\tilde{\varepsilon}_0, C, \tilde{c}_0>0$ independent of the $h_{j,\alpha}$ for fixed $\overline{C}_0$ such that for all $\varepsilon\in(0,\tilde{\varepsilon}_0]$, $t\in[0,T]$ and $\psi\in H^1(\Omega^{C\pm}_t)$ with $\psi(x)=0$ for a.e.~$x\in\Omega_t^{C\pm}$ with $z_\alpha^\pm(x,t)\geq\hat{\mu}_0$ it holds:
	\begin{align*}
	\int_{\Omega^{C\pm}_t}|\nabla\psi|^2+\frac{1}{\varepsilon^2}f''(u^A_{\varepsilon,\alpha}|_{(.,t)})\psi^2\,dx
	&+\int_{\partial\Omega\cap\partial\Omega_t^{C\pm}}\frac{1}{\varepsilon}\sigma_\alpha''(u^A_{\varepsilon,\alpha}|_{(.,t)})(\tr\,\psi)^2\,d\Hc^1\\
	&\geq -C\|\psi\|_{L^2(\Omega_t^{C\pm})}^2+\tilde{c}_0\varepsilon\|\nabla_\tau\psi\|_{L^2(\Omega_t^{C\pm})}^2.
	\end{align*}
\end{Theorem}

This is sufficient to prove Theorem \ref{th_SE_ACalpha}:

\begin{proof}[Proof of Theorem \ref{th_SE_ACalpha}] 
	For $\varepsilon_0>0$ small and all $\varepsilon\in(0,\varepsilon_0]$ it holds $f''(u^A_{\varepsilon,\alpha})\geq 0$ in $Q_T^\pm\setminus\Gamma(\delta_0)$ and $\sigma_\alpha''(u^A_{\varepsilon,\alpha})=0$ on $\partial\Omega\setminus\partial\Omega^{C\pm}_t$. Therefore it is enough to prove the estimate in Theorem \ref{th_SE_ACalpha} for $\Gamma_t(\delta_0)$ instead of $\Omega$. We reduce to further subsets. The estimate holds for $\Gamma_t(\delta_0,\mu_0)$ instead of $\Omega$ with $1$ in front of the $\nabla_\tau$-term and without the boundary term. The latter was already proven in the case $\alpha=\frac{\pi}{2}$, cf.~the proof of Theorem 4.1 in \cite{AbelsMoser}. Finally, one can combine this with Theorem \ref{th_SE_ACalpha_cp} similar as in the case $\alpha=\frac{\pi}{2}$ with a suitable partition of unity for 
	\begin{align}\label{eq_SE_ACalpha_sets}
	\Gamma_t(\delta_0)\subseteq\overline{\Gamma_t(\delta_0,\mu_0)}\cup\bigcup_\pm\overline{\Gamma_t^\pm(\delta_0,5\mu_0/4)},
	\end{align}
	cf.~the proof of Theorem 4.1 in \cite{AbelsMoser}. This is possible since 
	\[
	\Gamma_t^\pm(\delta_0,\frac{5}{4}\mu_0)
	\subseteq 
	\{
	x\in\Gamma_t(\delta_0):z_\alpha^\pm(x,t)\in(0,\hat{\mu}_0)
	\}
	\subseteq\Omega_t^{C\pm}
	\subseteq\Gamma_t^\pm(\delta_0,\frac{7}{4}\mu_0)
	\]
	for $t\in[0,T]$ due to \eqref{eq_asym_ACalpha_uA_delta0}.
\end{proof}

\subsection{Outline for the Proof of the Spectral Estimate close to the Contact Points}\label{sec_SE_ACalpha_outline}
Because of a Taylor expansion it is enough to prove Theorem \ref{th_SE_ACalpha_cp} for 
\[
\frac{1}{\varepsilon^2}f''(v_\alpha|_{(\rho_{\varepsilon,\alpha}(.,t),Z_{\varepsilon,\alpha}^\pm(.,t))})+\frac{1}{\varepsilon}f'''(v_\alpha|_{(\rho_{\varepsilon,\alpha}(.,t),Z_{\varepsilon,\alpha}^\pm(.,t))})u^{C\pm}_{1,\alpha}(.,t)
\] 
instead of $\frac{1}{\varepsilon^2}f''(u^A_{\varepsilon,\alpha}(.,t))$. Moreover, we can replace $\frac{1}{\varepsilon}\sigma_\alpha''(u^A_{\varepsilon,\alpha}(.,t))$ by 
\[
\frac{1}{\varepsilon}\sigma_\alpha''(v_\alpha|_{(\rho_{\varepsilon,\alpha}(.,t),0)})+\sigma_\alpha'''(v_\alpha|_{(\rho_{\varepsilon,\alpha}(.,t),0)})u^{C\pm}_{1,\alpha}(.,t)
\] 
due to Young's inequality and the second estimate in the following lemma. 
\begin{Lemma}\phantomsection{\label{th_SE_ACalpha_intpol_tr1}}
	There is a $\overline{C}_1>0$ (independent of $\psi$, $t$) such that 
	\begin{align*}
	\|\tr\,\psi\|_{L^2(\partial\Omega_t^{C\pm})}^2
	&\leq \overline{C}_1
	\left[\|\psi\|_{L^2(\Omega_t^{C\pm})}^2
	+\|\nabla \psi\|_{L^2(\Omega_t^{C\pm})}\|\psi\|_{L^2(\Omega_t^{C\pm})}\right],\\
	\|\tr\,\psi\|_{L^2(\partial\Omega\cap\partial\Omega_t^{C\pm})}^2
	&\leq \overline{C}_1
	\left[\|\psi\|_{L^2(\Omega_t^{C\pm})}^2
	+\|\nabla_\tau \psi\|_{L^2(\Omega_t^{C\pm})}\|\psi\|_{L^2(\Omega_t^{C\pm})}\right]
	\end{align*}
	for all $\psi\in H^1(\Omega_t^{C\pm})$ and $t\in[0,T]$.
\end{Lemma}
\begin{proof}
	The first estimate follows analogously to the proof of (4.12) in \cite{AbelsMoser}. For convenience we do not go into details. The proof of the second estimate is similar: it is equivalent to prove the estimate for $S:=\{(r,s)\in\R^2:r\in(-\delta_0,\delta_0),s\in\pm[1-\frac{3}{2}\mu_0,1]\mp\frac{\cos\alpha}{\sin\alpha}r\}$, $S_{\delta_0,\alpha}^\pm$
	instead of $\Omega^{C\pm}_t, \partial\Omega\cap\partial\Omega^{C\pm}_t$ as well as  $\partial_s$ instead of $\nabla_\tau$. For the latter we use the same idea as in the proof of (4.12) in \cite{AbelsMoser}. Here $\vec{w}\in C^1(\overline{S})^2$ with $\vec{w}\cdot N_{\partial S}\geq 1$ on $S_{\delta_0,\alpha}^\pm$ and $\vec{w}\cdot N_{\partial S}=0$ on $\partial S\setminus S_{\delta_0,\alpha}^\pm$ as well as $w_1=0$ yields the claim.
\end{proof}

We construct an approximate first eigenfunction of $\phi^A_{\varepsilon,\alpha}(.,t)$
\[
\Lc_{\varepsilon,t}^\pm:=-\Delta+\frac{1}{\varepsilon^2}f''(v_\alpha|_{(\rho_{\varepsilon,\alpha}(.,t),Z_{\varepsilon,\alpha}^\pm(.,t))})+\frac{1}{\varepsilon}f'''(v_\alpha|_{(\rho_{\varepsilon,\alpha}(.,t),Z_{\varepsilon,\alpha}^\pm(.,t))})u^{C\pm}_{1,\alpha}(.,t)\quad\text{ on }\Omega_t^{C\pm}
\] 
together with the linear Robin boundary condition $\Nc_{\varepsilon,t}^\pm u=0$ on $\partial\Omega^{C\pm}_t$, where
\[
\Nc_{\varepsilon,t}^\pm u:=\!\left[N_{\partial\Omega^{C\pm}_t}\cdot\nabla+ \chi_{\partial\Omega}\Bigl[\frac{1}{\varepsilon}\sigma_\alpha''(v_\alpha|_{(\rho_{\varepsilon,\alpha}(.,t),0)})+\sigma_\alpha'''(v_\alpha|_{(\rho_{\varepsilon,\alpha}(.,t),0)})u^{C\pm}_{1,\alpha}|_{(.,t)}\Bigr]\right]\!u\:\:\text{ on }\partial\Omega^{C\pm}_t.
\]
Here $\chi_{\partial\Omega}$ is the characteristic function of $\partial\Omega$. In analogy to the case $\alpha=\frac{\pi}{2}$, cf.~\cite{AbelsMoser}, Section 4.1, we use the ansatz
\begin{alignat*}{2}
\phi^A_{\varepsilon,\alpha}(.,t)&:=\frac{1}{\sqrt{\varepsilon}}[v^{C\pm}_{\varepsilon,0}(.,t)+\varepsilon v^{C\pm}_{\varepsilon,1}(.,t)]&\quad&\text{ on }\Omega^{C\pm}_t,\\
v^{C\pm}_{\varepsilon,0}(.,t)&:=\hat{v}^{C\pm}_0|_{(\rho_{\varepsilon,\alpha}(.,t),Z_{\varepsilon,\alpha}^\pm(.,t),z_\alpha^\pm(.,t),t)}:= q^\pm|_{(z_\alpha^\pm(.,t),t)}\partial_\rho v_\alpha|_{(\rho_{\varepsilon,\alpha}(.,t),Z_{\varepsilon,\alpha}^\pm(.,t))}&\quad&\text{ on }\Omega^{C\pm}_t,\\ v^{C\pm}_{\varepsilon,1}(.,t)&:=\hat{v}^{C\pm}_1|_{(\rho_{\varepsilon,\alpha}(.,t),Z_{\varepsilon,\alpha}^\pm(.,t),t)}&\quad&\text{ on }\Omega^{C\pm}_t,
\end{alignat*}
where $q^\pm:[0,\tilde{\mu}_0]\times[0,T]\rightarrow\R:(z,t)\mapsto q^\pm(z,t)$ and $\hat{v}^{C\pm}_0:\overline{\R^2_+}\times[0,\tilde{\mu}_0]\times[0,T]\rightarrow\R$ as well as $\hat{v}^{C\pm}_1:\overline{\R^2_+}\times[0,T]\rightarrow\R$. 

In Subsection \ref{sec_SE_ACalpha_asym} we expand $\Lc_{\varepsilon,t}^\pm\phi^A_{\varepsilon,\alpha}(.,t)$ and $\Nc_{\varepsilon,t}^\pm\phi^A_{\varepsilon,\alpha}(.,t)$ with similar ideas as in Section \ref{sec_asym_ACalpha_cp} and choose $q^\pm$ and $\hat{v}^{C\pm}_1$ appropriately.  The $q^\pm$-term is used to enforce the compatibility condition for the equations for $\hat{v}^{C\pm}_1$. Then in Subsection \ref{sec_SE_ACalpha_splitting} we decompose 
\begin{align}\label{eq_SE_ACalpha_H1tilde_Omega}
\tilde{H}^1(\Omega^{C\pm}_t):=
\left\{\psi\in H^1(\Omega^{C\pm}_t): \psi(x)=0\text{ for a.e.~}x\in\Omega_t^{C\pm}\text{ with }z_\alpha^\pm(x,t)\geq\hat{\mu}_0\right\}
\end{align}
in a suitable way. To this end we consider
\begin{align}\label{eq_SE_ACalpha_V_def}
\hat{V}_{\varepsilon,t}^\pm
&:=\left\{\phi=a(z_\alpha^\pm(.,t))\phi^A_{\varepsilon,\alpha}(.,t):
a\in\hat{H}^1(0,\tilde{\mu}_0)\right\},\\
\hat{H}^1(0,\tilde{\mu}_0)&:=
\left\{a\in H^1(0,\tilde{\mu}_0) : a=0\text{ on }[\hat{\mu}_0,\tilde{\mu}_0], a(0)=0\right\}.\label{eq_SE_ACalpha_H1tilde_interval}
\end{align}
Finally, in Subsection \ref{sec_SE_ACalpha_BLF} we analyze the bilinear form $B_{\varepsilon,t}^\pm$ related to $\Lc_{\varepsilon,t}^\pm$ on $\hat{V}_{\varepsilon,t}^\pm\times \hat{V}_{\varepsilon,t}^\pm$, $(\hat{V}_{\varepsilon,t}^\pm)^\perp\times (\hat{V}_{\varepsilon,t}^\pm)^\perp$ and $\hat{V}_{\varepsilon,t}^\pm\times (\hat{V}_{\varepsilon,t}^\pm)^\perp$. Here for $\phi,\psi\in H^1(\Omega_t^{C\pm})$ let
\begin{align}\begin{split}\label{eq_SE_ACalpha_Bepst}
B_{\varepsilon,t}^\pm(\phi,\psi):=\int_{\partial\Omega\cap\partial\Omega^{C\pm}_t}\left[\frac{1}{\varepsilon}\sigma_\alpha''(v_\alpha|_{(\rho_{\varepsilon,\alpha},0)})+\sigma_\alpha'''(v_\alpha|_{(\rho_{\varepsilon,\alpha},0)})u^{C\pm}_{1,\alpha}\right]\!|_{(.,t)}\tr\,\phi\,\tr\,\psi\,d\Hc^1\\
+\int_{\Omega_t^{C\pm}}\nabla\phi\cdot\nabla\psi+\left[\frac{1}{\varepsilon^2}f''(v_\alpha|_{(\rho_{\varepsilon,\alpha},Z_{\varepsilon,\alpha}^\pm)})+\frac{1}{\varepsilon}f'''(v_\alpha|_{(\rho_{\varepsilon,\alpha},Z_{\varepsilon,\alpha}^\pm)})u^{C\pm}_{1,\alpha}\right]\!|_{(.,t)}\phi\psi\,dx.\end{split}
\end{align}
\begin{Remark}\upshape
	Note that in \cite{AbelsMoser}, Section 4 the whole $H^1$-space is considered instead of \eqref{eq_SE_ACalpha_H1tilde_Omega}, but this is not needed for the proof of the spectral estimate close to the contact points because of the cutoff functions, cf.~\cite{AbelsMoser}, proof of Theorem 4.1. Therefore we can work with the smaller space in \eqref{eq_SE_ACalpha_H1tilde_Omega}. Then the behaviour of the approximate eigenfunction $\phi^A_{\varepsilon,\alpha}$ close to the interior boundary $[\partial\Omega_t^{C\pm}\setminus\partial\Omega]\cap\Gamma_t(\frac{\delta_0}{2})$ is not important. Otherwise we would have to carry out an asymptotic expansion close to the interior boundary, too.
\end{Remark}

\subsection{Asymptotic Expansion for the Approximate Eigenfunction}\label{sec_SE_ACalpha_asym}
\begin{proof}[Asymptotic Expansion of $\sqrt{\varepsilon}\Lc_{\varepsilon,t}^\pm\phi^A_{\varepsilon,\alpha}(.,t)$.] 
	In $\sqrt{\varepsilon}\Delta\phi^A_{\varepsilon,\alpha}(.,t)$  there are some additional terms due to $q^\pm$ compared to the formula in Lemma \ref{th_asym_ACalpha_cp_trafo}. More precisely, via direct computation we get
	\begin{align*}
	\sqrt{\varepsilon}\Delta\phi^A_{\varepsilon,\alpha} &=(q^\pm\partial_\rho^2v_\alpha+\varepsilon\partial_\rho\hat{v}^{C\pm}_1)\left[\frac{\Delta r}{\varepsilon}-(\Delta s\partial_sh_{\varepsilon,\alpha}+|\nabla s|^2\partial_s^2h_{\varepsilon,\alpha})\right]\\
	&+(q^\pm\partial_Z\partial_\rho v_\alpha+\varepsilon\partial_Z\hat{v}^{C\pm}_1)
	\frac{\Delta z^\pm_\alpha}{\varepsilon}
	+(q^\pm\partial_Z^2\partial_\rho v_\alpha+\varepsilon\partial_Z^2\hat{v}^{C\pm}_1)\frac{|\nabla z^\pm_\alpha|^2}{\varepsilon^2}\\
	&+ 2(q^\pm\partial_Z\partial_\rho^2 v_\alpha+\varepsilon\partial_Z\partial_\rho\hat{v}^{C\pm}_1)\,\frac{\nabla z^\pm_\alpha}{\varepsilon}\cdot\left[\frac{\nabla r}{\varepsilon}-\nabla s\partial_sh_{\varepsilon,\alpha}\right]\\
	&+(q^\pm\partial_\rho^3 v_\alpha+\varepsilon\partial_\rho^2 \hat{v}^{C\pm}_1)\left|\frac{\nabla r}{\varepsilon}-\nabla s\partial_sh_{\varepsilon,\alpha}\right|^2\\
	&+2\partial_zq^\pm\partial_Z\partial_\rho v_\alpha\frac{|\nabla z_\alpha^\pm|^2}{\varepsilon}
	+2\partial_zq^\pm\partial_\rho^2v_\alpha\nabla z_\alpha^\pm\cdot\left[\frac{\nabla r}{\varepsilon}-\nabla s\partial_sh_{\varepsilon,\alpha}\right]\\
	&+\partial_zq^\pm\partial_\rho v_\alpha\Delta z_\alpha^\pm+\partial_z^2q^\pm \partial_\rho v_\alpha|\nabla z_\alpha^\pm|^2,	
	\end{align*}
	with evaluations as in Lemma \ref{th_asym_ACalpha_cp_trafo} except that the $q^\pm$-terms are evaluated at $(z_\alpha^\pm(x,t),t)$. 
	
	The difficulty in expanding $\sqrt{\varepsilon}\Delta\phi^A_{\varepsilon,\alpha}$ is that for the $v_\alpha$-terms without a derivative in $Z$, i.e.~terms with the factors $\partial_\rho^k v_\alpha$, $k=1,2,3$, we do not have exponential decay estimates with respect to $Z$. Therefore we have to expand the corresponding factors in a more subtle way than in Section \ref{sec_asym_ACalpha_cp_bulk}, where this problem was solved with a suitable ansatz. However, we only need the expansion up to order $\frac{1}{\varepsilon}$ and for the remainder terms a decay in normal direction as in the case $\alpha=\frac{\pi}{2}$ in \cite{AbelsMoser}, Lemma 4.4 will be enough. Therefore we do not need to expand terms of order $\varepsilon^0$ in the formula for $\sqrt{\varepsilon}\Delta\phi^A_{\varepsilon,\alpha}$ above, in particular higher regularity for $\partial_s^2h_{\varepsilon,\alpha}$ is not necessary.
	This leads to the following expansion procedure for $\sqrt{\varepsilon}\Lc_{\varepsilon,t}^\pm\phi^A_{\varepsilon,\alpha}(.,t)$:
	
	Terms of order $\Oc(1)$ are not expanded. The other terms are expanded as follows. For $(x,t)$-terms that are not multiplied by a term with a $Z$-derivative we only use a Taylor-expansion in normal direction analogous to \cite{AbelsMoser}, (3.2) and replace $r$ by $\varepsilon(\rho+h_{\varepsilon,\alpha}(s,t))$. But then we leave untouched all the appearing $h_{\varepsilon,\alpha}$-terms that are not multiplied with a term including a $Z$-derivative. Moreover, for all the other $(x,t)$-terms we apply the full Taylor expansion \eqref{eq_asym_ACalpha_cp_taylor3} and replace $r$ as above and $s\mp 1$ via \eqref{eq_asym_ACalpha_cp_s2}. Then we rewrite the $h_{\varepsilon,\alpha}$-terms that are multiplied by a term with a $Z$-derivative via
	\[
	\partial_s^kh_{\varepsilon,\alpha}|_{(s(x,t),t)}=\partial_s^kh_{\varepsilon,\alpha}|_{(\pm 1\mp\frac{1}{\sin\alpha}z_\alpha^\pm(x,t),t)}=\partial_s^kh_{\varepsilon,\alpha}|_{(\pm 1,t)}+\Oc(|z_\alpha^\pm|_{(x,t)}|)\quad\text{ for }k=0,1.
	\]
	Regarding the $q^\pm$-terms we only rewrite the ones in front of the $\partial_Z\partial_\rho v_\alpha$, the ones multiplied with the $\varepsilon$-orders of $|\nabla z_\alpha^\pm|^2$ or $\nabla r\cdot\nabla z_\alpha^\pm$ as well as the one multiplied with $f^{(3)}(v_\alpha)$ via the formula
    $\partial_z^kq^\pm(z,t)=\partial_z^kq^\pm(0,t)+\Oc(|z|)$ for $k=0,1$. Note that the remainder term stemming from the $f^{(3)}(v_\alpha)$-term can be controlled because $\hat{u}^{C\pm}_{1,\alpha}\partial_\rho v_\alpha$ has appropriate decay. The $z$-remainders will only contribute to order $\varepsilon^0$ in the expansion of $\sqrt{\varepsilon}\Delta\phi^A_{\varepsilon,\alpha}$ due to $z_\alpha^\pm=\varepsilon Z_{\varepsilon,\alpha}^\pm$.
	
	At the lowest order $\Oc(\frac{1}{\varepsilon^2})$ in $\sqrt{\varepsilon}\Lc_{\varepsilon,t}^\pm\phi^A_{\varepsilon,\alpha}(.,t)$ we obtain 
	\[
	\frac{1}{\varepsilon^2}q^\pm(z,t)
	\left[-\partial_Z^2+2\cos\alpha\partial_\rho\partial_Z-\partial_\rho^2+f''(v_\alpha)\right]\partial_\rho v_\alpha =0
	\]
	due to \eqref{eq_hp_alpha_modelN1}. For the $\frac{1}{\varepsilon}$-order we get 
	\begin{align*}
	&\frac{1}{\varepsilon}
	\left[-\partial_Z^2+2\cos\alpha\partial_\rho\partial_Z-\partial_\rho^2+f''(v_\alpha)\right]\hat{v}^{C\pm}_1
	+\frac{1}{\varepsilon}q^\pm(0,t)f^{(3)}(v_\alpha)\hat{u}^{C\pm}_{1,\alpha}\partial_\rho v_\alpha\\
	&-\frac{1}{\varepsilon}q^\pm(z,t)\partial_\rho^2v_\alpha\Delta r|_{\overline{X}_0(s,t)}
	-\frac{1}{\varepsilon}q^\pm(0,t)\partial_Z\partial_\rho v_\alpha\Delta z_\alpha^\pm|_{\overline{p}^\pm(t)}\\
	&-\frac{1}{\varepsilon^2}q^\pm(0,t)\partial_Z^2\partial_\rho v_\alpha\left[\partial_r(|\nabla z_\alpha^\pm|^2\circ\overline{X})|_{(0,\pm 1,t)}\varepsilon(\rho+h_{1,\alpha}|_{(\pm 1,t)})\phantom{\frac{1}{\sin\alpha}}\right.\\
	&\left.\qquad\qquad\qquad\qquad\qquad+
	\partial_s(|\nabla z_\alpha^\pm|^2\circ\overline{X})|_{(0,\pm 1,t)}(\mp\varepsilon)\frac{1}{\sin\alpha}[Z+\cos\alpha(\rho+h_{1,\alpha}|_{(\pm 1,t)})]\right]\\
	&-2q^\pm(0,t)\partial_Z\partial_\rho^2v_\alpha\left[\partial_r((\nabla r\cdot\nabla z_\alpha^\pm)\circ\overline{X})|_{(0,\pm 1,t)}\varepsilon(\rho+h_{1,\alpha}|_{(\pm 1,t)})\phantom{\frac{1}{\sin\alpha}}\right.\\
	&\left.
	+\partial_s((\nabla r\cdot\nabla z_\alpha^\pm)\circ\overline{X})|_{(0,\pm 1,t)}(\mp\varepsilon)\frac{1}{\sin\alpha}[Z+\cos\alpha(\rho+h_{1,\alpha}|_{(\pm 1,t)})]-\frac{\mp\sin\alpha}{\varepsilon}\partial_sh_{1,\alpha}|_{(\pm 1,t)}\right]\\
	&+q^\pm(z,t)\partial_\rho^3v_\alpha\left[\partial_r(|\nabla r|^2\circ\overline{X})|_{(0,s,t)}\varepsilon(\rho+h_1|_{(s,t)})-\nabla r\cdot\nabla s|_{\overline{X}_0(s,t)}\partial_sh_{1,\alpha}|_{(s,t)}\right]\\
	&-\frac{2}{\varepsilon}\partial_zq^\pm(0,t)\partial_Z\partial_\rho v_\alpha
	-\frac{2}{\varepsilon}\partial_zq^\pm(z,t)\partial_\rho^2v_\alpha(\nabla r\cdot\nabla z_\alpha^\pm)|_{\overline{X}_0(s,t)}.
	\end{align*} 
	Here the penultimate line vanishes due to Theorem \ref{th_coord2D}. Moreover, $(\nabla r\cdot\nabla z_\alpha^\pm)|_{\overline{X}_0(s,t)}=-\cos\alpha$. We leave the two $\partial_\rho^2v_\alpha$-terms as remainders. Later we can improve the $\varepsilon$-order of these terms in one situation due to $\int_\R(\partial_\rho^2v_\alpha\partial_\rho v_\alpha)|_{(\rho,Z)}\,d\rho=0$ for $Z\geq 0$, see the estimate of $(II)$ in the proof of Lemma \ref{th_SE_ACalpha_VxV} below. Moreover, we require the other terms to add up to zero. This gives the following equation for $\hat{v}^{C\pm}_1$ on $\overline{\R^2_+}\times[0,T]$: \phantom{\qedhere}
	\begin{align*}
	&
	\left[-\partial_Z^2+2\cos\alpha\partial_\rho\partial_Z-\partial_\rho^2+f''(v_\alpha)\right]\hat{v}^{C\pm}_1\\
	&=-q^\pm(0,t)f^{(3)}(v_\alpha)\hat{u}^{C\pm}_{1,\alpha}\partial_\rho v_\alpha+\partial_Z\partial_\rho v_\alpha\left[q^\pm(0,t)\Delta z_\alpha^\pm|_{\overline{p}^\pm(t)}+2\partial_zq^\pm(0,t)\right]\\
	&+q^\pm(0,t)\partial_Z^2\partial_\rho v_\alpha\left[\partial_r(|\nabla z_\alpha^\pm|^2\circ\overline{X})|_{(0,\pm 1,t)}(\rho+h_{1,\alpha}|_{(\pm 1,t)})\phantom{\frac{1}{\sin\alpha}}\right.\\
	&\left.\qquad\qquad\qquad\qquad\qquad+
	\partial_s(|\nabla z_\alpha^\pm|^2\circ\overline{X})|_{(0,\pm 1,t)}\frac{\mp 1}{\sin\alpha}[Z+\cos\alpha(\rho+h_{1,\alpha}|_{(\pm 1,t)})]\right]\\
	&+2q^\pm(0,t)\partial_Z\partial_\rho^2v_\alpha\left[\partial_r((\nabla r\cdot\nabla z_\alpha^\pm)\circ\overline{X})|_{(0,\pm 1,t)}(\rho+h_{1,\alpha}|_{(\pm 1,t)})\phantom{\frac{1}{\sin\alpha}}\right.\\
	&\left.
	+\partial_s((\nabla r\cdot\nabla z_\alpha^\pm)\circ\overline{X})|_{(0,\pm 1,t)}\frac{\mp 1}{\sin\alpha}[Z+\cos\alpha(\rho+h_{1,\alpha}|_{(\pm 1,t)})]\pm\sin\alpha\partial_sh_{1,\alpha}|_{(\pm 1,t)}\right].
	\end{align*}
	\end{proof}

\begin{proof}[Asymptotic Expansion of $\sqrt{\varepsilon}\Nc_{\varepsilon,t}^\pm\phi^A_{\varepsilon,\alpha}(.,t)$ on $\partial\Omega\cap\partial\Omega^{C\pm}_t$.] In $\overline{\Omega^{C\pm}_t}$ it holds
	\begin{align}\begin{split}\label{eq_SE_ACalpha_nabla_phiA}
	\sqrt{\varepsilon}\nabla\phi^A_{\varepsilon,\alpha}&=
	\partial_zq^\pm\partial_\rho v_\alpha
	\nabla z_\alpha^\pm
	+(q^\pm\partial_Z\partial_\rho v_\alpha+\varepsilon\partial_Z\hat{v}^{C\pm}_1)\frac{\nabla z_\alpha^\pm}{\varepsilon}\\
	&+(q^\pm\partial_\rho^2 v_\alpha+\varepsilon\partial_\rho\hat{v}^{C\pm}_1)\left[\frac{\nabla r}{\varepsilon}-\nabla s\partial_sh_{\varepsilon,\alpha}\right]\end{split}
	\end{align}
	with evaluations as in Lemma \ref{th_asym_ACalpha_cp_trafo} except that the $q^\pm$-terms are evaluated at $(z_\alpha^\pm(x,t),t)$. In $\sqrt{\varepsilon} N_{\partial\Omega}\cdot\nabla\phi^A_{\varepsilon,\alpha}$ the $q^\pm$-terms are evaluated at $z=0$. Moreover, we expand the $(x,t)$-terms via \eqref{eq_asym_ACalpha_cp_taylor4} and insert $r=\varepsilon(\rho+h_{\varepsilon,\alpha}(s,t))$. Note that there are no $h_{\varepsilon,\alpha}$-terms in the lowest order and we only have to expand up to $\Oc(\varepsilon^0)$. Therefore we use $\partial_s^kh_{\varepsilon,\alpha}|_{(s,t)}=\partial_s^kh_{\varepsilon,\alpha}|_{(\pm 1,t)}+\Oc(|s\mp 1|)$ for $k=0,1$ and replace $s\mp 1$ by \eqref{eq_asym_ACalpha_cp_s2} with $Z=0$.
	
	At the lowest order $\Oc(\frac{1}{\varepsilon})$ in $\sqrt{\varepsilon}\Nc_{\varepsilon,t}^\pm\phi^A_{\varepsilon,\alpha}(.,t)$ we obtain 
	\[
	\frac{1}{\varepsilon}q^\pm(0,t)\left[-\partial_Z+\cos\alpha\partial_\rho +\sigma_\alpha''(v_\alpha|_{Z=0})\right]
	\partial_\rho v_\alpha|_{Z=0}
	\]
	due to
	$N_{\partial\Omega}=-\nabla z_\alpha^\pm|_{\overline{p}^\pm(t)}$ and $N_{\partial\Omega}\cdot\nabla r|_{\overline{p}^\pm(t)}=\cos\alpha$, cf.~Section \ref{sec_asym_ACalpha_cp_robin_m1}. This is zero because of \eqref{eq_hp_alpha_modelN2}. The $\Oc(1)$-order equals\phantom{\qedhere}
	\begin{align*}
	&\left[-\partial_Z+\cos\alpha\partial_\rho +\sigma_\alpha''(v_\alpha|_{Z=0})\right]\hat{v}^{C\pm}_1|_{Z=0}+q^\pm(0,t)\sigma_\alpha'''(v_\alpha|_{Z=0})u^{C\pm}_{1,\alpha}|_{Z=0}\partial_\rho v_\alpha|_{Z=0}\\
	&-\partial_zq^\pm(0,t)\partial_\rho v_\alpha|_{Z=0}+q^\pm(0,t)\partial_Z\partial_\rho v_\alpha|_{Z=0}\left[(\rho+h_{1,\alpha}|_{(\pm 1,t)})\partial_r((N_{\partial\Omega}\cdot\nabla z_\alpha^\pm)\circ\overline{X}^\pm_1)|_{(0,t)}\right]\\
	&+q^\pm(0,t)\partial_\rho^2v_\alpha|_{Z=0}\left[(\rho+h_{1,\alpha}|_{(\pm 1,t)})\partial_r((N_{\partial\Omega}\cdot\nabla r)\circ\overline{X}^\pm_1)|_{(0,t)}\mp\sin\alpha\partial_sh_{1,\alpha}|_{(\pm 1,t)}\right],
	\end{align*}
	where $\overline{X}^\pm_1$ is defined as in \eqref{eq_asym_ACalpha_cp_sbar}.
	We require that this term vanishes. This yields a boundary condition for $\hat{v}^{C\pm}_1$ on $\partial\R^2_+\times[0,T]$. 
	
	Together with the equation derived in the asymptotic expansion of $\sqrt{\varepsilon}\Lc_{\varepsilon,t}^\pm\phi^A_{\varepsilon,\alpha}(.,t)$ we obtain equations of type \eqref{eq_hp_alpha_modelL1}-\eqref{eq_hp_alpha_modelL2} with the additional parameter $t\in[0,T]$ for $\hat{v}^{C\pm}_1$. Because of  $\alpha\in\frac{\pi}{2}+[-\alpha_0,\alpha_0]$, we have solution theorems due to Remark \ref{th_asym_ACalpha_decay_param} and Theorem \ref{th_hp_alpha_sol_reg3}. Note that the corresponding right hand sides are contained in $B([0,T];H^2_{(\beta,\frac{\gamma_0}{2})}(\R^2_+)\times H^{5/2}_{(\beta)}(\R))$ for some $\beta>0$ provided that $q^\pm\in B([0,T],C^1([0,\tilde{\mu}_0]))$. Hence under this condition on $q^\pm$ we obtain a unique solution $\hat{v}^{C\pm}_1\in B([0,T];H^4_{(\beta,\frac{\gamma_0}{2})}(\R^2_+))\hookrightarrow B([0,T];C^2_{(\beta,\frac{\gamma_0}{2})}(\overline{\R^2_+}))$ for some possibly smaller $\beta>0$ if and only if \eqref{eq_hp_alpha_lin_comp} holds for the associated right hand sides. The latter is equivalent to an equation for $q^\pm$ only involving $q^\pm(0,t)$ and $\partial_zq^\pm(0,t)$ linearly. Moreover, note that the only terms where $\partial_zq^\pm(0,t)$ enters are 
	\[
	\partial_zq^\pm(0,t)\left[2\int_{\R^2_+}\partial_Z\partial_\rho v_\alpha\partial_\rho v_\alpha\,d(\rho,Z)+\int_\R(\partial_\rho v_\alpha)^2|_{Z=0}\,d\rho\right].
	\]
	Because of the estimates in Remark \ref{th_hp_alpha_solN3_rem} it follows that $\partial_zq^\pm(0,t)$ is a determined bounded function on $[0,T]$ if for example $q^\pm(0,t)=1$ for $t\in[0,T]$. Therefore with a simple ansatz and cutoff we can construct $q^\pm\in B([0,T],C^2([0,2\mu_0]))$ such that \eqref{eq_hp_alpha_lin_comp} holds for this situation as well as $q^\pm(0,t)=1, q^\pm(.,t)=1$ on $[\hat{\mu}_0,\tilde{\mu}_0]$ for all $t\in[0,T]$ and $\frac{1}{2}\leq q^\pm\leq 2$. \end{proof}

\begin{Lemma}\label{th_SE_ACalpha_phi_A}
	The mapping $\phi^A_{\varepsilon,\alpha}(.,t)$ is $C^2(\overline{\Omega_t^{C\pm}})$ and fulfills uniformly in $t\in[0,T]$:
	\begin{alignat*}{2}
	\left|\sqrt{\varepsilon}\Lc_{\varepsilon,t}^\pm\phi^A_{\varepsilon,\alpha}(.,t)+\frac{1}{\varepsilon}\tilde{q}^\pm(.,t)\right|&\leq Ce^{-c|\rho_{\varepsilon,\alpha}(.,t)|}&\quad\!&\text{ in }\Omega_t^{C\pm},\\
	\left|\sqrt{\varepsilon}\Nc_{\varepsilon,t}^\pm\phi^A_{\varepsilon,\alpha}(.,t)\right|&\leq C\varepsilon e^{-c|\rho_{\varepsilon,\alpha}(.,t)|}&\quad\!&\text{ on }\partial\Omega_t^{C\pm}\cap\partial\Omega,\\
	\left|\sqrt{\varepsilon}\Nc_{\varepsilon,t}^\pm\phi^A_{\varepsilon,\alpha}(.,t)\right|&\leq Ce^{-c/\varepsilon}&\quad\!&\text{ on }\partial\Omega_t^{C\pm}\setminus\Gamma_t(\frac{\delta_0}{2}),
	\end{alignat*}
	where we have set
	\[
	\tilde{q}^\pm(.,t):=[\Delta r|_{\overline{X}_0(s(.,t),t)}q^\pm|_{(z_\alpha^\pm(.,t),t)}
	-2\cos\alpha\partial_zq^\pm|_{(z_\alpha^\pm(.,t),t)}]\partial_\rho^2 v_\alpha|_{(\rho_{\varepsilon,\alpha}(.,t),Z_{\varepsilon,\alpha}^\pm(.,t))}.
	\]
\end{Lemma}

\begin{proof}
	The assertions follow from the construction and rigorous remainder estimates for the expansions above. Note that no $Z$-terms are multiplied with $\partial_\rho^2v_\alpha$, $\partial_\rho^3 v_\alpha$.
\end{proof}

\subsection{Notation for Transformations and the Splitting}\label{sec_SE_ACalpha_splitting} 
We introduce the notation
\[
X^\pm:[-\delta_0,\delta_0]\times[0,\tilde{\mu}_0]\times[0,T]\rightarrow\bigcup_{t\in[0,T]}\overline{\Omega_t^{C\pm}}:(r,z,t)\mapsto X(r,\pm1\mp\frac{1}{\sin\alpha}[z+\cos\alpha\,r],t)
\]
and $\overline{X}^\pm:=(X^\pm,\textup{pr}_t)$. Here note that $(X^\pm(.,t))^{-1}=(r,z_\alpha^\pm)(.,t)$. Furthermore, we set $X_0^\pm:=X^\pm(0,.,.)$ and $\overline{X}_0^\pm:=\overline{X}^\pm(0,.,.)$. Moreover, let
\begin{align*}
J_t^\pm(r,z)&:=|\det D_{(r,z)}X^\pm(r,z,t)|=J_t\left(r,\pm1\mp\frac{1}{\sin\alpha}[z+\cos\alpha\,r]\right)\frac{1}{\sin\alpha},\\
\tilde{h}_{j,\alpha}^\pm(r,z,t)&:=h_{j,\alpha}(\pm 1\mp\frac{1}{\sin\alpha}[z+\cos\alpha\,r],t)
\end{align*}
for  $(r,z,t)\in[-\delta_0,\delta_0]\times[0,\tilde{\mu}_0]\times[0,T]$ and $j=1,2$ as well as $\tilde{h}_{\varepsilon,\alpha}^\pm:=\tilde{h}_{1,\alpha}^\pm+\varepsilon \tilde{h}_{2,\alpha}^\pm$. Integrals over $\Omega_t^{C\pm}$ can be transformed to $(-\delta_0,\delta_0)\times(0,\tilde{\mu}_0)$ via $X^\pm(.,t)$ for $t\in[0,T]$, where the determinant factor is given by $J_t^\pm$. Hereby $\rho_{\varepsilon,\alpha}(.,t)$ transforms to 
\[
\rho_{\varepsilon,\alpha}|_{\overline{X}^\pm(r,z,t)}=\frac{r-\varepsilon\tilde{h}_{\varepsilon,\alpha}^\pm(r,z,t)}{\varepsilon}.
\] 
After applying the Fubini Theorem we can use the results from the 1D-preliminaries in \cite{MoserACvACND}, Section 6.1, for fixed $z$. 
We set $r_{\varepsilon,z,t}^\pm:[-\delta_0,\delta_0]\rightarrow\R:r\mapsto r-\varepsilon\tilde{h}_{\varepsilon,\alpha}^\pm(r,z,t)$ for all $z\in[0,\tilde{\mu}_0]$ and $t\in[0,T]$. Then due to \cite{MoserACvACND}, Section 6.1.1, the map
\[
F_{\varepsilon,z,t}^\pm:\frac{1}{\varepsilon}r_{\varepsilon,z,t}^\pm([-\delta_0,\delta_0])\rightarrow[-\delta_0,\delta_0]:\rho\mapsto (r_{\varepsilon,z,t}^\pm)^{-1}(\varepsilon\rho)
\]
is well-defined for all $z$, $t$ as above if $\varepsilon\in(0,\varepsilon_0]$ for some $\varepsilon_0>0$ independent of $z$, $t$. Finally, we set $\tilde{J}_{\varepsilon,z,t}^\pm:=J_t^\pm(F_{\varepsilon,z,t}^\pm(.),z)$ for $z\in[0,\tilde{\mu}_0]$ and $t\in[0,T]$. 

Now we characterize the splitting of $\tilde{H}^1(\Omega^{C\pm}_t)$. 
\begin{Lemma}\phantomsection{\label{th_SE_ACalpha_split_L2}}
	Let $\tilde{H}^1(\Omega^{C\pm}_t)$, $\hat{V}_{\varepsilon,t}^\pm$ and $\hat{H}^1(0,\tilde{\mu}_0)$ be as in \eqref{eq_SE_ACalpha_H1tilde_Omega}-\eqref{eq_SE_ACalpha_H1tilde_interval}. Then
	\begin{enumerate}
		\item $\hat{V}_{\varepsilon,t}^\pm$ is a subspace of $\tilde{H}^1(\Omega_t^{C\pm})$ and for $\varepsilon_0>0$ small there are $c_1,C_1>0$ such that 
		\[
		c_1\|a\|_{L^2(0,\tilde{\mu}_0)}\leq\|\phi\|_{L^2(\Omega_t^{C\pm})}\leq C_1\|a\|_{L^2(0,\tilde{\mu}_0)}
		\]
		for all $\phi=a(z_\alpha^\pm(.,t))\phi^A_{\varepsilon,\alpha}(.,t)\in \hat{V}_{\varepsilon,t}^\pm$ and $\varepsilon\in(0,\varepsilon_0]$, $t\in[0,T]$.
		\item Let $(\hat{V}_{\varepsilon,t}^\pm)^\perp$ be the $L^2$-orthogonal complement of $\hat{V}_{\varepsilon,t}^\pm$ in $\tilde{H}^1(\Omega_t^{C\pm})$. Then for $\psi\in\tilde{H}^1(\Omega_t^{C\pm})$:
		\[
		\psi\in(\hat{V}_{\varepsilon,t}^\pm)^\perp\quad\Leftrightarrow\quad\int_{-\delta_0}^{\delta_0}(\phi^A_{\varepsilon,\alpha}(.,t)\psi)|_{X^\pm(r,z,t)}J_t^\pm(r,z)\,dr=0\quad\text{ for a.e.~}z\in(0,\tilde{\mu}_0).
		\]
		Additionally, $\tilde{H}^1(\Omega_t^{C\pm})=\hat{V}_{\varepsilon,t}^\pm\oplus (\hat{V}_{\varepsilon,t}^\pm)^\perp$ for every $\varepsilon\in(0,\varepsilon_0]$ if $\varepsilon_0>0$ is small.
	\end{enumerate}
\end{Lemma}
\begin{proof}[Proof. Ad 1] 
	Analogously to the case $\alpha=\frac{\pi}{2}$ it follows that $a(z_\alpha^\pm(.,t))\in H^1(\Omega_t^{C\pm})$ for all $a\in H^1(0,\tilde{\mu}_0)$, cf.~the proof of \cite{AbelsMoser}, Lemma 4.6. Therefore $\hat{V}_{\varepsilon,t}^\pm$ is a subspace of $\tilde{H}^1(\Omega_t^{C\pm})$. Now we show the norm equivalence for $\varepsilon\in(0,\varepsilon_0]$ and $\varepsilon_0>0$ small. To this end we consider $\psi=a(z_\alpha^\pm(.,t))\phi^A_{\varepsilon,\alpha}(.,t)\in \hat{V}_{\varepsilon,t}^\pm$. Then the transformation rule and Fubini's Theorem imply
	\begin{align}\label{eq_SE_ACalpha_split_L2}
	\|\psi\|_{L^2(\Omega_t^{C\pm})}^2=\int_0^{\tilde{\mu}_0}a(z)^2\int_{-\delta_0}^{\delta_0}(\phi^A_{\varepsilon,\alpha}|_{\overline{X}^\pm(r,z,t)})^2 J_t^\pm(r,z)\,dr\,dz.
	\end{align}
	The leading order term with respect to $\varepsilon$ in the inner integral is $\frac{1}{\varepsilon}q^\pm(z,t)^2$ times
	\begin{align*}
	\int_{-\delta_0}^{\delta_0}(\partial_\rho v_\alpha)^2|_{(\rho_{\varepsilon,\alpha}|_{\overline{X}^\pm(r,z,t)},\frac{z}{\varepsilon})} J_t^\pm|_{(r,z)}\,dr=\int_{r_{\varepsilon,z,t}([-\delta_0,\delta_0])/\varepsilon}(\frac{d}{d\rho} F_{\varepsilon,z,t}^\pm)|_{\rho}(\partial_\rho v_\alpha)^2|_{(\rho,\frac{z}{\varepsilon})} \tilde{J}_{\varepsilon,z,t}^\pm|_{\rho}\,d\rho,
	\end{align*}
	where we used the transformation in \cite{MoserACvACND}, Lemma 6.5,~1. Because of \cite{MoserACvACND}, Remark 6.4, 2., the decay of $\partial_\rho v_\alpha$, the estimate $0<\frac{d}{d\rho}F_{\varepsilon,z,t}^\pm=\varepsilon\Oc(1)$ due to \cite{MoserACvACND}, Lemma 6.5, Remark \ref{th_hp_alpha_solN3_rem} and $c\leq J,q^\pm\leq C$ for some $c,C>0$, it follows that the above integral can be estimated from above and below by constants $\tilde{c},\tilde{C}>0$ independent of $t\in[0,T]$, $\varepsilon\in(0,\varepsilon_0]$ provided that $\varepsilon_0=\varepsilon_0(\overline{C}_0)>0$ is small. For the remainder in the inner integral in \eqref{eq_SE_ACalpha_split_L2} we use \cite{MoserACvACND}, Lemma 6.5, and obtain an estimate of the absolute value to $C\varepsilon$. For $\varepsilon_0>0$ small this shows the claim.\qedhere$_{1.}$\end{proof}

\begin{proof}[Ad 2] 
	Let $t\in[0,T]$ be fixed. By definition it holds
	\[
	(\hat{V}_{\varepsilon,t}^\pm)^\perp=\left\{\psi\in \tilde{H}^1(\Omega_t^{C\pm}):\int_{\Omega_t^{C\pm}}\psi a(z_\alpha^\pm(.,t))\phi^A_{\varepsilon,\alpha}(.,t)\,dx=0\text{ for all }a\in\hat{H}^1(0,\tilde{\mu}_0)\right\}.
	\]
	The integral equals $\int_0^{\tilde{\mu}_0}a(z)\int_{-\delta_0}^{\delta_0}(\phi^A_{\varepsilon,\alpha}(.,t)\psi)|_{X^\pm(r,z,t)}J_t^\pm(r,z)\,dr\,dz$. Hence the Fundamental Theorem of Calculus of Variations yields the characterization. Moreover, by definition it holds $\hat{V}_{\varepsilon,t}^\pm\cap(\hat{V}_{\varepsilon,t}^\pm)^\perp=\{0\}$. It is left to show $\hat{V}_{\varepsilon,t}^\pm+(\hat{V}_{\varepsilon,t}^\pm)^\perp=\tilde{H}^1(\Omega_t^{C\pm})$. Due to the proof of the first part this follows analogously as in the case $\alpha=\frac{\pi}{2}$, cf.~the proof of \cite{AbelsMoser}, Lemma 4.6., 2.\qedhere$_{2.}$\end{proof}

\subsection{Analysis of the Bilinear Form}\label{sec_SE_ACalpha_BLF}
First we consider $B_{\varepsilon,t}^\pm$ on $\hat{V}_{\varepsilon,t}^\pm\times \hat{V}_{\varepsilon,t}^\pm$.
\begin{Lemma}\label{th_SE_ACalpha_VxV}
	There are $\varepsilon_0,C>0$ such that 
	\[
	B_{\varepsilon,t}^\pm(\phi,\phi)\geq-C\|\phi\|_{L^2(\Omega_t^{C\pm})}^2+\overline{c}\|a\|_{H^1(0,\tilde{\mu}_0)}^2,\quad \overline{c}:=\frac{1}{2}\|\theta_0'\|_{L^2(\R)}^2
	\]
	for all $\phi=a(z_\alpha^\pm(.,t))\phi^A_{\varepsilon,\alpha}(.,t)\in \hat{V}_{\varepsilon,t}^\pm$ and $\varepsilon\in(0,\varepsilon_0],t\in[0,T]$. 
\end{Lemma}
\begin{proof} 
	Consider $\phi$ as in the lemma. With the analogous computation as in the case $\alpha=\frac{\pi}{2}$, cf.~the proof of \cite{AbelsMoser}, Lemma 4.7, it follows that\phantom{\qedhere}
	\begin{align*}
	B_{\varepsilon,t}^\pm(\phi,\phi)&=\int_{\Omega_t^{C\pm}}|\nabla(a(z_\alpha^\pm))\phi^A_{\varepsilon,\alpha}|^2|_{(.,t)}\,dx+\int_{\Omega_t^{C\pm}}(a^2(z_\alpha^\pm)\phi^A_{\varepsilon,\alpha})|_{(.,t)}\Lc_{\varepsilon,t}^\pm\phi^A_{\varepsilon,\alpha}|_{(.,t)}\,dx\\
	&+\int_{\partial\Omega_t^{C\pm}}\left[\Nc_{\varepsilon,t}^\pm\phi^A_{\varepsilon,\alpha}|_{(.,t)}\,\tr(a^2(z_\alpha^\pm)\phi^A_{\varepsilon,\alpha}|_{(.,t)})\right]\,d\Hc^1=:(I)+(II)+(III).
	\end{align*}
	\begin{proof}[Ad $(I)$] 
		It holds $|\nabla(a(z_\alpha^\pm(.,t)))|^2=\left[|\nabla z_\alpha^\pm|^2(a')^2(z_\alpha^\pm)\right]|_{(.,t)}$ and therefore\phantom{\qedhere}
		\[
		(I)=\int_0^{\tilde{\mu}_0}(a')^2(z)\int_{-\delta_0}^{\delta_0}\left[|\nabla z_\alpha^\pm|^2(\phi^A_{\varepsilon,\alpha})^2\right]|_{\overline{X}^\pm(r,z,t)}J_t^\pm(r,z)\,dr\,dz.
		\]
		Note that $|\nabla z_\alpha^\pm|^2|_{\overline{X}^\pm(0,z,t)}=1$ and $J_t^\pm(0,t)=1$ due to Remark \ref{th_coord2D_rem1}, Theorem \ref{th_coord2D}, Remark \ref{th_coord2D_rem2},~3.~and \eqref{eq_asym_ACalpha_cp_zalpha}. Therefore the Taylor Theorem yields that these terms are $1+\Oc(|r|)$. Moreover, Remark \ref{th_hp_alpha_solN3_rem} yields $\int_\R(\partial_\rho v_\alpha)^2(\rho,Z)\,d\rho\geq\frac{3}{4}\|\theta_0'\|_{L^2(\R)}^2$ for $Z\geq 0$. With transformations and remainder estimates in \cite{MoserACvACND}, Lemma 6.5, exponential decay estimates as well as \cite{MoserACvACND}, Remark 6.4,~2.~it follows that the inner integral in $(I)$ is estimated from below by $\frac{2}{3}\|\theta_0'\|_{L^2(\R)}^2$ for all $\varepsilon\in(0,\varepsilon_0], t\in[0,T]$, if $\varepsilon_0>0$ is small.\end{proof}
	
	\begin{proof}[Ad $(II)$] It holds\phantom{\qedhere}
		\[
		(II)=\int_0^{\tilde{\mu}_0}a^2(z)\int_{-\delta_0}^{\delta_0} \phi^A_{\varepsilon,\alpha}|_{\overline{X}^\pm(r,z,t)}(\Lc_{\varepsilon,t}^\pm\phi^A_{\varepsilon,\alpha}(.,t))|_{X^\pm(r,z,t)}J_t^\pm(r,z)\,dr\,dz.
		\]
		We estimate the inner integral. Lemma \ref{th_SE_ACalpha_phi_A} implies
		\begin{align*}
		&\left|\sqrt{\varepsilon}\Lc_{\varepsilon,t}^\pm\phi^A_{\varepsilon,\alpha}(.,t))|_{X^\pm(r,z,t)}+
		\frac{1}{\varepsilon}
		[\Delta r|_{\overline{X}_0^\pm(z,t)}q^\pm|_{(z,t)}
		-2\cos\alpha\partial_zq^\pm|_{(z,t)}]\partial_\rho^2 v_\alpha(\rho_{\varepsilon,\alpha}|_{\overline{X}^\pm(r,z,t)},\frac{z}{\varepsilon})\right|\\
		&\leq Ce^{-c|\rho_{\varepsilon,\alpha}(\overline{X}^\pm(r,z,t))|}\quad\text{ for }(r,z)\in[-\delta_0,\delta_0]\times[0,\tilde{\mu}_0].
		\end{align*}
		Using \cite{MoserACvACND}, Lemma 6.5, $\int_\R(\partial_\rho^2v_\alpha\partial_\rho v_\alpha)(\rho,Z)\,d\rho=0$ for all $Z\geq 0$ due to integration by parts and $J_t(r,z)=J_t(0,z)+\Oc(|r|)$, we obtain $|(II)|\leq C\|a\|_{L^2(0,\tilde{\mu}_0)}^2$ with $C>0$ independent of $\phi\in \hat{V}_{\varepsilon,t}^\pm$ and all $\varepsilon\in(0,\varepsilon_0]$, $t\in[0,T]$ if $\varepsilon_0>0$ is small.\end{proof}
	
	\begin{proof}[Ad $(III)$] 
		The representation for line integrals and properties of the trace operator imply
		\begin{align*}
		(III)&=\sum_\pm\int_0^{\tilde{\mu}_0}a^2(z)\left[\phi^A_{\varepsilon,\alpha} \Nc_{\varepsilon,t}^\pm\phi^A_{\varepsilon,\alpha}\right]|_{\overline{X}^\pm(\pm\delta_0,z,t)}|\partial_z X^\pm(\pm\delta_0,z,t)|\,dz\\
		&+a^2(0)\int_{-\delta_0}^{\delta_0} \left[\phi^A_{\varepsilon,\alpha} \Nc_{\varepsilon,t}^\pm\phi^A_{\varepsilon,\alpha}\right]|_{\overline{X}^\pm(r,0,t)}|\partial_rX^\pm(r,0,t)|\,dr.
		\end{align*}
		Using Lemma \ref{th_SE_ACalpha_phi_A} and for the last integral \cite{MoserACvACND}, Lemma 6.5, we obtain
		\[
		|(III)|\leq Ce^{-c/\varepsilon}\|a\|_{L^2(0,\tilde{\mu}_0)}^2+C\varepsilon a^2(0).
		\]
		Due to $H^1(0,\tilde{\mu}_0)\hookrightarrow C_b^0([0,\tilde{\mu}_0])$, the claim follows with Lemma \ref{th_SE_ACalpha_split_L2},~1.\end{proof}
\end{proof}

Next we analyze $B^\pm_{\varepsilon,t}$ on $(\hat{V}_{\varepsilon,t}^\pm)^\perp\times(\hat{V}_{\varepsilon,t}^\pm)^\perp$. To this end we need the following auxiliary lemma:
\begin{Lemma}\label{th_SE_ACalpha_intpol_tr2}
	There is a $\overline{C}_2>0$ independent of $\Omega,\Gamma,\alpha\in\frac{\pi}{2}+[-\alpha_0,\alpha_0]$ and some $\tilde{\delta}_0>0$ with
	\begin{align*}
	\|\tr\,\psi\|_{L^2(\partial\Omega\cap\partial\Omega_{t,\delta}^{C\pm})}^2
	&\leq\overline{C}_2
	\left[\|\psi\|_{L^2(\Omega_{t,\delta}^{C\pm})}^2
	+\|\nabla_\tau \psi\|_{L^2(\Omega_{t,\delta}^{C\pm})}\|\psi\|_{L^2(\Omega_{t,\delta}^{C\pm})}\right],\\
	\|\nabla_\tau \psi\|_{L^2(\Omega_{t,\delta}^{C\pm})}&\leq 2 \|\nabla \psi\|_{L^2(\Omega_{t,\delta}^{C\pm})}
	\end{align*}
	for all $\psi\in H^1(\Omega_{t,\delta}^{C\pm})$ and $t\in[0,T]$, $\delta\in(0,\tilde{\delta}_0]$, where $\Omega_{t,\delta}^{C\pm}:=\Omega_t^{C\pm}\cap\Gamma_t(\delta)$.
\end{Lemma}
\begin{proof}
	The proof is analogous to the one of Lemma \ref{th_SE_ACalpha_intpol_tr1}, but we have to be careful in order to obtain constants independent of $\Omega,\Gamma$ and $\alpha\in\frac{\pi}{2}+[-\alpha_0,\alpha_0]$. Note that with $S$ from the proof of Lemma \ref{th_SE_ACalpha_intpol_tr1} the first estimate for $S\cap[(-\delta,\delta)\times\R]$, $S_{\delta,\alpha}^\pm$, $\partial_s$ instead of $\Omega_{t,\delta}^{C\pm}$, $\partial\Omega\cap\partial\Omega_{t,\delta}^{C\pm}$, $\nabla_\tau$ holds with a uniform constant independent of $\alpha\in\frac{\pi}{2}+[-\alpha_0,\alpha_0]$. This follows as in the proof of Lemma \ref{th_SE_ACalpha_intpol_tr1} since $\vec{w}$ can be chosen in a uniform way for all those $\alpha$. Moreover,
	\[
	\int_{\partial\Omega\cap\partial\Omega_{t,\delta}^{C\pm}}|\tr\psi|^2\,d\Hc^1
	=\int_{S_{\delta,\alpha}^\pm}|\tr\psi|^2|_{X(.,t)}|\det d_.(X(.,t)|_{S_{\delta,\alpha}^\pm})|\,d\Hc^1.
	\]
	Let $\gamma^\pm:(-\delta,\delta)\rightarrow S_{\delta,\alpha}^\pm:r\mapsto(r,s^\pm(r))$ with $s^\pm$ as in \eqref{eq_coord2D_trapeze_bdry}. It holds 
	\[
	\left|d_{\gamma^\pm(r)}[X(.,t)|_{S_{\delta,\alpha}^\pm}]\left(\frac{(\gamma^\pm)'(r)}{|(\gamma^\pm)'(r)|}\right)\right|=\frac{1}{|(\gamma^\pm)'(r)|}|\partial_rX^\pm(r,0,t)|
	\]
	and $|\partial_rX^\pm(r,0,t)|=|\partial_rX^\pm(0,0,t)|+\Oc(|r|)$, where
	\[
	\partial_rX^\pm(0,0,t)=D_{(r,s)}X|_{(0,t)}
	\begin{pmatrix}
	1\\ 
	\mp\cos\alpha/\sin\alpha
	\end{pmatrix}
	=\textup{Id}\cdot(\gamma^\pm)'(r)
	\]
	due to Remark \ref{th_coord2D_rem1} and Theorem \ref{th_coord2D}. This shows $|\det d_.(X(.,t)|_{S_{\delta,\alpha}^\pm})|\leq 1+C(\Gamma)\delta$. Additionally, integrals over $S\cap[(-\delta,\delta)\times\R]$ are transformed to $\Omega_{t,\delta}^{C\pm}$ via $X(.,t)$ with the determinant factor $J_t$, where $J_t(r,s)=1+\Oc(|\delta|)$ in $\Omega_{t,\delta}^{C\pm}$ because of Remark \ref{th_coord2D_rem1} and Remark \ref{th_coord2D_rem2},~3. Altogether we obtain the first estimate. For the second one we use $|\nabla s|^2=1+\Oc(|\delta|)$ in $\Omega_{t,\delta}^{C\pm}$ due to Remark \ref{th_coord2D_rem1} and
	\[
	|\nabla\psi|^2|_{X(.,t)}\geq (1-C(\Gamma)\delta)\partial_s(\psi|_{X(.,t)})^2\quad\text{ in }\Omega_{t,\delta}^{C\pm}.
	\]
	The latter follows from Theorem \ref{th_coord2D}, a Taylor expansion and Young's inequality. 
\end{proof}

\begin{Lemma}\label{th_SE_ACalpha_VpxVp}
	There are $\hat{\alpha}_0, \nu>0$ independent of $\Omega, \Gamma$ such that, if $\alpha\in\frac{\pi}{2}+[-\hat{\alpha}_0,\hat{\alpha}_0]$, then there is an $\varepsilon_0>0$ such that for all $\psi\in(\hat{V}_{\varepsilon,t}^\pm)^\perp$ and $\varepsilon\in(0,\varepsilon_0]$, $t\in[0,T]$ it holds
	\[
	B^\pm_{\varepsilon,t}(\psi,\psi)\geq
	\nu\left[\frac{1}{\varepsilon^2}\|\psi\|_{L^2(\Omega_t^{C\pm})}^2+\|\nabla\psi\|_{L^2(\Omega_t^{C\pm})}^2\right].
	\]  
\end{Lemma}
\begin{proof}
	First we prove that it is sufficient to show the existence of $\tilde{\alpha}_0,\tilde{\nu}>0$ independent of $\Omega,\Gamma,\alpha$ and the existence of some $\tilde{\varepsilon}_0>0$ such that if $\alpha\in\frac{\pi}{2}+[-\tilde{\alpha}_0,\tilde{\alpha}_0]$, then \phantom{\qedhere}
	\begin{align}\label{eq_SE_ACalpha_VpxVp_1}
	\tilde{B}^\pm_{\varepsilon,t}(\psi,\psi):=\int_{\Omega^{C\pm}_t}|\nabla\psi|^2+\frac{1}{\varepsilon^2}f''(\theta_0|_{\rho_{\varepsilon,\alpha}(.,t)})\psi^2\,dx\geq\frac{\tilde{\nu}}{\varepsilon^2}\|\psi\|_{L^2(\Omega^{C\pm}_t)}^2
	\end{align}
	for all $\psi\in(\hat{V}_{\varepsilon,t}^\pm)^\perp$ and $\varepsilon\in(0,\tilde{\varepsilon}_0]$, $t\in[0,T]$. In order to show with \eqref{eq_SE_ACalpha_VpxVp_1} the estimate in the lemma note that due to Remark \ref{th_hp_alpha_solN3_rem} and Definition \ref{th_ACalpha_sigma_def} there is a $\overline{C}>0$ independent of $\Omega,\Gamma$ and $\alpha\in\frac{\pi}{2}+[-\alpha_0,\alpha_0]$ such that
	\[	|f''(v_\alpha(\rho,Z))-f''(\theta_0(\rho))|\leq\overline{C}|\alpha-\frac{\pi}{2}|\text{ for all }(\rho,Z)\in\overline{\R^2_+}\quad\text{ and }\quad|\sigma_\alpha''|+|\sigma_\alpha'''|\leq\overline{C}|\alpha-\frac{\pi}{2}|.
	\]
	Moreover, to control the $\frac{1}{\varepsilon}\sigma_\alpha''$-term in $B_{\varepsilon,t}^\pm$ we use Lemma \ref{th_SE_ACalpha_intpol_tr2} and $\sigma_\alpha''(v_\alpha(\rho_{\varepsilon,\alpha}(.,t),0))=0$ in $\Omega_t^{C\pm}\setminus\Gamma(\tilde{\delta}_0)$ for $\varepsilon$ small because of Definition \ref{th_ACalpha_sigma_def} and Remark \ref{th_hp_alpha_solN3_rem}, where $\tilde{\delta}_0$ is as in Lemma \ref{th_SE_ACalpha_intpol_tr2}. For the $\sigma_\alpha'''$-term we use Lemma \ref{th_SE_ACalpha_intpol_tr1}. Therefore $|B_{\varepsilon,t}^\pm(\psi,\psi)
	-\tilde{B}_{\varepsilon,t}^\pm(\psi,\psi)|$ is estimated by
	\[
	\frac{\overline{C}|\alpha-\frac{\pi}{2}|+C\varepsilon}{\varepsilon^2}
	\|\psi\|_{L^2(\Omega^{C\pm}_t)}^2
	+\frac{2\overline{C}\,\overline{C}_2|\alpha-\frac{\pi}{2}|+C\varepsilon}{\varepsilon}\|\psi\|_{L^2(\Omega^{C\pm}_t)}\|\nabla\psi\|_{L^2(\Omega^{C\pm}_t)}
	\]
	for all $\psi\in(\hat{V}_{\varepsilon,t}^\pm)^\perp$ and $\varepsilon\in(0,\varepsilon_0]$, $t\in[0,T]$ if $\varepsilon_0>0$ is small. Let $\alpha\in\frac{\pi}{2}+[-\tilde{\alpha}_0,\tilde{\alpha}_0]$. Then for $\beta\in(0,1)$ it follows with Young's inequality that
	\begin{align*}
	B_{\varepsilon,t}^\pm(\psi,\psi)
	&\geq(1-\beta+\beta)\tilde{B}_{\varepsilon,t}^\pm(\psi,\psi)-|B_{\varepsilon,t}^\pm(\psi,\psi)
	-\tilde{B}_{\varepsilon,t}^\pm(\psi,\psi)|\\
	&\geq\frac{(1-\beta)\tilde{\nu}-\beta\sup_{\R}|f''(\theta_0)|-\overline{C}\,(\overline{C}_2+1)|\alpha-\frac{\pi}{2}|-C\varepsilon}{\varepsilon^2}\|\psi\|_{L^2(\Omega^{C\pm}_t)}^2\\
	&+(\beta-\overline{C}\,\overline{C}_2|\alpha-\frac{\pi}{2}|-C\varepsilon)\|\nabla\psi\|_{L^2(\Omega^{C\pm}_t)}^2
	\end{align*}
	for all $\psi\in(\hat{V}_{\varepsilon,t}^\pm)^\perp$ and $\varepsilon\in(0,\varepsilon_0]$, $t\in[0,T]$. We choose $\beta:=\frac{1}{4}\min\{1,\tilde{\nu}/\sup_{\R}|f''(\theta_0)|\}$ and then $\hat{\alpha}_0>0$ small such that 
	\[
	\frac{\tilde{\nu}}{2}-\overline{C}\,(\overline{C}_2+1)\hat{\alpha}_0\geq\frac{\tilde{\nu}}{4}\quad\text{ and }\quad \beta-\overline{C}\,\overline{C}_2\hat{\alpha}_0\geq\frac{\beta}{2}.
	\]
	Therefore the claim follows with $\nu:=\min\{\frac{\tilde{\nu}}{8},\frac{\beta}{4}\}$ provided that $\varepsilon_0>0$ is small.
	
	In the following we prove \eqref{eq_SE_ACalpha_VpxVp_1} with similar ideas as in the case $\alpha=\frac{\pi}{2}$, cf.~the proof of \cite{AbelsMoser}, Lemma 4.8. Let $\tilde{\psi}_t^\pm:=\psi|_{X^\pm(.,t)}$ for $\psi\in(\hat{V}_{\varepsilon,t}^\pm)^\perp$. Because of the chain rule we obtain
	$\nabla\psi|_{X^\pm(.,t)}=\nabla r|_{\overline{X}^\pm(.,t)}\partial_r\tilde{\psi}_t^\pm+\nabla z_\alpha^\pm|_{\overline{X}^\pm(.,t)}\partial_z\tilde{\psi}_t^\pm$ and therefore
	\[
	|\nabla\psi|^2|_{X^\pm(.,t)}=(\nabla_{(r,z)}\tilde{\psi}_t^\pm)^\top
	\begin{pmatrix}
	|\nabla r|^2 & \nabla r\cdot\nabla z_\alpha^\pm\\
	\nabla r\cdot\nabla z_\alpha^\pm & |\nabla z_\alpha^\pm|^2
	\end{pmatrix}|_{\overline{X}^\pm(.,t)}
	\nabla_{(r,z)}\tilde{\psi}_t^\pm,
	\]
	where $|\nabla r|^2=1+\Oc(|r|^2)$, $|\nabla r\cdot\nabla z_\alpha^\pm|=|\cos\alpha|+\Oc(|r|)$ and $|\nabla s|^2=1+\Oc(|r|)$ due to Remark \ref{th_coord2D_rem1},
	Theorem \ref{th_coord2D} and Taylor's Theorem. Therefore Young's inequality yields
	\begin{align}\label{eq_SE_ACalpha_VpxVp_2}
	|\nabla\psi|^2|_{X^\pm(.,t)}\geq
	(1-\overline{C}_3|\alpha-\frac{\pi}{2}|-C|r|)\left[(\partial_r\tilde{\psi}_t^\pm)^2 + (\partial_z\tilde{\psi}_t^\pm)^2\right]
	\end{align}
	with $\overline{C}_3>0$ independent of $\Omega,\Gamma$ and $\alpha\in\frac{\pi}{2}+[-\alpha_0,\alpha_0]$. 
	To get $C|r|$ small enough (which will be precise later), we fix $\tilde{\delta}>0$ small and estimate separately for $r$ in
	\begin{align}\label{eq_SE_ACalpha_VpxVp_int_split}
	I_{z,t}^{\pm,\varepsilon}:=(r_{\varepsilon,z,t}^\pm)^{-1}[(-\tilde{\delta},\tilde{\delta})]\quad\text{ and }\quad \hat{I}_{z,t}^{\pm,\varepsilon}:=(-\delta_0,\delta_0)\setminus I_{z,t}^{\pm,\varepsilon}.
	\end{align}
	If $\varepsilon_0=\varepsilon_0(\tilde{\delta},\overline{C}_0)>0$ is small, then for all $\varepsilon\in(0,\varepsilon_0]$ and $z\in[0,\tilde{\mu}_0]$, $t\in[0,T]$ it holds 
	\[
	f''(\theta_0(\rho_{\varepsilon,\alpha}|_{\overline{X}^\pm(r,z,t)}))\geq c_0:=\frac{1}{2}\min\{f''(\pm 1)\}>0\quad \text{ for }r\in\hat{I}_{z,t}^{\pm,\varepsilon},\quad
	|r|\leq 2\tilde{\delta}\quad \text{ for }r\in I_{z,t}^{\pm,\varepsilon}
	\]
	where we used \cite{MoserACvACND}, Remark 6.4,~2.~for the first estimate and \cite{MoserACvACND}, Lemma 6.2,~1.~for the second one. With $\overline{C}_3,C$ as in \eqref{eq_SE_ACalpha_VpxVp_2} we define $\tilde{c}=\tilde{c}(\alpha,\tilde{\delta}):=\overline{C}_3|\alpha-\frac{\pi}{2}|+2C\tilde{\delta}$. For $\varepsilon\in(0,\varepsilon_0]$, $t\in[0,T]$ we get
	\begin{align*}
	\tilde{B}_{\varepsilon,t}^\pm(\psi,\psi)&\geq\int_0^{\tilde{\mu}_0}\int_{\hat{I}_{z,t}^{\pm,\varepsilon}}\frac{c_0}{\varepsilon^2}(\tilde{\psi}_t^\pm)^2 J_t^\pm|_{(r,z)}\,dr\,dz
	+\int_0^{\tilde{\mu}_0}\int_{I_{z,t}^{\pm,\varepsilon}}(1-\tilde{c})(\partial_z\tilde{\psi}_t^\pm)^2J_t^\pm|_{(r,z)}\,dr\,dz\\
	&+\int_0^{\tilde{\mu}_0}\int_{I_{z,t}^{\pm,\varepsilon}}\left[(1-\tilde{c})(\partial_r\tilde{\psi}_t^\pm)^2+\frac{1}{\varepsilon^2}f''(\theta_0(\rho_{\varepsilon,\alpha}|_{\overline{X}^\pm(.,t)}))(\tilde{\psi}_t^\pm)^2\right] J_t^\pm|_{(r,z)}\,dr\,dz.
	\end{align*}
	We use the notation from the beginning of Section \ref{sec_SE_ACalpha_splitting}. The transformation in \cite{MoserACvACND}, Lemma 6.5,~1.~yields that the inner integral in the second line equals $1/\varepsilon^2$ times
	\begin{align}\label{eq_SE_ACalpha_VpxVp_Bdef}
	B_{\varepsilon,z,t}^{\pm,\tilde{c}}(\Psi_{\varepsilon,z,t}^\pm,\Psi_{\varepsilon,z,t}^\pm)
	:=\int_{I_{\varepsilon,\tilde{\delta}}}\left[(1-\tilde{c})(\frac{d}{dz}\Psi_{\varepsilon,z,t}^\pm)^2+f''(\theta_0(z))(\Psi_{\varepsilon,z,t}^\pm)^2\right]\tilde{J}_{\varepsilon,z,t}^\pm\,dz,
	\end{align}
	where we set $I_{\varepsilon,\tilde{\delta}}:=(-\frac{\tilde{\delta}}{\varepsilon},\frac{\tilde{\delta}}{\varepsilon})$ 
	and $\Psi_{\varepsilon,z,t}^\pm:=\sqrt{\varepsilon}\tilde{\psi}_t^\pm(F_{\varepsilon,z,t}^\pm(.),z)$.
	Hence \eqref{eq_SE_ACalpha_VpxVp_1} follows if we show for $\tilde{\alpha}_0>0$ small independent of $\Omega,\Gamma$ and $\tilde{\delta}>0$ small, that $\tilde{c}\leq 1$ and with the $c_0$ from above
	\begin{align}\label{eq_SE_ACalpha_VpxVp_3}
	B_{\varepsilon,z,t}^{\pm,\tilde{c}}(\Psi_{\varepsilon,z,t}^\pm,\Psi_{\varepsilon,z,t}^\pm)
	\geq \overline{\nu}\|\Psi_{\varepsilon,z,t}^\pm\|^2_{L^2(I_{\varepsilon,\tilde{\delta}},\tilde{J}_{\varepsilon,z,t}^\pm)}-\frac{c_0}{2}\|\tilde{\psi}_t^\pm(.,z)\|^2_{L^2(\hat{I}_{z,t}^{\pm,\varepsilon},J_t^\pm(.,z))}
	\end{align}
	for $\varepsilon\in(0,\varepsilon_0]$, a.e.~$z\in[0,\tilde{\mu}_0]$ and all $t\in[0,T]$ with some $\varepsilon_0>0$ independent of $\varepsilon,z,t$ and $\overline{\nu}>0$ independent of $\Omega,\Gamma,\alpha,\tilde{\delta},\varepsilon_0,\varepsilon,z,t$ provided that $\alpha\in\frac{\pi}{2}+[-\tilde{\alpha}_0,\tilde{\alpha}_0]$. 
	
	Here $L^2(I_{\varepsilon,\tilde{\delta}},\tilde{J}_{\varepsilon,z,t}^\pm)$ is the space of $L^2$-functions on $I_{\varepsilon,\tilde{\delta}}$ with respect to the weight $\tilde{J}_{\varepsilon,z,t}^\pm$. We denote the scalar-product in $L^2(I_{\varepsilon,\tilde{\delta}},\tilde{J}_{\varepsilon,z,t}^\pm)$ by $(.,.)_{\varepsilon,z,t}$ and the norm with $\|.\|_{\varepsilon,z,t}$. For the proof of \eqref{eq_SE_ACalpha_VpxVp_3} we need properties of $B^{\pm,0}_{\varepsilon,z,t}$. The latter is defined as in \eqref{eq_SE_ACalpha_VpxVp_Bdef} with $\tilde{c}$ replaced by $0$. With respect to $(.,.)_{\varepsilon,z,t}$, $B_{\varepsilon,z,t}^{\pm,0}$ is the bilinear form associated to
	\[
	\Lc_{\varepsilon,z,t}^{\pm,0}:=-(\tilde{J}_{\varepsilon,z,t}^\pm)^{-1}\frac{d}{dz}\left(\tilde{J}_{\varepsilon,z,t}^\pm\frac{d}{dz}\right)+f''(\theta_0)
	\]
	on $H^2(I_{\varepsilon,\tilde{\delta}})$ with homogeneous Neumann boundary condition. In this situation we can apply the results in \cite{MoserACvACND}, Section 6.1.3.2 on spectral properties for perturbed 1D Allen-Cahn-type operators, in particular \cite{MoserACvACND}, Theorem 6.8.
	
	\begin{proof}[Proof of \eqref{eq_SE_ACalpha_VpxVp_3}] The integral characterization for $\psi\in (\hat{V}_{\varepsilon,t}^\pm)^\perp$ in Lemma \ref{th_SE_ACalpha_split_L2},~2.~yields
		\[
		\left|
		\int_{I_{z,t}^{\pm,\varepsilon}}(\phi^A_{\varepsilon,\alpha}(.,t)\psi)|_{X^\pm(r,z,t)} J_t^\pm(r,z)\,dr
		\right|
		\leq C(\tilde{\delta})e^{-c\tilde{\delta}/\varepsilon}
		\|\tilde{\psi}_t^\pm(.,z)\|_{L^2(\hat{I}_{z,t}^{\pm,\varepsilon},J_t^\pm(.,z))}
		\]	
		for $\varepsilon$ small, a.e.~$z\in[0,\tilde{\mu}_0]$ and all $t\in[0,T]$. The lowest order term in the integral is 
		\[
		\frac{1}{\sqrt{\varepsilon}}q^\pm(z,t)\int_{I_{z,t}^{\pm,\varepsilon}}\partial_\rho v_\alpha(\rho_{\varepsilon,\alpha}|_{\overline{X}^\pm(r,z,t)},\frac{z}{\varepsilon})(\tilde{\psi}_t^\pm J_t^\pm)(r,z)\,dr
		=q^\pm(z,t)(\Psi_{\varepsilon,z,t}^\pm,\partial_\rho v_\alpha(.,\frac{z}{\varepsilon}))_{\varepsilon,z,t}.
		\]
		The remaining term in the integral due to $\phi^A_{\varepsilon,\alpha}$ can be estimated with the Hölder inequality, the decay of $\hat{v}^{C\pm}_1$ and \cite{MoserACvACND}, Lemma 6.5, by $C\varepsilon\|\Psi_{\varepsilon,z,t}^\pm\|_{\varepsilon,z,t}$. Moreover, due to Remark \ref{th_hp_alpha_solN3_rem} it holds 
		\[
		|\partial_\rho v_\alpha(\rho,Z)-\theta_0'(\rho)|\leq\overline{C}_4|\alpha-\frac{\pi}{2}|e^{-\beta_0|\rho|}\quad\text{ for all }(\rho,Z)\in\overline{\R^2_+}
		\] 
		and some $\overline{C}_4>0$ independent of $\Omega,\Gamma,\alpha$. Together with 
		\[
		\tilde{J}_{\varepsilon,z,t}^\pm(\rho)=J_t(F_{\varepsilon,z,t}^\pm(\rho),z)=J_t(0,z)+\Oc(|F_{\varepsilon,z,t}^\pm(\rho)|)=1+\Oc(\varepsilon(|\rho|+C))
		\] 
		because of Remark \ref{th_coord2D_rem1}, Remark \ref{th_coord2D_rem2},~3.~and \cite{MoserACvACND}, Corollary 6.3, we obtain
		\[
		|(\Psi_{\varepsilon,z,t}^\pm,\partial_\rho v_\alpha(.,\frac{z}{\varepsilon}))_{\varepsilon,z,t}-(\Psi_{\varepsilon,z,t}^\pm,\theta_0')_{\varepsilon,z,t}|\leq(\overline{C}_4|\alpha-\frac{\pi}{2}|+C\varepsilon)\|\Psi_{\varepsilon,z,t}^\pm\|_{\varepsilon,z,t}
		\]
		for some $\overline{C}_4>0$ independent of $\Omega,\Gamma,\alpha$. Using $0<\frac{1}{2}\leq q^\pm\leq 2$ we obtain altogether
		\[
		|(\Psi_{\varepsilon,z,t}^\pm,\theta_0')_{\varepsilon,z,t}|\leq (\overline{C}_4|\alpha-\frac{\pi}{2}|+C\varepsilon)\|\Psi_{\varepsilon,z,t}^\pm\|_{\varepsilon,z,t}+C(\tilde{\delta})e^{-c\tilde{\delta}/\varepsilon}\|\tilde{\psi}_t^\pm(.,z)\|_{L^2(\hat{I}_{z,t}^{\pm,\varepsilon},J_t^\pm(.,z))}
		\]
		for $\varepsilon$ small. \cite{MoserACvACND}, Theorem 6.8,~2.~and uniform bounds for $q^\pm$, $J_t^\pm$ yield for the positive normalized eigenfunction $\Psi_{\varepsilon,z,t}^{\pm,1}$ to the first eigenvalue $\lambda_{\varepsilon,z,t}^{\pm,1}$ of $\Lc_{\varepsilon,z,t}^{\pm,0}$ the estimate
		\begin{align}\begin{split}\label{eq_SE_ACalpha_VpxVp_4}
		|(\Psi_{\varepsilon,z,t}^\pm,\Psi_{\varepsilon,z,t}^{\pm,1})_{\varepsilon,z,t}|
		&\leq (\overline{C}_4|\alpha-\frac{\pi}{2}|+
		C(\tilde{\delta})\varepsilon)\|\Psi_{\varepsilon,z,t}^\pm\|_{\varepsilon,z,t}\\
		&+C(\tilde{\delta})e^{-c\tilde{\delta}/\varepsilon}\|\tilde{\psi}_t^\pm(.,z)\|_{L^2(\hat{I}_{z,t}^{\pm,\varepsilon},J_t^\pm(.,z))}\end{split}
		\end{align}
		for a.e.~$z\in[0,\tilde{\mu}_0]$, all $t\in[0,T]$ and $\varepsilon\in(0,\varepsilon_0]$, if $\varepsilon_0>0$ is small. 
		
		With the analogous computation as in the case $\alpha=\frac{\pi}{2}$, cf.~the proof of (4.10) in \cite{AbelsMoser}, it follows from \cite{MoserACvACND}, Theorem 6.8,~1.~and~3.~that, if $\tilde{c}(\alpha,\tilde{\delta})=\overline{C}_3|\alpha-\frac{\pi}{2}|+2C\tilde{\delta}\leq 1$, then it holds
		\begin{align*}
		B_{\varepsilon,z,t}^{\pm,\tilde{c}}(\Psi_{\varepsilon,z,t}^\pm,\Psi_{\varepsilon,z,t}^\pm)
		\geq\|\Psi_{\varepsilon,z,t}^\pm\|^2_{\varepsilon,z,t}\left[\nu_2(1-\tilde{c}(\alpha,\tilde{\delta}))-\tilde{c}(\alpha,\tilde{\delta})\,\sup_{\rho\in\R}|f''(\theta_0(\rho))|\right]\\
		-(1-\tilde{c}(\alpha,\tilde{\delta}))(\Oc(\varepsilon^2)+\nu_2)\left|(\Psi_{\varepsilon,z,t}^\pm,\Psi_{\varepsilon,z,t}^{\pm,1})_{\varepsilon,z,t}\right|^2
		\end{align*}
		for a.e.~$z\in[0,\tilde{\mu}_0]$, all $t\in[0,T]$ and $\varepsilon\in(0,\varepsilon_0]$ if $\varepsilon_0=\varepsilon_0(\tilde{\delta},\overline{C}_0)>0$ is small. We combine this with \eqref{eq_SE_ACalpha_VpxVp_4} in order to show \eqref{eq_SE_ACalpha_VpxVp_3}. Note that $\nu_2$ from \cite{MoserACvACND}, Theorem 6.8, does not depend on $\alpha,\tilde{\delta}$. Therefore we can first choose $\tilde{\alpha}_0>0$ small such that
		\[
		\overline{C}_3\tilde{\alpha}_0\leq \frac{1}{4},\quad
		(\nu_2+\sup_{\rho\in\R}|f''(\theta_0(\rho))|)\overline{C}_3\tilde{\alpha}_0\leq\frac{\nu_2}{4}
		\quad\text{ and }\quad \frac{1}{2}\overline{C}_4^2\tilde{\alpha}_0^2\leq\frac{1}{4}.
		\]
		Note that this can be achieved independent of $\Omega,\Gamma$.
		Then let $\tilde{\delta}>0$ be small such that 
		\[
		2C\tilde{\delta}\leq\frac{1}{2}\quad\text{ and }\quad
		(\nu_2+\sup_{\rho\in\R}|f''(\theta_0(\rho))|)2C\tilde{\delta}
		\leq\frac{\nu_2}{4}.
		\]
		These estimates imply $\tilde{c}(\alpha,\tilde{\delta})\leq\frac{1}{2}$ and that the term in the square brackets above is estimated from below by $\frac{\nu_2}{2}$. Finally, we can choose $\varepsilon_0>0$ small such that \eqref{eq_SE_ACalpha_VpxVp_3} holds with $\overline{\nu}=\frac{\nu_2}{8}$. Finally, altogether we have proven Lemma \ref{th_SE_ACalpha_VpxVp}.\end{proof}
\end{proof}

For $B_{\varepsilon,t}^\pm$ on $\hat{V}_{\varepsilon,t}^\pm\times(\hat{V}_{\varepsilon,t}^\pm)^\perp$ we obtain
\begin{Lemma}\label{th_SE_ACalpha_VxVp}
	There is an $\overline{\alpha}_0>0$ independent of $\Omega,\Gamma$ such that, if $\alpha\in\frac{\pi}{2}+[-\overline{\alpha}_0,\overline{\alpha}_0]$, then there are $\varepsilon_0,C>0$ such that
	\[
	|B_{\varepsilon,t}^\pm(\phi,\psi)|\leq\frac{C}{\varepsilon}\|\phi\|_{L^2(\Omega_t^{C\pm})}\|\psi\|_{L^2(\Omega_t^{C\pm})}+\frac{1}{4}B_{\varepsilon,t}^\pm(\psi,\psi)+(\frac{\overline{c}}{4}+C\varepsilon^2)\|a\|_{H^1(0,\tilde{\mu}_0)}^2
	\]
	for all $\phi=a(z_\alpha^\pm(.,t))\phi^A_{\varepsilon,\alpha}(.,t)\in \hat{V}_{\varepsilon,t}^\pm$, $\psi\in (\hat{V}_{\varepsilon,t}^\pm)^\perp$ and $\varepsilon\in(0,\varepsilon_0]$, $t\in[0,T]$, where the constant $\overline{c}=\frac{1}{2}\|\theta_0'\|_{L^2(\R)}^2$ is as in Lemma \ref{th_SE_ACalpha_VxV}.
\end{Lemma}

\begin{proof}
	Analogously as in the case $\alpha=\frac{\pi}{2}$, cf.~the proof of \cite{AbelsMoser}, Lemma 4.10, we have\phantom{\qedhere}
	\begin{align*}
	B_{\varepsilon,t}^\pm(\phi,\psi)
	=\int_{\Omega_t^{C\pm}} a(z_\alpha^\pm)|_{(.,t)}\psi\Lc_{\varepsilon,t}^\pm\phi^A_{\varepsilon,\alpha}|_{(.,t)}\,dx
	+\int_{\partial\Omega_t^{C\pm}}\Nc_{\varepsilon,t}^\pm\phi^A_{\varepsilon,\alpha}|_{(.,t)}\tr\left[a(z_\alpha^\pm(.,t))\psi\right]\,d\Hc^1\\
	+\int_{\Omega_t^{C\pm}}\nabla(a(z_\alpha^\pm))|_{(.,t)}\cdot\left[\phi^A_{\varepsilon,\alpha}|_{(.,t)}\nabla\psi-\nabla\phi^A_{\varepsilon,\alpha}|_{(.,t)}\psi\right]\,dx=:(I)+(II)+(III).
	\end{align*}
	
	\begin{proof}[Ad $(I)$] 
		It holds $|(I)|\leq\|a(z_\alpha^\pm|_{(.,t)})\Lc_{\varepsilon,t}^\pm\phi^A_{\varepsilon,\alpha}(.,t)\|_{L^2(\Omega_t^{C\pm})}\|\psi\|_{L^2(\Omega_t^{C\pm})}$ because of the Hölder Inequality, where
		\[
		\|a(z_\alpha^\pm|_{(.,t)})\Lc_{\varepsilon,t}^\pm\phi^A_{\varepsilon,\alpha}|_{(.,t)}\|_{L^2(\Omega_t^{C\pm})}^2=\int_0^{\tilde{\mu}_0}a^2(z)\int_{-\delta_0}^{\delta_0} (\Lc_{\varepsilon,t}^\pm\phi^A_{\varepsilon,\alpha}(.,t))^2|_{X^\pm(r,z,t)}\,J_t^\pm(r,z)\,dr\,dz.
		\]
		Analogously to the case $\alpha=\frac{\pi}{2}$, cf.~the proof of \cite{AbelsMoser}, Lemma 4.10, we obtain from Lemma \ref{th_SE_ACalpha_phi_A} that $(\Lc_{\varepsilon,t}^\pm\phi^A_{\varepsilon,\alpha}|_{(.,t)})^2$ is estimated by\phantom{\qedhere}
		\[
		\frac{1}{\varepsilon^3}\!\left|[\Delta r|_{\overline{X}_0(s(.,t),t)}q^\pm|_{(z_\alpha^\pm(.,t),t)}
		\!-\!2\cos\alpha\partial_zq^\pm|_{(z_\alpha^\pm(.,t),t)}]\partial_\rho^2 v_\alpha|_{(\rho_{\varepsilon,\alpha}(.,t),Z_{\varepsilon,\alpha}^\pm(.,t))}\right|^2\!\!+\frac{\tilde{C}}{\varepsilon^2}e^{-c|\rho_{\varepsilon,\alpha}(.,t)|}.
		\]
		Therefore \cite{MoserACvACND}, Lemma 6.5, implies that the inner integral above is estimated by $C/\varepsilon^2$ and because of Lemma \ref{th_SE_ACalpha_split_L2},~1.~we get
		\[
		|(I)|\leq \frac{C}{\varepsilon}\|a\|_{L^2(0,\tilde{\mu}_0)}\|\psi\|_{L^2(\Omega_t^{C\pm})}
		\leq\frac{\tilde{C}}{\varepsilon}\|\phi\|_{L^2(\Omega_t^{C\pm})}\|\psi\|_{L^2(\Omega_t^{C\pm})}
		\] 
		for all $t\in[0,T]$ and $\varepsilon\in(0,\varepsilon_0]$ provided that $\varepsilon_0>0$ is small. \phantom{\qedhere}\end{proof}
	
	\begin{proof}[Ad $(II)$] 
		Because of Hölder's inequality we obtain\phantom{\qedhere} 
		\[	
		|(II)|\leq\|\tr\,\psi\|_{L^2(\partial\Omega_t^{C\pm})}\|\tr(a(z_\alpha^\pm|_{(.,t)}))\Nc_{\varepsilon,t}^\pm\phi^A_{\varepsilon,\alpha}|_{(.,t)}\|_{L^2(\partial\Omega_t^{C\pm})}.
		\] 
		For the second integral we use the representation of integrals over curves, cf.~also the estimate of $(III)$ in the proof of Lemma \ref{th_SE_ACalpha_VxV}. Then Lemma \ref{th_SE_ACalpha_phi_A} and \cite{MoserACvACND}, Lemma 6.5, yield
		\[
		\|a(z_\alpha^\pm)\Nc_{\varepsilon,t}^\pm\phi^A_{\varepsilon,\alpha}|_{(.,t)}\|_{L^2(\partial\Omega_t^{C\pm})}\leq C\varepsilon|a(0)|+ Ce^{-c/\varepsilon}\|a\|_{L^2(0,\tilde{\mu}_0)}\leq C\varepsilon\|a\|_{H^1(0,\tilde{\mu}_0)}.
		\]
		We estimate $\|\tr\,\psi\|_{L^2(\partial\Omega_t^{C\pm})}$ with Lemma \ref{th_SE_ACalpha_intpol_tr1}. Then Young's inequality and Lemma \ref{th_SE_ACalpha_VpxVp} imply
		\[
		|(II)|\leq
		\frac{\nu}{8\varepsilon \overline{C}_1}\|\tr\,\psi\|_{L^2(\partial\Omega_t^{C\pm})}^2
		+\tilde{C}\varepsilon^3\|a\|_{H^1(0,\tilde{\mu}_0)}^2\leq\frac{1}{8}B_{\varepsilon,t}^\pm(\psi,\psi)
		+\tilde{C}\varepsilon^3\|a\|_{H^1(0,\tilde{\mu}_0)}^2,
		\]
		where $\overline{C}_1$ is as in Lemma \ref{th_SE_ACalpha_intpol_tr1}.
	\end{proof}
	
	\begin{proof}[Ad $(III)$]
		We proceed in the analogous way as in the case $\alpha=\frac{\pi}{2}$, cf.~the estimate of $(III)$ in the proof of \cite{AbelsMoser}, Lemma 4.10. However, there are some new terms due to $\nabla r\cdot\nabla z_\alpha^\pm$ and since $\partial_\rho v_\alpha$ depends on $Z$. It holds $(III)=\int_0^{\tilde{\mu}_0}a'(z)g_t^\pm(z)\,dz$ with
		\[
		g_t^\pm(z):=\int_{-\delta_0}^{\delta_0} \nabla z_\alpha^\pm|_{\overline{X}^\pm(r,z,t)}\cdot\left[\phi^A_{\varepsilon,\alpha}(.,t)\nabla\psi-\nabla\phi^A_{\varepsilon,\alpha}(.,t)\psi\right]|_{X^\pm(r,z,t)}\,J_t^\pm(r,z)\,dr.
		\]
		We insert $\nabla\psi|_{X^\pm(.,t)}=
		\nabla r|_{\overline{X}^\pm(.,t)}\partial_r\tilde{\psi}_t^\pm
		+\nabla z_\alpha^\pm|_{\overline{X}^\pm(.,t)} \partial_z\tilde{\psi}_t^\pm$ with
		$\tilde{\psi}_t^\pm:=\psi|_{X^\pm(.,t)}$. 
		For the $\partial_z\tilde{\psi}_t^\pm$-term in $g_t^\pm$ we use $|\nabla z_\alpha^\pm|^2|_{\overline{X}^\pm(r,z,t)}=1+\Oc(|r|)$ due to Remark \ref{th_coord2D_rem1} and Theorem \ref{th_coord2D}. Therefore $|g_t^\pm(z)|$ is for a.e.~$z\in[0,\tilde{\mu}_0]$ and all $t\in[0,T]$ estimated by
		\begin{align*}
		\left|\int_{-\delta_0}^{\delta_0}\left[\phi^A_{\varepsilon,\alpha}|_{\overline{X}^\pm(.,t)}\partial_z\tilde{\psi}_t^\pm J_t^\pm\right]|_{(r,z)}\,dr\right|+\int_{-\delta_0}^{\delta_0}\left|(\tilde{\psi}_t^\pm J_t^\pm)|_{(r,z)}\nabla z_\alpha^\pm\cdot\nabla\phi^A_{\varepsilon,\alpha}|_{\overline{X}^\pm(r,z,t)}\right|\,dr\\
		+\int_{-\delta_0}^{\delta_0}\left[
		C\left|r\, \partial_z\tilde{\psi}_t^\pm|_{(r,z)}\right|
		+\left|\nabla r\cdot\nabla z_\alpha^\pm|_{\overline{X}^\pm(r,z,t)} \partial_r\tilde{\psi}_t^\pm|_{(r,z)}\right|\right] \cdot\left|\phi^A_{\varepsilon,\alpha}|_{\overline{X}^\pm(r,z,t)}
		J_t^\pm|_{(r,z)}\right|\,dr.
		\end{align*}
		We use $\psi\in(\hat{V}_{\varepsilon,t}^\pm)^\perp$ to rewrite the first term. With the properties of Sobolev spaces on product sets, cf.~\cite{MoserACvACND}, Lemma 2.10, and since integration yields a bounded linear operator on $L^2(-\delta,\delta)$, we can differentiate the identity in Lemma \ref{th_SE_ACalpha_split_L2},~2.~and use the product rule. Hence the first term is estimated by 
		\[
		\left|\int_{-\delta_0}^{\delta_0}\left[\left(\partial_z(\phi^A_{\varepsilon,\alpha}|_{\overline{X}^\pm(.,t)})J_t^\pm+\phi^A_{\varepsilon,\alpha}|_{\overline{X}^\pm(.,t)}\partial_z J_t^\pm\right)\tilde{\psi}_t^\pm\right]|_{(r,z)}\,dr\right|.
		\]
		Now we use the structure of $\phi^A_{\varepsilon,\alpha}$. In \eqref{eq_SE_ACalpha_nabla_phiA} we computed $\nabla\phi^A_{\varepsilon,\alpha}$ in $\Omega^{C\pm}_t$.	Moreover, it holds 
		\begin{align}\label{eq_SE_ACalpha_dz_phiA}
		&\sqrt{\varepsilon}\partial_z(\phi^A_{\varepsilon,\alpha}|_{\overline{X}^\pm})
		=-\partial_z \tilde{h}_\varepsilon^\pm\left[
		q^\pm|_{(z_\alpha^\pm,t)}\partial_\rho^2 v_\alpha|_{(\rho_{\varepsilon,\alpha},Z_{\varepsilon,\alpha}^\pm)}+\varepsilon\partial_\rho\hat{v}^{C\pm}_1|_{(\rho_{\varepsilon,\alpha},Z_{\varepsilon,\alpha}^\pm,t)}\right]\!|_{\overline{X}^\pm}\\
		&+\left[
		\partial_z q^\pm|_{(z_\alpha^\pm,t)}\partial_\rho v_\alpha|_{(\rho_{\varepsilon,\alpha},Z_{\varepsilon,\alpha}^\pm)}
		+\frac{1}{\varepsilon}q^\pm|_{(z_\alpha^\pm,t)}\partial_Z\partial_\rho v_\alpha|_{(\rho_{\varepsilon,\alpha},Z_{\varepsilon,\alpha}^\pm)}+
		\partial_Z\hat{v}^{C\pm}_1|_{(\rho_{\varepsilon,\alpha},Z_{\varepsilon,\alpha}^\pm,t)}
		\right]\!|_{\overline{X}^\pm}\notag
		\end{align} 
		for all $(r,z)\in[-\delta_0,\delta_0]\times[0,\tilde{\mu}_0]$. Consider the estimate for $|g_t^\pm(z)|$ after inserting $\phi^A_{\varepsilon,\alpha}$, $\nabla\phi^A_{\varepsilon,\alpha}$ and $\partial_z(\phi^A_{\varepsilon,\alpha}|_{\overline{X}^\pm})$. There are four new critical terms compared to the case $\alpha=\frac{\pi}{2}$, cf.~the estimate of $(III)$ in the proof of \cite{AbelsMoser}, Lemma 4.10. First, from $\partial_z(\phi^A_{\varepsilon,\alpha}|_{\overline{X}^\pm(.,t)})$ there is the contribution
		\[
		\frac{1}{\varepsilon^{3/2}}|q^\pm|_{(z,t)}|
		\left|\int_{-\delta_0}^{\delta_0}\partial_Z\partial_\rho v_\alpha|_{(\rho_{\varepsilon,\alpha}(\overline{X}^\pm(.,t)),\frac{z}{\varepsilon})}[\tilde{\psi}_t^\pm J_t^\pm](r,z)\,dr\right|.
		\]
		Moreover, due to $|\nabla z_\alpha^\pm|^2=1+\Oc(|r|)$ and $\nabla r\cdot\nabla z_\alpha^\pm=-\cos\alpha+\Oc(|r|)$ we get from the $\nabla\phi^A_{\varepsilon,\alpha}$-term the two remainders
		\[
		\frac{1}{\varepsilon^{3/2}}|q^\pm(z,t)|\int_{-\delta_0}^{\delta_0}\left[
		|\partial_Z\partial_\rho v_\alpha|_{(\rho_{\varepsilon,\alpha}(\overline{X}^\pm(.,t)),\frac{z}{\varepsilon})}|
		+|\cos\alpha\,\partial_\rho v_\alpha|_{(\rho_{\varepsilon,\alpha}(\overline{X}^\pm(.,t)),\frac{z}{\varepsilon})}|\right]\! |\tilde{\psi}_t^\pm J_t^\pm|(r,z)\,dr.
		\]
		Finally, the $\nabla r\cdot\nabla z_\alpha^\pm$ multiplied by $\partial_r\tilde{\psi}_t^\pm$ yields the term
		\[
		\frac{1}{\sqrt{\varepsilon}}|q^\pm(z,t)|\int_{-\delta_0}^{\delta_0}
		|\cos\alpha\partial_\rho v_\alpha|_{(\rho_{\varepsilon,\alpha}(\overline{X}^\pm(.,t)),\frac{z}{\varepsilon})}||\partial_r\tilde{\psi}_t^\pm J_t^\pm|(r,z)\,dr.
		\]
		For the $\partial_Z\partial_\rho v_\alpha$-terms we use 
		\[
		|\partial_Z\partial_\rho v_\alpha(\rho,Z)|\leq\overline{C}|\alpha-\frac{\pi}{2}|e^{-\beta_0|\rho|}\quad\text{ for all }(\rho,Z)\in\overline{\R^2_+}
		\] 
		because of Remark \ref{th_asym_ACalpha_decay_param},
		where $\overline{C},\beta_0>0$ are independent of $\Omega,\Gamma,\alpha$. Moreover, we split the last integral with $\partial_r\tilde{\psi}_t^\pm$ as in the proof of Lemma \ref{th_SE_ACalpha_VpxVp}, cf.~\eqref{eq_SE_ACalpha_VpxVp_int_split}, and we use \eqref{eq_SE_ACalpha_VpxVp_2}. Note that with $\hat{\alpha}_0$ and $\tilde{\delta}$ as in the proof of Lemma \ref{th_SE_ACalpha_VpxVp}, the prefactor in \eqref{eq_SE_ACalpha_VpxVp_2} is contained in $[\frac{1}{2},1]$ provided that $\alpha\in\frac{\pi}{2}+[-\hat{\alpha}_0,\hat{\alpha_0}]$ and $|r|\leq 2\tilde{\delta}$. Therefore the Hölder Inequality, \cite{MoserACvACND}, Lemma 6.5, Remark \ref{th_hp_alpha_solN3_rem}, $\frac{1}{2}\leq q^\pm\leq 2$ and $J_t^\pm=1+\Oc(|r|)$ yield
		\[
		|g_t^\pm(z)|\leq(\frac{1}{\varepsilon}\overline{C}_5|\alpha-\frac{\pi}{2}|+C) \left[\|\tilde{\psi}_t^\pm(.,z)\|_{L^2(-\delta_0,\delta_0;J_t^\pm(.,z))}+\varepsilon\|\nabla\psi|_{X^\pm(.,z)}\|_{L^2(-\delta_0,\delta_0;J_t^\pm(.,z))}\right]
		\]
		for a.e.~$z\in[0,\tilde{\mu}_0]$ and some $\overline{C}_5>0$ independent of $\Omega,\Gamma$ and $\alpha\in\frac{\pi}{2}+[-\hat{\alpha}_0,\hat{\alpha}_0]$. Therefore the Hölder Inequality and Young Inequality yield
		\begin{align*}
		|(III)|\leq (\overline{C}_6|\alpha-\frac{\pi}{2}|+C\varepsilon^2)\|a'\|_{L^2(0,\tilde{\mu}_0)}^2+\frac{\nu}{8}\left[\frac{1}{\varepsilon^2}\|\psi\|_{L^2(\Omega_t^{C\pm})}^2+\|\nabla\psi\|_{L^2(\Omega_t^{C\pm})}^2\right]
		\end{align*}
		for some $\overline{C}_6>0$ independent of $\Omega,\Gamma$, where $\nu$ is as in Lemma \ref{th_SE_ACalpha_VpxVp}. The last term is dominated by $\frac{1}{8}B_{\varepsilon,t}^\pm(\psi,\psi)$ due to Lemma \ref{th_SE_ACalpha_VpxVp}. Finally, we can choose $\overline{\alpha}_0>0$ small independent of $\Omega,\Gamma$ such that $\overline{C}_6\overline{\alpha}_0\leq\frac{\overline{c}}{4}$, where $\overline{c}$ is as in Lemma \ref{th_SE_ACalpha_VxV}. This shows the claim.\end{proof}
\end{proof}

Finally, we combine Lemma \ref{th_SE_ACalpha_VxV}-\ref{th_SE_ACalpha_VxVp}.

\begin{Theorem}\label{th_SE_ACalpha_cp2}
	There is an $\overline{\alpha}_0>0$ independent of $\Omega,\Gamma$ such that, if $\alpha\in\frac{\pi}{2}+[-\overline{\alpha}_0,\overline{\alpha}_0]$, then there are $\varepsilon_0, C, c_0>0$ such that for all $\varepsilon\in(0,\varepsilon_0], t\in[0,T]$ and $\psi\in H^1(\Omega_t^{C\pm})$ with $\psi(x)=0$ for a.e.~$x\in\Omega_t^{C\pm}$ with $z_\alpha^\pm(x,t)\geq\hat{\mu}_0$ it holds
	\[
	B_{\varepsilon,t}^\pm(\psi,\psi)\geq -C\|\psi\|_{L^2(\Omega_t^{C\pm})}^2+c_0\varepsilon\|\nabla_\tau\psi\|_{L^2(\Omega_t^{C\pm})}^2.
	\] 
\end{Theorem}

\begin{Remark}\upshape\begin{enumerate}
		\item In the proof below the estimate is slightly better.
		\item Theorem \ref{th_SE_ACalpha_cp2} directly implies Theorem \ref{th_SE_ACalpha_cp}, cf.~the inception of Section \ref{sec_SE_ACalpha_outline}.
	\end{enumerate}
\end{Remark}

\begin{proof}[Proof of Theorem \ref{th_SE_ACalpha_cp2}]
	Let $t\in[0,T]$ and $\psi\in\tilde{H}^1(\Omega_t^{C\pm})$. Due to Lemma \ref{th_SE_ACalpha_split_L2} we can uniquely write
	\[
	\psi=\phi+\phi^\perp\quad\text{ with } \phi=[a(z_\alpha^\pm)\phi^A_{\varepsilon,\alpha}]|_{(.,t)}\in \hat{V}_{\varepsilon,t}^\pm\text{ and }
	\phi^\perp\in(\hat{V}_{\varepsilon,t}^\pm)^\perp.
	\]
	Analogously to the case $\alpha=\frac{\pi}{2}$, cf.~the proof of \cite{AbelsMoser}, Theorem 4.11 we obtain from
	Lemma \ref{th_SE_ACalpha_VxV}, Lemma \ref{th_SE_ACalpha_VxVp} and Lemma \ref{th_SE_ACalpha_VpxVp} that there are $C,\varepsilon_0>0$ independent of $\psi,\varepsilon,t$ such that for all $\varepsilon\in(0,\varepsilon_0]$ it holds
	\[
	B_{\varepsilon,t}^\pm(\psi,\psi)\geq-C\|\phi\|_{L^2(\Omega_t^{C\pm})}^2+\frac{\nu}{4\varepsilon^2}\|\phi^\perp\|_{L^2(\Omega_t^{C\pm})}^2+\frac{\overline{c}}{4}\|a\|_{H^1(0,\tilde{\mu}_0)}^2+\frac{\nu}{2}\|\nabla(\phi^\perp)\|_{L^2(\Omega_t^{C\pm})}^2.
	\]
	
	It remains to include the $\nabla_\tau\psi$-term in the estimate. By the triangle inequality we have 
	\[
	\|\nabla_\tau\psi\|_{L^2(\Omega_t^{C\pm})}\leq\|\nabla_\tau\phi\|_{L^2(\Omega_t^{C\pm})}+\|\nabla_\tau(\phi^\perp)\|_{L^2(\Omega_t^{C\pm})}.
	\] 
	Theorem \ref{th_coord2D} yields $\|\nabla_\tau(\phi^\perp)\|_{L^2(\Omega_t^{C\pm})}\leq C\|\nabla(\phi^\perp)\|_{L^2(\Omega_t^{C\pm})}$. Moreover, by definition 
	\[
	\nabla_\tau\psi|_{X(r,s,t)}=
	\nabla s|_{\overline{X}(r,s,t)}\partial_s(\phi|_{X(r,s,t)})
	=\nabla s|_{\overline{X}(r,s,t)}\partial_s(a(z_\alpha^\pm|_{X(r,s,t)})\phi^A_{\varepsilon,\alpha}|_{\overline{X}(r,s,t)}).
	\]
	Note that $\partial_s(z_\alpha^\pm|_{\overline{X}(r,s,t)})=\mp\sin\alpha$ due to the definition \eqref{eq_asym_ACalpha_cp_zalpha} of $z_\alpha^\pm$ and therefore
	\begin{align*}
	\sqrt{\varepsilon}\partial_s(\phi^A_{\varepsilon,\alpha}|_{\overline{X}})
	&=(\mp\sin\alpha)\partial_zq^\pm(z_\alpha^\pm|_{\overline{X}},t)\partial_\rho v_\alpha(\rho_{\varepsilon,\alpha},Z_{\varepsilon,\alpha}^\pm)|_{\overline{X}}\\
	&+q^\pm(z_\alpha^\pm|_{\overline{X}},t)\left[
	-\partial_sh_{\varepsilon,\alpha}\partial_\rho^2v_\alpha(\rho_{\varepsilon,\alpha},Z_{\varepsilon,\alpha}^\pm)|_{\overline{X}} +\frac{\mp\sin\alpha}{\varepsilon}\partial_Z
	\partial_\rho v_\alpha
	(\rho_{\varepsilon,\alpha},Z_{\varepsilon,\alpha}^\pm)|_{\overline{X}}
	\right]\\
	&+\varepsilon\left[
	-\partial_sh_{\varepsilon,\alpha}\partial_\rho \hat{v}^{C\pm}_1(\rho_{\varepsilon,\alpha},Z_{\varepsilon,\alpha}^\pm,t)|_{\overline{X}}
	+\frac{\mp\sin\alpha}{\varepsilon}\partial_Z \hat{v}^{C\pm}_1
	(\rho_{\varepsilon,\alpha},Z_{\varepsilon,\alpha}^\pm,t)|_{\overline{X}}
	\right],
	\end{align*} 
	where the $h_{\varepsilon,\alpha}$-terms are evaluated at $(s,t)$. We estimate all appearing terms in $\|\nabla_\tau\phi\|_{L^2(\Omega_t^{C\pm})}^2$ using several times $(d+\tilde{d})^2\leq 2(d^2+\tilde{d}^2)$ for $d,\tilde{d}\geq 0$. All terms are multiplied by a $\frac{1}{\varepsilon}$-factor (or better) except the $\frac{1}{\varepsilon^3}|\partial_Z\partial_\rho v_\alpha|^2$-term. Let us first estimate all terms except the latter one. We transform to $(r,z)$-coordinates, use the Fubini Theorem and \cite{MoserACvACND}, Lemma 6.5. Then these terms are controlled by $C\|a\|_{H^1(0,\tilde{\mu}_0)}^2$. Now we consider the $\partial_Z\partial_\rho v_\alpha$-term. The latter is estimated by
	\begin{align}\label{eq_SE_ACalpha_cp2}
	C\frac{1}{\varepsilon^3}\int_0^{\tilde{\mu}_0}a^2(z)\int_{-\delta_0}^{\delta_0}|\partial_Z\partial_\rho v_\alpha(\rho_{\varepsilon,\alpha}|_{\overline{X}^\pm(r,z,t)},\frac{z}{\varepsilon})|^2\,dr\,dz.
	\end{align}
	We use $|a(z)|\leq C\|a\|_{H^1(0,\tilde{\mu}_0)}$ for all $z\in[0,\tilde{\mu}_0]$ due to the Fundamental Theorem and
	\[
	|\partial_Z\partial_\rho v_\alpha(\rho,Z)|\leq Ce^{-\beta_0|\rho|-\gamma_0Z}\quad\text{ for all }(\rho,Z)\in\overline{\R^2_+}
	\]
	because of Remark \ref{th_asym_ACalpha_decay_param}. Therefore \cite{MoserACvACND}, Lemma 6.5, for the inner integral and another scaling argument for the $z$-integral yields that \eqref{eq_SE_ACalpha_cp2} is estimated by $C\frac{1}{\varepsilon}\|a\|_{H^1(0,\tilde{\mu}_0)}^2$. Finally, combining this with the previous estimate for $B_{\varepsilon,t}^\pm$ this yields the claim.\end{proof}

\section{Difference Estimates and Proofs of the Convergence Theorems}
\label{sec_DC}
In this section we use a Gronwall-type argument in order to control the difference of the exact and approximate solutions for \hyperlink{ACalpha}{(AC$_\alpha$)}. In the last Section \ref{sec_SE_ACalpha} we proved a spectral estimate as a preparation. Additionally, we have to estimate some nonlinear terms due to terms involving the potential. Therefore as preparation we remark a uniform a priori bound for exact classical solutions of \hyperlink{ACalpha}{(AC$_\alpha$)} in Section \ref{sec_DC_prelim_bdd_scal}. Moreover, we note some Gagliardo-Nirenberg estimates in Section \ref{sec_DC_prelim_GN}. Then in Section \ref{sec_DC_ACalpha} we prove a difference estimate and we show Theorem \ref{th_ACalpha_conv}.

\subsection{Preliminaries}\label{sec_DC_prelim}
\subsubsection{Uniform A Priori Bound for Classical Solutions of (AC$_\alpha$)}\label{sec_DC_prelim_bdd_scal}
Let $\Omega$, $Q_T$, $\partial Q_T$ be as in Remark \ref{th_intro_coord},~1.~and $\varepsilon>0$. We show uniform bounds for classical solutions of \hyperlink{ACalpha}{(AC$_\alpha$)}. Let $f$ satisfy \eqref{eq_AC_fvor1} and $R_0\geq 1$ such that the condition \eqref{eq_AC_fvor2} for $f'$ is fulfilled. Moreover, let $\alpha\in(0,\pi)$ and $\sigma_\alpha:\R\rightarrow\R$ be smooth with $\textup{supp}\,\sigma_\alpha'\subset(-1,1)$.
\begin{Lemma}\label{th_DC_bdd_scal}
	Let $u_{0,\varepsilon,\alpha}\in C^0(\overline{\Omega})$ and $u_{\varepsilon,\alpha}\in 
	C^0(\overline{Q_T})\cap C^1(\overline{\Omega}\times(0,T])\cap C^2(\Omega\times(0,T])$ be a solution of \eqref{eq_ACalpha1}-\eqref{eq_ACalpha3}. Then 
	\[
	\|u_{\varepsilon,\alpha}\|_{L^\infty(Q_T)}\leq\max\{R_0,\|u_{0,\varepsilon,\alpha}\|_{L^\infty(\Omega)}\}.
	\]
\end{Lemma}
\begin{proof}
	One can prove this via contradiction and ideas from the proof of the parabolic weak maximum principle. See \cite{MoserDiss}, Lemma 7.1 for details.
\end{proof}

\subsubsection{Gagliardo-Nirenberg Inequalities}\label{sec_DC_prelim_GN}
We remark some Gagliardo-Nirenberg estimates. 
\begin{Lemma}[\textbf{Gagliardo-Nirenberg Inequality}]\label{th_DC_GN}
	Let $n\in\N$, $1\leq p,q,r\leq\infty$ and $\theta\in[0,1]$ be such that
	\[
	\theta\left(\frac{1}{p}-\frac{1}{n}\right)+\frac{1-\theta}{q}=\frac{1}{r},
	\]
	where $\frac{1}{\infty}:=0$. Moreover, if $p=n>1$, then let $\theta<1$. Then it holds
	\[
	\|u\|_{L^r(\R^n)}\leq c\|u\|_{L^q(\R^n)}^{1-\theta}\|\nabla u\|_{L^p(\R^n)}^{\theta}
	\]
	for all $u\in L^q(\R^n)\cap W^{1,p}(\R^n)$ and some constant $c=c(n,p,q,r)>0$.
\end{Lemma}
\begin{proof}
	Cf.~Leoni \cite{Leoni}, Theorem 12.83.
\end{proof}

\begin{Remark}\label{th_DC_GN_rem}\upshape
	Using appropriate extension operators the estimate in Lemma \ref{th_DC_GN} can be transferred to domains with uniform Lipschitz boundary provided that $\|\nabla u\|_{L^p(\R^n)}$ in the estimate is replaced by $\|u\|_{W^{1,p}(\R^n)}$ and the constant in the estimate additionally depends on the operator norm of the extension operator. For such operators cf.~Leoni \cite{Leoni}, Theorem 13.8 and Theorem 13.17. We remark that the operator norms in \cite{Leoni} are controlled just via the standard quantities and the geometrical parameters of $\Omega$ and $\partial\Omega$. Therefore if the latter can be estimated uniformly, the operator norms and the constants in the above Gagliardo-Nirenberg estimates can be chosen uniformly with respect to the domain $\Omega$.
\end{Remark}

\subsection[Difference Estimate and Proof of the Convergence Thm. for (AC$_\alpha$) in 2D]{Difference Estimate and Proof of the Convergence Theorem for (AC$_\alpha$) in 2D}\label{sec_DC_ACalpha}
We show in Section \ref{sec_DC_ACalpha_DE} the difference estimate for exact solutions and appropriate approximate solutions for the Allen-Cahn equation with nonlinear Robin-boundary condition \eqref{eq_ACalpha1}-\eqref{eq_ACalpha3} in 2D. In Section \ref{sec_DC_ACalpha_conv} we prove the main Theorem \ref{th_ACalpha_conv} by checking the requirements for the difference estimate for the approximate solution from Section \ref{sec_asym_ACalpha_uA}.

\subsubsection{Difference Estimate}\label{sec_DC_ACalpha_DE}
\begin{Theorem}[\textbf{Difference Estimate for (AC$_\alpha$)}]\label{th_DC_ACalpha}
	Let $\Omega$, $Q_T$ and $\partial Q_T$ be as in Remark \ref{th_intro_coord},~1. Let $\Gamma=(\Gamma_t)_{t\in[0,T_0]}$ for some $T_0>0$ be as in Section \ref{sec_coord2D} with contact angle $\alpha\in(0,\pi)$ and $\delta>0$ be such that Theorem \ref{th_coord2D} holds with $\delta$ replaced by $2\delta$. Moreover, let $\Gamma_t(\delta)$, $\Gamma(\delta)$, $\nabla_\tau$ and $\partial_n$ be as in Remark \ref{th_coord2D_rem2}. Let $f$ fulfill \eqref{eq_AC_fvor1}-\eqref{eq_AC_fvor2} and $\sigma_\alpha$ for $\alpha\in(0,\pi)$ be as in Definition \ref{th_ACalpha_sigma_def}.
	
	Let $\varepsilon_0>0$, $u^A_{\varepsilon,\alpha}\in C^2(\overline{Q_{T_0}})$ and $u_{0,\varepsilon,\alpha}\in C^2(\overline{\Omega})$ with $\partial_{N_{\partial\Omega}} u_{0,\varepsilon,\alpha}+\frac{1}{\varepsilon}\sigma_\alpha'(u_{0,\varepsilon,\alpha})=0$ on $\partial\Omega$ and $u_{\varepsilon,\alpha}\in C^2(\overline{Q_{T_0}})$ be exact solutions to \eqref{eq_ACalpha1}-\eqref{eq_ACalpha3} with $u_{0,\varepsilon,\alpha}$ in \eqref{eq_ACalpha3},  $\varepsilon\in(0,\varepsilon_0]$. 
	
	For a $R>0$, $M\in\N, M\geq 3$ and some $\delta_0\in(0,\delta]$ we require the following conditions:
	\begin{enumerate}
		\item \textup{Uniform Boundedness:} $\sup_{\varepsilon\in(0,\varepsilon_0]}\|u^A_{\varepsilon,\alpha}\|_{L^\infty(Q_{T_0})}+\|u_{0,\varepsilon,\alpha}\|_{L^\infty(\Omega)}<\infty$.
		\item \textup{Spectral Estimate:} There are $c_0,C>0$ such that
		\begin{align*}
		\int_\Omega|\nabla\psi|^2
		&+\frac{1}{\varepsilon^2}f''(u^A_{\varepsilon,\alpha}|_{(.,t)})\psi^2\,dx
		+\int_{\partial\Omega}\frac{1}{\varepsilon}\sigma_\alpha''(u^A_{\varepsilon,\alpha}|_{(.,t)})(\tr\,\psi)^2\,d\Hc^1\\
		&\geq -C\|\psi\|_{L^2(\Omega)}^2+\|\nabla\psi\|_{L^2(\Omega\setminus\Gamma_t(\delta_0))}^2+c_0\varepsilon\|\nabla_\tau\psi\|_{L^2(\Gamma_t(\delta_0))}^2
		\end{align*}
		for all $\psi\in H^1(\Omega)$ and $\varepsilon\in(0,\varepsilon_0],t\in[0,T_0]$.
		\item \textup{Approximate Solution:} The remainders 
		\[
		r^A_{\varepsilon,\alpha}:=\partial_t u^A_{\varepsilon,\alpha}-\Delta u^A_{\varepsilon,\alpha}+\frac{1}{\varepsilon^2}f'(u^A_{\varepsilon,\alpha})\quad\text{ and }\quad s^A_{\varepsilon,\alpha}:=\partial_{N_{\partial\Omega}}u^A_{\varepsilon,\alpha}+\frac{1}{\varepsilon}\sigma_\alpha'(u^A_{\varepsilon,\alpha})
		\]
		in \eqref{eq_ACalpha1}-\eqref{eq_ACalpha2} for $u^A_{\varepsilon,\alpha}$ and the difference $\overline{u}_{\varepsilon,\alpha}:=u_{\varepsilon,\alpha}-u^A_{\varepsilon,\alpha}$ satisfy
		\begin{align}\begin{split}\label{eq_DC_ACalpha_uA}
		&\left|\int_{\partial\Omega}s^A_{\varepsilon,\alpha}\tr\,\overline{u}_{\varepsilon,\alpha}(t)\,d\Hc^1+\int_\Omega r^A_{\varepsilon,\alpha}\overline{u}_{\varepsilon,\alpha}(t)\,dx\right|\\
		&\leq C\varepsilon^{M+\frac{1}{2}}(\|\overline{u}_{\varepsilon,\alpha}(t)\|_{L^2(\Omega)}+\|\nabla_\tau\overline{u}_{\varepsilon,\alpha}(t)\|_{L^2(\Gamma_t(\delta_0))}+\|\nabla\overline{u}_{\varepsilon,\alpha}(t)\|_{L^2(\Omega\setminus\Gamma_t(\delta_0))})
		\end{split}
		\end{align}
		for all $\varepsilon\in(0,\varepsilon_0]$ and $T\in(0,T_0]$.
		\item \textup{Well-Prepared Initial Data:} For all $\varepsilon\in(0,\varepsilon_0]$ we have
		\begin{align}\label{eq_DC_ACalpha_u0} \|u_{0,\varepsilon,\alpha}-u^A_{\varepsilon,\alpha}|_{t=0}\|_{L^2(\Omega)}\leq R\varepsilon^M.
		\end{align}
	\end{enumerate} 
	Then the following assertions hold.
	\begin{enumerate}
		\item Let $M>3$. Then there exist $\beta,\varepsilon_1>0$ such that for $g_\beta(t):=e^{-\beta t}$ we have
		\begin{align}
		\begin{split}
		\sup_{t\in[0,T]}\|(g_\beta\overline{u}_{\varepsilon,\alpha})(t)\|_{L^2(\Omega)}^2+\|g_\beta\nabla\overline{u}_{\varepsilon,\alpha}\|_{L^2(Q_T\setminus\Gamma(\delta_0))}^2&\leq 2R^2\varepsilon^{2M},\\
		c_0\varepsilon\|g_\beta\nabla_\tau\overline{u}_{\varepsilon,\alpha}\|^2_{L^2(Q_T\cap\Gamma(\delta_0))}+\varepsilon^2\|g_\beta\partial_n\overline{u}_{\varepsilon,\alpha}\|^2_{L^2(Q_T\cap\Gamma(\delta_0))}&\leq 2R^2\varepsilon^{2M}\label{eq_DC_ACalpha_DE}\end{split}
		\end{align}
		for all $\varepsilon\in(0,\varepsilon_1]$ and $T\in(0,T_0]$.
		\item Let $M=3$ and \eqref{eq_DC_ACalpha_uA} hold for some $\tilde{M}>M$ instead of $M$. Then there are $\beta,\tilde{R},\varepsilon_1>0$ such that, provided that \eqref{eq_DC_ACalpha_u0} is valid for $\tilde{R}$ instead of $R$, then $\eqref{eq_DC_ACalpha_DE}$ for $\tilde{R}$ instead of $R$ holds for all $\varepsilon\in(0,\varepsilon_1], T\in(0,T_0]$. 
		\item Let $M=3$. Then there exist $\varepsilon_1,T_1>0$ such that $\eqref{eq_DC_ACalpha_DE}$ is true for $\beta=0$ and for all $\varepsilon\in(0,\varepsilon_1], T\in(0,T_1]$.
	\end{enumerate}
\end{Theorem}
\begin{Remark}\upshape\phantomsection{\label{th_DC_ACalpha_rem}}
	\begin{enumerate}
		\item The parameter $M$ is related to the order of the approximate solution in Section \ref{sec_asym_ACalpha}. The $\delta_0$ is introduced because in the application of Theorem \ref{th_DC_ACalpha} later we use the spectral estimate in Theorem \ref{th_SE_ACalpha}. There $\delta_0$ was chosen small in order to have \eqref{eq_asym_ACalpha_uA_delta0}. The parameter $\beta$ is used in order to get a result for all times $T\in(0,T_0]$.
		\item Weaker requirements in the theorem also yield a result, e.g.~when the two additional terms on the right hand side of the spectral estimate absent or one only has an estimate with the full $H^1$-norm on the right hand side in \eqref{eq_DC_ACalpha_uA}. But then the obtained assertions may be weaker as well, e.g.~the critical order $3$ for $M$ could be larger and the $\varepsilon$-orders in \eqref{eq_DC_ACalpha_DE} may be weakened. The reason is that the $H^1$-norm can be estimated with the spectral term but one has to invest the error $\varepsilon^{-2}$ times the $L^2$-norm. 
		\item The parameter $3$ is critical for $M$ in our proof, see \eqref{eq_DC_ACalpha_DE2} below. One problem are estimates of cubic terms. Moreover, case $M=3$ is complicated since in \eqref{eq_DC_ACalpha_DE2} there is a term of order larger than $2$ in $R$ and a linear term in $R$, but the needed order is $2$ in $R$. The linear term in $R$ stems from \eqref{eq_DC_ACalpha_uA}. For the parameter $\beta$ there is a similar issue in the critical case. Compared to the case $\alpha=\frac{\pi}{2}$, cf.~\cite{AbelsMoser}, Theorem 5.1, the critical order for $M$ is increased by one and the estimate \eqref{eq_DC_ACalpha_DE} is slightly weaker. This is because we only have a spectral estimate with the $\varepsilon$-factor in front of the $\nabla_\tau$-term.
		\item In the proof of the difference estimate for the case $\alpha=\frac{\pi}{2}$, cf.~\cite{AbelsMoser}, proof of Theorem 5.1, we applied Gagliardo-Nirenberg inequalities for the integral on $\Gamma_t(\delta)$ subsequently in tangential and normal direction. These estimates are difficult to adapt for the case $\alpha\neq\frac{\pi}{2}$ because the relevant domain is a trapeze, not a rectangle. Even if this works, the possible increase in the $\varepsilon$-order is just $\frac{1}{4}$ and thus does not lower the critical integer order for $M$. Therefore we use a standard Gagliardo-Nirenberg Inequality on whole $\Omega$, see the computation after \eqref{eq_DC_ACalpha_proof4} below.
	\end{enumerate}
\end{Remark}
\begin{proof}[Proof of Theorem \ref{th_DC_ACalpha}]
	For the proof we can assume w.l.o.g.~$\delta_0=\delta$, otherwise one can simply shrink $\delta$. The continuity of the objects on the left hand side in \eqref{eq_DC_ACalpha_DE} implies that
	\begin{align}\label{eq_DC_ACalpha_T_epsR}
	T_{\varepsilon,\beta,R}:=\sup\,\{\tilde{T}\in(0,T_0]: \eqref{eq_DC_ACalpha_DE}\text{ holds for }\varepsilon,R\text{ and all }T\in(0,\tilde{T}]\}
	\end{align}
	is well-defined for all $\varepsilon\in(0,\varepsilon_0],\beta\geq 0$ and  $T_{\varepsilon,\beta,R}>0$. In the different cases we have to prove:
	\begin{enumerate}
		\item If $M>3$, then there exist $\beta,\varepsilon_1>0$ such that $T_{\varepsilon,\beta,R}=T_0$ for all $\varepsilon\in(0,\varepsilon_1]$.
		\item If $M=3$, then there are $\beta,\tilde{R},\varepsilon_1>0$ such that $T_{\varepsilon,\beta,\tilde{R}}=T_0$ provided that $\varepsilon\in(0,\varepsilon_1]$ and \eqref{eq_DC_ACalpha_uA} is valid for some $\tilde{M}>3$ instead of $M$ and \eqref{eq_DC_ACalpha_u0} holds with $R$ replaced by $\tilde{R}$.
		\item If $M=3$, then there exist $T_1,\varepsilon_1>0$ such that $T_{\varepsilon,0,R}\geq T_1$ for all $\varepsilon\in(0,\varepsilon_1]$.
	\end{enumerate}
	
	We do a general computation first and consider the distinct cases later. The difference of the left hand sides in \eqref{eq_ACalpha1} for $u_{\varepsilon,\alpha}$ and $u^A_{\varepsilon,\alpha}$ gives
	\begin{align}\label{eq_DC_ACalpha_proof1}
	\left[\partial_t-\Delta+\frac{1}{\varepsilon^2}f''(u^A_{\varepsilon,\alpha})\right]\overline{u}_{\varepsilon,\alpha}=-r^A_{\varepsilon,\alpha}-r_\varepsilon(u_{\varepsilon,\alpha},u^A_{\varepsilon,\alpha}),
	\end{align} 
	where $r_\varepsilon(u_{\varepsilon,\alpha},u^A_{\varepsilon,\alpha}):=\frac{1}{\varepsilon^2}\left[f'(u_{\varepsilon,\alpha})-f'(u^A_{\varepsilon,\alpha})-f''(u^A_{\varepsilon,\alpha})\overline{u}_{\varepsilon,\alpha}\right]$. Taking the product of \eqref{eq_DC_ACalpha_proof1} with $g_\beta^2\overline{u}_{\varepsilon,\alpha}$ and integrating over $Q_T$ for $T\in(0,T_{\varepsilon,\beta,R}]$ for fixed $\varepsilon\in(0,\varepsilon_0]$ and $\beta\geq 0$ yields
	\begin{align}\label{eq_DC_ACalpha_proof2}
	\int_0^T g _\beta^2\int_\Omega\overline{u}_{\varepsilon,\alpha}\left[\partial_t-\Delta+\frac{1}{\varepsilon^2}f''(u^A_{\varepsilon,\alpha})\right]\overline{u}_{\varepsilon,\alpha}
	=-\int_0^T g _\beta^2\int_\Omega [r^A_{\varepsilon,\alpha}+r_\varepsilon(u_{\varepsilon,\alpha},u^A_{\varepsilon,\alpha})]\overline{u}_{\varepsilon,\alpha}
	\end{align} 
	for all $T\in(0,T_{\varepsilon,\beta,R}]$, $\varepsilon\in(0,\varepsilon_0]$ and $\beta\geq 0$.
	We estimate all terms appropriately. Because of $\frac{1}{2}\partial_t|\overline{u}_{\varepsilon,\alpha}|^2=\overline{u}_{\varepsilon,\alpha}\partial_t\overline{u}_{\varepsilon,\alpha}$, integration by parts in time and $\partial_tg_\beta=-\beta g_\beta$ we get
	\[
	\int_0^T\int_\Omega g_\beta^2\partial_t\overline{u}_{\varepsilon,\alpha}\overline{u}_{\varepsilon,\alpha}\,dx\,dt=\frac{1}{2}g_\beta|_T^2\|\overline{u}_{\varepsilon,\alpha}|_T\|_{L^2(\Omega)}^2-\frac{1}{2}\|\overline{u}_{\varepsilon,\alpha}(0)\|_{L^2(\Omega)}^2+\beta\int_0^T g_\beta^2 \|\overline{u}_{\varepsilon,\alpha}\|_{L^2(\Omega)}^2\,dt,
	\]
	where $\|\overline{u}_{\varepsilon,\alpha}(0)\|_{L^2(\Omega)}^2\leq R^2\varepsilon^{2M}$ due to \eqref{eq_DC_ACalpha_u0} (\enquote{well-prepared initial data}). For the remaining expression on the left hand side in \eqref{eq_DC_ACalpha_proof2} we apply integration by parts in space. This implies
	\begin{align}\notag
	&\int_0^Tg_\beta^2\int_\Omega \overline{u}_{\varepsilon,\alpha}\left[-\Delta+\frac{1}{\varepsilon^2}f''(u^A_{\varepsilon,\alpha})\right]\overline{u}_{\varepsilon,\alpha}\,dx\,dt\\
	&=\int_0^Tg_\beta^2\left[\int_\Omega|\nabla\overline{u}_{\varepsilon,\alpha}|^2+\frac{1}{\varepsilon^2}f''(u^A_{\varepsilon,\alpha})\overline{u}_{\varepsilon,\alpha}^2\,dx+\int_{\partial\Omega}\frac{1}{\varepsilon}\sigma_\alpha''(u^A_{\varepsilon,\alpha})(\tr\,\overline{u}_{\varepsilon,\alpha})^2\,d\Hc^1\right]\,dt\notag\\
	&
	+\int_0^Tg_\beta^2\int_{\partial\Omega}\left[s^A_{\varepsilon,\alpha}+s_{\varepsilon,\alpha}(u_{\varepsilon,\alpha},u^A_{\varepsilon,\alpha})\right]\tr\,\overline{u}_{\varepsilon,\alpha}\,d\Hc^1\,dt,\label{eq_DC_ACalpha_proof_s_eps}
	\end{align}
	where we have set $s_{\varepsilon,\alpha}(u_{\varepsilon,\alpha},u^A_{\varepsilon,\alpha}):=\frac{1}{\varepsilon}\left[\sigma_\alpha'(u_{\varepsilon,\alpha})-\sigma_\alpha'(u^A_{\varepsilon,\alpha})-\sigma_\alpha''(u^A_{\varepsilon,\alpha})\overline{u}_{\varepsilon,\alpha}\right]\!|_{\partial\Omega}$.
	Using requirement 2.~(\enquote{spectral estimate}) in the theorem we obtain that the first integral on the right hand side in the latter display is estimated from below by
	\[
	-C\int_0^T g_\beta^2\|\overline{u}_{\varepsilon,\alpha}\|_{L^2(\Omega)}^2\,dt+\|g_\beta\nabla\overline{u}_{\varepsilon,\alpha}\|_{L^2(Q_T\setminus\Gamma(\delta))}^2+c_0\varepsilon\|g_\beta\nabla_\tau\overline{u}_{\varepsilon,\alpha}\|_{L^2(Q_T\cap\Gamma(\delta))}^2.
	\]
	For the remainder terms with $r^A_{\varepsilon,\alpha}$ and $s^A_{\varepsilon,\alpha}$ we employ \eqref{eq_DC_ACalpha_uA} (\enquote{approximate solution}). Hence
	\[
	\left|\int_0^Tg_\beta^2\left[\int_{\partial\Omega}s^A_{\varepsilon,\alpha}\tr\,\overline{u}_{\varepsilon,\alpha}(t)\,d\Hc^1+\int_\Omega r^A_{\varepsilon,\alpha}\overline{u}_{\varepsilon,\alpha}(t)\,dx\right]dt\right|
	\leq \overline{C}_1R\|g_\beta\|_{L^2(0,T)}\varepsilon^{2M}
	\]
	because of \eqref{eq_DC_ACalpha_DE} for all $T\in(0,T_{\varepsilon,\beta,R}]$ and $\varepsilon\in(0,\varepsilon_0]$, where we have used the Hölder Inequality to estimate $\|g_\beta\|_{L^1(0,T)}\leq \sqrt{T_0}\|g_\beta\|_{L^2(0,T)}$.
	
	We derive an estimate for the $r_\varepsilon$-term in \eqref{eq_DC_ACalpha_proof2} and the $s_{\varepsilon,\alpha}$-term in \eqref{eq_DC_ACalpha_proof_s_eps}. The requirement 1.~(\enquote{uniform boundedness}) in the theorem and Lemma \ref{th_DC_bdd_scal} implies
	\begin{align}\label{eq_DC_ACalpha_proof3}
	\sup_{\varepsilon\in(0,\varepsilon_0]}\left[\|u_{\varepsilon,\alpha}\|_{L^\infty(Q_{T_0})}+\|u^A_{\varepsilon,\alpha}\|_{L^\infty(Q_{T_0})}\right]<\infty.
	\end{align}
	Hence we can use the Taylor Theorem and get
	\begin{align}\begin{split}\label{eq_DC_ACalpha_proof4}
	&\left|\int_0^Tg_\beta^2\left[\int_\Omega r_\varepsilon(u_{\varepsilon,\alpha},u^A_{\varepsilon,\alpha})\overline{u}_{\varepsilon,\alpha}\,dx+\int_{\partial\Omega}s_{\varepsilon,\alpha}(u_{\varepsilon,\alpha},u^A_{\varepsilon,\alpha})\tr\,\overline{u}_{\varepsilon,\alpha}\,d\Hc^1\right]dt\right|\\
	&\leq C\int_0^Tg_\beta^2\left[\frac{1}{\varepsilon^2}\|\overline{u}_{\varepsilon,\alpha}\|_{L^3(\Omega)}^3+\frac{1}{\varepsilon}\|\tr\,\overline{u}_{\varepsilon,\alpha}\|_{L^3(\partial\Omega)}^3\right]dt.\end{split}
	\end{align}
	For the estimate of the $L^3(\Omega)$-norm we use a standard Gagliardo-Nirenberg Inequality on $\Omega$, see Lemma \ref{th_DC_GN} and Remark \ref{th_DC_GN_rem}. This yields due to \eqref{eq_DC_ACalpha_DE}
	\[
	\int_0^T\frac{g_\beta^2}{\varepsilon^2}\|\overline{u}_{\varepsilon,\alpha}\|_{L^3(\Omega)}^3\,dt\leq 
	\int_0^T\frac{g_\beta^2}{\varepsilon^2}\|\overline{u}_{\varepsilon,\alpha}\|_{L^2(\Omega)}^2\|\overline{u}_{\varepsilon,\alpha}\|_{H^1(\Omega)}\,dt\leq CR^3\varepsilon^{2M}\varepsilon^{M-3}\|g_\beta^{-1}\|_{L^2(0,T)}
	\]
	for all $T\in(0,T_{\varepsilon,\beta,R}]$ and $\varepsilon\in(0,\varepsilon_0]$. For the $L^3(\partial\Omega)$-norm in \eqref{eq_DC_ACalpha_proof4} we use the idea from Evans \cite{Evans},~5.10, problem 7 again. Let $\vec{w}\in C^1(\overline{\Omega})$ with $\vec{w}\cdot N_{\partial\Omega}\geq 1$. Then because of $|\overline{u}_{\varepsilon,\alpha}|^3(t)\in C^1(\overline
	\Omega)$ with  $\nabla(|\overline{u}_{\varepsilon,\alpha}|^3)(t)=3\textup{sign}(\overline{u}_{\varepsilon,\alpha})|\overline{u}_{\varepsilon,\alpha}|^2\nabla \overline{u}_{\varepsilon,\alpha}(t)$ we get
	\begin{align*}
	\|\tr\,\overline{u}_{\varepsilon,\alpha}|_t\|_{L^3(\partial\Omega)}^3
	\leq \int_{\partial\Omega}
	\vec{w}\cdot N_{\partial\Omega}|\overline{u}_{\varepsilon,\alpha}|^3|_t\,d\Hc^1
	\leq \int_\Omega \left[|\diverg\vec{w}||\overline{u}_{\varepsilon,\alpha}|^3+3|\overline{u}_{\varepsilon,\alpha}|^2|\nabla\overline{u}_{\varepsilon,\alpha}\cdot\vec{w}|\right]\!|_t dx\\
	\leq C[\|\overline{u}_{\varepsilon,\alpha}\|_{L^3(\Omega)}^3\!+\!\|\overline{u}_{\varepsilon,\alpha}\|_{L^4(\Omega)}^2\!\|\nabla \overline{u}_{\varepsilon,\alpha}\|_{L^2(\Omega)}]|_t
	\leq C[\|\overline{u}_{\varepsilon,\alpha}\|_{L^3(\Omega)}^3\!+\!\|\overline{u}_{\varepsilon,\alpha}\|_{L^2(\Omega)}\!\|\overline{u}_{\varepsilon,\alpha}\|_{H^1(\Omega)}^2]|_t,
	\end{align*}
	where we used the Gagliardo-Nirenberg Inequality for the $L^4(\Omega)$-norm, see Lemma \ref{th_DC_GN} and Remark \ref{th_DC_GN_rem}. With $|\nabla\overline{u}_{\varepsilon,\alpha}|\leq C(|\partial_n\overline{u}_{\varepsilon,\alpha}|+|\nabla_\tau\overline{u}_{\varepsilon,\alpha}|)$ and \eqref{eq_DC_ACalpha_DE} we obtain
	\[
	\frac{1}{\varepsilon}\int_0^T g_\beta^2 \|\tr\,\overline{u}_{\varepsilon,\alpha}\|_{L^3(\partial\Omega)}^3\,dt
	\leq CR^3\varepsilon^{2M}\varepsilon^{M-3}\|g_\beta^{-1}\|_{L^2(0,T)}
	\]
	for all $T\in(0,T_{\varepsilon,\beta,R}]$ and $\varepsilon\in(0,\varepsilon_0]$.
	
	Finally, we estimate $\partial_n\overline{u}_{\varepsilon,\alpha}$. To this end we use $|\partial_n\overline{u}_{\varepsilon,\alpha}|\leq C|\nabla\overline{u}_{\varepsilon,\alpha}|$ and
	\begin{align*}
	&\varepsilon^2\|g_\beta\partial_n\overline{u}_{\varepsilon,\alpha}\|_{L^2(Q_T\cap\Gamma(\delta))}^2
	\leq C\int_0^T g_\beta^2\left[\|\overline{u}_{\varepsilon,\alpha}\|_{L^2(\Omega)}^2+\varepsilon\|\tr\,\overline{u}_{\varepsilon,\alpha}\|_{L^2(\partial\Omega)}^2\right] dt\\
	&+C\varepsilon^2\int_0^T g_\beta^2\left[\int_\Omega|\nabla\overline{u}_{\varepsilon,\alpha}|^2+\frac{1}{\varepsilon^2}f''(u^A_{\varepsilon,\alpha})(\overline{u}_{\varepsilon,\alpha})^2\,dx+\int_{\partial\Omega}\frac{1}{\varepsilon}\sigma_\alpha''(u^A_{\varepsilon,\alpha})(\tr\,\overline{u}_{\varepsilon,\alpha})^2\,d\Hc^1\right] dt
	\end{align*}
	with a constant $C>0$ independent of $\varepsilon$, $T$ and $R$. The second line can be absorbed with $\frac{1}{2}$ of the spectral term above if $\varepsilon\in(0,\varepsilon_1]$ and $\varepsilon_1>0$ is small (independent of $T$, $R$). Moreover, for the $\|\tr\,\overline{u}_{\varepsilon,\alpha}\|_{L^2(\partial\Omega)}^2$-term we use the analogous idea that we applied for the estimate of $\|\tr\,\overline{u}_{\varepsilon,\alpha}\|_{L^3(\partial\Omega)}^3$ above. This yields 
	\[
	\varepsilon\int_0^T g_\beta^2 \|\tr\,\overline{u}_{\varepsilon,\alpha}\|_{L^2(\partial\Omega)}^2 \,dt
	\leq \varepsilon\int_0^T g_\beta^2 \|\overline{u}_{\varepsilon,\alpha}\|_{L^2(\Omega)}\|\overline{u}_{\varepsilon,\alpha}\|_{H^1(\Omega)}\,dt. 
	\]
	Here note that because of \eqref{eq_DC_ACalpha_DE} it follows that for all $T\in(0,T_{\varepsilon,\beta,R}]$ and $\varepsilon\in(0,\varepsilon_0]$
	\[
	\varepsilon^2\int_0^T g_\beta^2\|\overline{u}_{\varepsilon,\alpha}\|_{H^1(\Omega)}^2\,dt\leq CR^2\varepsilon^{2M}.
	\]
    Therefore with the Young Inequality the contribution of the $\|\tr\,\overline{u}_{\varepsilon,\alpha}\|_{L^2(\partial\Omega)}^2$-term is controlled by
    \[
    \tilde{C}\int_0^T g_\beta^2\|\overline{u}_{\varepsilon,\alpha}\|_{L^2(\Omega)}^2\,dt+ \frac{1}{8}R^2\varepsilon^{2M}
    \]
    for all $T\in(0,T_{\varepsilon,\beta,R}]$ and $\varepsilon\in(0,\varepsilon_0]$ with some $\tilde{C}>0$ large.
	
	Altogether we have
	\begin{align}\notag
	&\frac{1}{2}g_\beta(T)\|\overline{u}_{\varepsilon,\alpha}(T)\|_{L^2(\Omega)}^2
	+\frac{1}{2}\|g_\beta\nabla\overline{u}_{\varepsilon,\alpha}\|_{L^2(Q_T\setminus\Gamma(\delta))}^2\\
	&+\frac{1}{2}c_0\varepsilon\|g_\beta\nabla_\tau\overline{u}_{\varepsilon,\alpha}\|_{L^2(Q_T\cap\Gamma(\delta))}^2
	+\frac{1}{2}\varepsilon^2\|g_\beta\partial_n\overline{u}_{\varepsilon,\alpha}\|_{L^2(Q_T\cap\Gamma(\delta))}^2\notag\\
	\begin{split}&\leq
	(\frac{1}{2}+\frac{1}{8})R^2\varepsilon^{2M}+\int_0^T(-\beta+\overline{C}_0)g_\beta^2\|\overline{u}_{\varepsilon,\alpha}(t)\|_{L^2(\Omega)}^2\,dt+\overline{C}_1R\varepsilon^{2M}\|g_\beta\|_{L^2(0,T)}\\
	&+CR^3\varepsilon^{2M} \varepsilon^{M-3}\|g_\beta^{-1}\|_{L^2(0,T)}\label{eq_DC_ACalpha_DE2}\end{split}
	\end{align}
	for all $T\in(0,T_{\varepsilon,\beta,R}]$, $\varepsilon\in(0,\varepsilon_1]$ and some constants $\overline{C}_0,\overline{C}_1,C>0$ independent of $\varepsilon,T,R$.
	
	Now let us look at the different cases in the theorem.
	
	\begin{proof}[Ad 1] 
		If $M>3$, then we choose $\beta\geq \overline{C}_0$ large such that $\overline{C}_1R\|g_\beta\|_{L^2(0,T_0)}\leq\frac{R^2}{8}$.
		Therefore \eqref{eq_DC_ACalpha_DE2} is controlled by $\frac{7}{8}R^2\varepsilon^{2M}$ for all $T\in (0,T_{\varepsilon,\beta,R}]$ and $\varepsilon\in(0,\varepsilon_1]$, if $\varepsilon_1>0$ is small. By contradiction and continuity this yields $T_{\varepsilon,\beta,R}=T_0$ for all $\varepsilon\in(0,\varepsilon_1]$.\qedhere$_{1.}$
	\end{proof}
	
	\begin{proof}[Ad 2] 
		Let $M=3$ and let \eqref{eq_DC_ACalpha_uA} hold for some $\tilde{M}>M$ instead of $M$. Then the term in \eqref{eq_DC_ACalpha_DE2} where $R$ enters linearly is enhanced by the factor $\varepsilon^{\tilde{M}-M}$. We fix $\beta\geq\overline{C}_0$ and choose $R>0$ small such that the $R^3$-term in \eqref{eq_DC_ACalpha_DE2} is estimated by $\frac{1}{8}R^2\varepsilon^{2M}$. Then $\varepsilon_1>0$ can be taken small such that \eqref{eq_DC_ACalpha_DE2} is bounded by $\frac{7}{8}R^2\varepsilon^{2M}$ for all $T\in (0,T_{\varepsilon,\beta,R}]$ and $\varepsilon\in(0,\varepsilon_1]$. Via contradiction and continuity we obtain $T_{\varepsilon,\beta,R}=T_0$ for all $\varepsilon\in(0,\varepsilon_1]$.\qedhere$_{2.}$
	\end{proof}
	
	\begin{proof}[Ad 3] 
		Let $M=3$ and $\beta=0$. Then \eqref{eq_DC_ACalpha_DE2} is controlled by
		\[
		\left[(\frac{1}{2}+\frac{1}{8})R^2+CR^2T+CRT^{\frac{1}{2}}+CR^3 T^{\frac{1}{2}}\right]\varepsilon^{2M}.
		\]
		There are $\varepsilon_1,T_1>0$ such that this is estimated by $\frac{7}{8}R^2\varepsilon^{2M}$ for all $T\in(0,\min(T_{\varepsilon,\beta,R},T_1)]$ and $\varepsilon\in(0,\varepsilon_1]$. Hence $T_{\varepsilon,0,R}\geq T_1$ for all $\varepsilon\in(0,\varepsilon_1]$.
		\qedhere$_{3.}$\end{proof}
	
	The proof of Theorem \ref{th_DC_ACalpha} is completed.
\end{proof}

\subsubsection{Proof of Theorem \ref{th_ACalpha_conv}}\label{sec_DC_ACalpha_conv}
Let $f$ satisfy \eqref{eq_AC_fvor1}-\eqref{eq_AC_fvor2} and $\sigma_\alpha$ for $\alpha\in(0,\pi)$ be as in Definition \ref{th_ACalpha_sigma_def}. Then let $\alpha_0>0$ be as in Remark \ref{th_asym_ACalpha_decay_param} and $\overline{\alpha}_0\in(0,\alpha_0]$ such that Theorem \ref{th_SE_ACalpha} holds. Moreover, let $\Omega$, $Q_T$ and $\partial Q_T$ be as in Remark \ref{th_intro_coord},~1. Additionally, let $\Gamma=(\Gamma_t)_{t\in[0,T_0]}$ for some $T_0>0$ be a smooth solution to \eqref{MCF} with $\alpha$-contact angle condition parametrized as in Section \ref{sec_coord_surface_requ} for some $\alpha\in\frac{\pi}{2}+[-\overline{\alpha}_0,\overline{\alpha}_0]$ and let $\delta>0$ be such that Theorem \ref{th_coord2D} holds for $\delta$ replaced by $2\delta$. We use the notation from Section \ref{sec_coord_surface_requ} and Section \ref{sec_coord2D}. Furthermore, let $\delta_0\in(0,\delta]$ be such that \eqref{eq_asym_ACalpha_uA_delta0} holds. Moreover, let $M\in\N$ with $M\geq 3$ and let $(u^A_{\varepsilon,\alpha})_{\varepsilon>0}$ be the approximate solution on $\overline{Q_{T_0}}$ defined in Section \ref{sec_asym_ACalpha_uA} (which was constructed via asymptotic expansions in Section \ref{sec_asym_ACalpha}) and let $\varepsilon_0>0$ be such that Lemma \ref{th_asym_ACalpha_uA} (\enquote{remainder estimate}) holds for $\varepsilon\in(0,\varepsilon_0]$. The property $\lim_{\varepsilon\rightarrow 0} u^A_{\varepsilon,\alpha}=\pm 1$ uniformly on compact subsets of $Q_{T_0}^\pm$ holds due to the construction in Section \ref{sec_asym_ACalpha}. 

Theorem \ref{th_ACalpha_conv} follows directly from Theorem \ref{th_DC_ACalpha} if we prove the conditions 1.-4.~in Theorem \ref{th_DC_ACalpha}. The requirement 1.~(\enquote{uniform boundedness}) is satisfied because of Lemma \ref{th_asym_ACalpha_uA} for $u^A_{\varepsilon,\alpha}$ and for $u_{0,\varepsilon,\alpha}$ this is required in Theorem \ref{th_ACalpha_conv}. Condition 2.~(\enquote{spectral estimate}) follows from to Theorem \ref{th_SE_ACalpha}. Assumption 4.~(\enquote{well prepared initial data}) is a requirement for $u_{0,\varepsilon,\alpha}$ and holds in the situation of Theorem \ref{th_ACalpha_conv}. It is left to prove 3.~(\enquote{approximate solution}). This is similar to the case $\alpha=\frac{\pi}{2}$, cf.~the proof of Theorem 1.1 in \cite{AbelsMoser}. 

First we estimate the boundary term in \eqref{eq_DC_ACalpha_uA}. Lemma \ref{th_asym_ACalpha_uA} implies $s^A_{\varepsilon,\alpha}=0$ on $\partial\Omega\setminus\Gamma_t(2\delta)$ and $|s^A_{\varepsilon,\alpha}|\leq C\varepsilon^Me^{-c|\rho_{\varepsilon,\alpha}|}$, where $\rho_{\varepsilon,\alpha}$ is from \eqref{eq_asym_ACalpha_rho}. Hence
\[
\left|\int_{\partial\Omega}s^A_{\varepsilon,\alpha}\tr\,\overline{u}_{\varepsilon,\alpha}(t)\,d\Hc^1\right|\leq \|s^A_{\varepsilon,\alpha}\|_{L^2(\partial\Omega\cap\Gamma_t(2\delta))}
\|\tr\,\overline{u}_{\varepsilon,\alpha}(t)\|_{L^2(\partial\Omega\cap\Gamma_t(2\delta))}.
\]
Due to the substitution rule and a scaling argument with \cite{MoserACvACND}, Lemma 6.5, we obtain the estimate $\|s^A_{\varepsilon,\alpha}\|_{L^2(\partial\Omega\cap\Gamma_t(2\delta))}\leq C\varepsilon^{M+\frac{1}{2}}$. Moreover, analogously to Lemma \ref{th_SE_ACalpha_intpol_tr1} it follows that
\[
\|\tr\,\overline{u}_{\varepsilon,\alpha}(t)\|_{L^2(\partial\Omega\cap\Gamma_t(2\delta))}
\leq C
(\|\overline{u}_{\varepsilon,\alpha}(t)\|_{L^2(\Gamma_t(2\delta))}
+\|\nabla_\tau\overline{u}_{\varepsilon,\alpha}(t)\|_{L^2(\Gamma_t(2\delta))}).
\]
Because of $|\nabla_\tau\overline{u}_{\varepsilon,\alpha}|\leq C|\nabla\overline{u}_{\varepsilon,\alpha}|$, the estimate for the $s^A_{\varepsilon,\alpha}$-term in \eqref{eq_DC_ACalpha_uA} follows. 

Finally, we consider the $r^A_{\varepsilon,\alpha}$-term in \eqref{eq_DC_ACalpha_uA}. By Lemma \ref{th_asym_ACalpha_uA} we get $r^A_{\varepsilon,\alpha}=0$ in $\Omega\setminus\Gamma_t(2\delta)$ and
\begin{alignat*}{2}
|r^A_{\varepsilon,\alpha}|
\leq 
C(\varepsilon^{M-1} e^{-c(|\rho_{\varepsilon,\alpha}|+Z_{\varepsilon,\alpha}^\pm)}
+\varepsilon^M e^{-c|\rho_{\varepsilon,\alpha}|}+\varepsilon^{M+1})
\quad \text{ in }\Gamma^\pm(2\delta,1).
\end{alignat*}
An integral transformation yields
\[
\left|\int_\Omega r^A_{\varepsilon,\alpha}\overline{u}_{\varepsilon,\alpha}(t)\,dx\right|
\leq\int_{\Gamma_t(2\delta)}|r^A_{\varepsilon,\alpha}\overline{u}_{\varepsilon,\alpha}(t)|\,dx=\int_{S_{2\delta,\alpha}}|r^A_{\varepsilon,\alpha}\overline{u}_{\varepsilon,\alpha}|_{\overline{X}(r,s,t)}| J_t(r,s)\,d(r,s),
\]
where $S_{2\delta,\alpha}$ is as in \eqref{eq_coord2D_trapeze} and $J_t$ is uniformly bounded in $t\in[0,T_0]$ by Remark \ref{th_coord2D_rem2}, 3. We choose $\mu>0$ such that for $I_\mu:=(-1-\mu,1+\mu)$ it holds $S_{2\delta,\alpha}\subseteq(-2\delta,2\delta)\times I_\mu$. Moreover, we denote with $e_0(\overline{u}_{\varepsilon,\alpha}|_{\overline{X}})$ the extension of $\overline{u}_{\varepsilon,\alpha}|_{\overline{X}}$ by zero to $(-2\delta,2\delta)\times I_\mu$. With a scaling argument and $|\frac{r}{\varepsilon}|+|\frac{s^\pm}{\varepsilon}|\leq C(|\rho_{\varepsilon,\alpha}|+Z_{\varepsilon,\alpha}^\pm+1)$ because of \eqref{eq_asym_ACalpha_cp_s1} it follows that 
\[
\left|\int_\Omega r^A_{\varepsilon,\alpha}\overline{u}_{\varepsilon,\alpha}(t)\,dx\right|\leq
C\varepsilon^{M-\frac{1}{2}}\int_{-1-\mu}^{1+\mu} \|e_0(\overline{u}_{\varepsilon,\alpha}|_{\overline{X}})(.,s,t)\|_{L^2(-2\delta,2\delta)}\left[\sum_\pm e^{-c\frac{|s\mp 1|}{\varepsilon}}+\varepsilon\right]\,ds.
\]
Note that $H^1(-s,s,B)\hookrightarrow L^\infty(-s,s,B)$ for all $s\in[1,1+\mu]$ and any Banach space $B$ with uniform embedding constant. For $B$ we use $L^2$-spaces over suitable intervals $I(s)$ (possibly empty for $|s|$ large) with $\bigcup_{s\in(1,1+\mu)}I(s)\times(-s,s)=(S_{2\delta,\alpha})^\circ$. Hence the features of Sobolev spaces on product sets, cf.~\cite{MoserACvACND}, Lemma 2.10, yield 
\[
\|e_0(\overline{u}_{\varepsilon,\alpha}|_{\overline{X}})(.,t)\|_{L^\infty(I_\mu,L^2(-2\delta,2\delta))}\leq C(\|\overline{u}_{\varepsilon,\alpha}|_{\overline{X}}\|_{L^2(S_{2\delta,\alpha})}+\|\nabla_\tau\overline{u}_{\varepsilon,\alpha}|_{\overline{X}}\|_{L^2(S_{2\delta,\alpha})}).
\] 
With a scaling argument for the exponential term and an integral transformation we obtain
\[
\left|\int_\Omega r^A_{\varepsilon,\alpha}\overline{u}_{\varepsilon,\alpha}(t)\,dx\right|\leq C\varepsilon^{M+\frac{1}{2}}(\|\overline{u}_{\varepsilon,\alpha}(t)\|_{L^2(\Gamma_t(2\delta))}+\|\nabla_\tau\overline{u}_{\varepsilon,\alpha}(t)\|_{L^2(\Gamma_t(2\delta))}).
\]
Since $|\nabla_\tau\overline{u}_{\varepsilon,\alpha}|\leq C|\nabla\overline{u}_{\varepsilon,\alpha}|$, this shows \eqref{eq_DC_ACalpha_uA}. Hence Theorem \ref{th_ACalpha_conv} is proven.\hfill$\square$\\
\newline
\textit{Acknowledgments.} The second author gratefully acknowledges support through DFG, GRK 1692 \enquote{Curvature, Cycles and Cohomology} during parts of the work.

\makeatletter
\renewenvironment{thebibliography}[1]
{\section{\bibname}% <-- this line was changed from \chapter* to \section*
	\@mkboth{\MakeUppercase\bibname}{\MakeUppercase\bibname}%
	\list{\@biblabel{\@arabic\c@enumiv}}%
	{\settowidth\labelwidth{\@biblabel{#1}}%
		\leftmargin\labelwidth
		\advance\leftmargin\labelsep
		\@openbib@code
		\usecounter{enumiv}%
		\let\p@enumiv\@empty
		\renewcommand\theenumiv{\@arabic\c@enumiv}}%
	\sloppy
	\clubpenalty4000
	\@clubpenalty \clubpenalty
	\widowpenalty4000%
	\sfcode`\.\@m}
{\def\@noitemerr
	{\@latex@warning{Empty `thebibliography' environment}}%
	\endlist}
\makeatother

\footnotesize

\bibliographystyle{siam}

\end{document}

%% file: double_well.pdf_tex
%% Accompanies image file 'double_well.pdf' (pdf, eps, ps)
\begingroup%
  \makeatletter%
  \providecommand\color[2][]{%
    \errmessage{(Inkscape) Color is used for the text in Inkscape, but the package 'color.sty' is not loaded}%
    \renewcommand\color[2][]{}%
  }%
  \providecommand\transparent[1]{%
    \errmessage{(Inkscape) Transparency is used (non-zero) for the text in Inkscape, but the package 'transparent.sty' is not loaded}%
    \renewcommand\transparent[1]{}%
  }%
  \providecommand\rotatebox[2]{#2}%
  \newcommand*\fsize{\dimexpr\f@size pt\relax}%
  \newcommand*\lineheight[1]{\fontsize{\fsize}{#1\fsize}\selectfont}%
  \ifx\svgwidth\undefined%
    \setlength{\unitlength}{1035.59767355bp}%
    \ifx\svgscale\undefined%
      \relax%
    \else%
      \setlength{\unitlength}{\unitlength * \real{\svgscale}}%
    \fi%
  \else%
    \setlength{\unitlength}{\svgwidth}%
  \fi%
  \global\let\svgwidth\undefined%
  \global\let\svgscale\undefined%
  \makeatother%
  \begin{picture}(1,0.7549303)%
    \lineheight{1}%
    \setlength\tabcolsep{0pt}%
    \put(0,0){\includegraphics[width=\unitlength,page=1]{double_well.pdf}}%
    \put(0.19716394,0.0405114){\color[rgb]{0,0,0}\makebox(0,0)[lt]{\begin{minipage}{0.11705153\unitlength}\raggedright $-1$\\ \end{minipage}}}%
    \put(0.77789796,0.03996379){\color[rgb]{0,0,0}\makebox(0,0)[lt]{\begin{minipage}{0.12144096\unitlength}\raggedright $1$\end{minipage}}}%
    \put(0,0){\includegraphics[width=\unitlength,page=2]{double_well.pdf}}%
    \put(0.52422331,0.82492551){\color[rgb]{0,0,0}\makebox(0,0)[lt]{\begin{minipage}{0.06481274\unitlength}\raggedright \end{minipage}}}%
    \put(0.52278033,0.74102212){\color[rgb]{0,0,0}\makebox(0,0)[lt]{\begin{minipage}{0.16723281\unitlength}\raggedright $f(u)$\end{minipage}}}%
    \put(1.00464891,0.05445472){\color[rgb]{0,0,0}\makebox(0,0)[lt]{\begin{minipage}{0.05601102\unitlength}\raggedright $u$\end{minipage}}}%
    \put(0,0){\includegraphics[width=\unitlength,page=3]{double_well.pdf}}%
  \end{picture}%
\endgroup%

%% file: sigma_alpha.pdf_tex
%% Accompanies image file 'sigma_alpha.pdf' (pdf, eps, ps)
\begingroup%
  \makeatletter%
  \providecommand\color[2][]{%
    \errmessage{(Inkscape) Color is used for the text in Inkscape, but the package 'color.sty' is not loaded}%
    \renewcommand\color[2][]{}%
  }%
  \providecommand\transparent[1]{%
    \errmessage{(Inkscape) Transparency is used (non-zero) for the text in Inkscape, but the package 'transparent.sty' is not loaded}%
    \renewcommand\transparent[1]{}%
  }%
  \providecommand\rotatebox[2]{#2}%
  \newcommand*\fsize{\dimexpr\f@size pt\relax}%
  \newcommand*\lineheight[1]{\fontsize{\fsize}{#1\fsize}\selectfont}%
  \ifx\svgwidth\undefined%
    \setlength{\unitlength}{2156.97711614bp}%
    \ifx\svgscale\undefined%
      \relax%
    \else%
      \setlength{\unitlength}{\unitlength * \real{\svgscale}}%
    \fi%
  \else%
    \setlength{\unitlength}{\svgwidth}%
  \fi%
  \global\let\svgwidth\undefined%
  \global\let\svgscale\undefined%
  \makeatother%
  \begin{picture}(1,0.26128054)%
    \lineheight{1}%
    \setlength\tabcolsep{0pt}%
    \put(0,0){\includegraphics[width=\unitlength,page=1]{sigma_alpha.pdf}}%
    \put(0.06311544,0.01423958){\color[rgb]{0,0,0}\makebox(0,0)[lt]{\begin{minipage}{0.06845518\unitlength}\raggedright $-1$\end{minipage}}}%
    \put(0.40200875,0.01428642){\color[rgb]{0,0,0}\makebox(0,0)[lt]{\begin{minipage}{0.04763679\unitlength}\raggedright $1$\end{minipage}}}%
    \put(1.00146126,0.01836889){\color[rgb]{0,0,0}\makebox(0,0)[lt]{\begin{minipage}{0.02825134\unitlength}\raggedright $u$\end{minipage}}}%
    \put(0.77657572,0.25615369){\color[rgb]{0,0,0}\makebox(0,0)[lt]{\begin{minipage}{0.15197811\unitlength}\raggedright $\sigma_\alpha'(u)$\end{minipage}}}%
    \put(0.24732425,0.25470224){\color[rgb]{0,0,0}\makebox(0,0)[lt]{\begin{minipage}{0.172339\unitlength}\raggedright $\sigma_\alpha(u)$\end{minipage}}}%
    \put(0,0){\includegraphics[width=\unitlength,page=2]{sigma_alpha.pdf}}%
    \put(0.58850906,0.01704927){\color[rgb]{0,0,0}\makebox(0,0)[lt]{\begin{minipage}{0.06845518\unitlength}\raggedright $-1$\end{minipage}}}%
    \put(0.93010754,0.01766671){\color[rgb]{0,0,0}\makebox(0,0)[lt]{\begin{minipage}{0.04763679\unitlength}\raggedright $1$\end{minipage}}}%
    \put(0.47337847,0.01945548){\color[rgb]{0,0,0}\makebox(0,0)[lt]{\begin{minipage}{0.02825134\unitlength}\raggedright $u$\end{minipage}}}%
    \put(0,0){\includegraphics[width=\unitlength,page=3]{sigma_alpha.pdf}}%
  \end{picture}%
\endgroup%

%% file: transition_zone_alpha.pdf_tex
%% Accompanies image file 'transition_zone_alpha.pdf' (pdf, eps, ps)
\begingroup%
  \makeatletter%
  \providecommand\color[2][]{%
    \errmessage{(Inkscape) Color is used for the text in Inkscape, but the package 'color.sty' is not loaded}%
    \renewcommand\color[2][]{}%
  }%
  \providecommand\transparent[1]{%
    \errmessage{(Inkscape) Transparency is used (non-zero) for the text in Inkscape, but the package 'transparent.sty' is not loaded}%
    \renewcommand\transparent[1]{}%
  }%
  \providecommand\rotatebox[2]{#2}%
  \newcommand*\fsize{\dimexpr\f@size pt\relax}%
  \newcommand*\lineheight[1]{\fontsize{\fsize}{#1\fsize}\selectfont}%
  \ifx\svgwidth\undefined%
    \setlength{\unitlength}{1632.26281834bp}%
    \ifx\svgscale\undefined%
      \relax%
    \else%
      \setlength{\unitlength}{\unitlength * \real{\svgscale}}%
    \fi%
  \else%
    \setlength{\unitlength}{\svgwidth}%
  \fi%
  \global\let\svgwidth\undefined%
  \global\let\svgscale\undefined%
  \makeatother%
  \begin{picture}(1,0.28761466)%
    \lineheight{1}%
    \setlength\tabcolsep{0pt}%
    \put(0,0){\includegraphics[width=\unitlength,page=1]{transition_zone_alpha.pdf}}%
    \put(0.01088484,0.13841001){\color[rgb]{0,0,0}\makebox(0,0)[lt]{\begin{minipage}{0.30802447\unitlength}\raggedright $u_{\varepsilon,\alpha}(.,t)$\\ $\approx -1$\end{minipage}}}%
    \put(0.24313031,0.22257239){\color[rgb]{0,0,0}\makebox(0,0)[lt]{\begin{minipage}{0.35295126\unitlength}\raggedright $u_{\varepsilon,\alpha}(.,t)\approx 1$\end{minipage}}}%
    \put(0.20863257,0.02572724){\color[rgb]{0,0,0}\makebox(0,0)[lt]{\begin{minipage}{0.34426892\unitlength}\raggedright diffuse interface $\sim\varepsilon$\end{minipage}}}%
    \put(0.34870786,0.06912518){\color[rgb]{0,0,0}\makebox(0,0)[lt]{\begin{minipage}{0.10043025\unitlength}\raggedright $\Omega$\end{minipage}}}%
    \put(0,0){\includegraphics[width=\unitlength,page=2]{transition_zone_alpha.pdf}}%
    \put(0.89307509,0.06179209){\color[rgb]{0,0,0}\makebox(0,0)[lt]{\begin{minipage}{0.10043025\unitlength}\raggedright $\Omega$\end{minipage}}}%
    \put(0.4643639,0.15705292){\color[rgb]{0,0,0}\makebox(0,0)[lt]{\begin{minipage}{0.34716552\unitlength}\raggedright $\overset{\varepsilon\rightarrow 0}{\longrightarrow}$\end{minipage}}}%
    \put(0.72345937,0.21900667){\color[rgb]{0,0,0}\makebox(0,0)[lt]{\begin{minipage}{0.13521973\unitlength}\raggedright $\Gamma_t$\end{minipage}}}%
    \put(0,0){\includegraphics[width=\unitlength,page=3]{transition_zone_alpha.pdf}}%
    \put(0.74774526,0.24971255){\color[rgb]{0,0,0}\makebox(0,0)[lt]{\lineheight{1.25}\smash{\begin{tabular}[t]{l}$\alpha$\end{tabular}}}}%
    \put(0,0){\includegraphics[width=\unitlength,page=4]{transition_zone_alpha.pdf}}%
    \put(0.77104362,0.06315952){\color[rgb]{0,0,0}\makebox(0,0)[lt]{\lineheight{1.25}\smash{\begin{tabular}[t]{l}$\alpha$\end{tabular}}}}%
  \end{picture}%
\endgroup%

%% file: TubularNeighbNotation.pdf_tex
%% Accompanies image file 'TubularNeighbNotation.pdf' (pdf, eps, ps)
\begingroup%
  \makeatletter%
  \providecommand\color[2][]{%
    \errmessage{(Inkscape) Color is used for the text in Inkscape, but the package 'color.sty' is not loaded}%
    \renewcommand\color[2][]{}%
  }%
  \providecommand\transparent[1]{%
    \errmessage{(Inkscape) Transparency is used (non-zero) for the text in Inkscape, but the package 'transparent.sty' is not loaded}%
    \renewcommand\transparent[1]{}%
  }%
  \providecommand\rotatebox[2]{#2}%
  \newcommand*\fsize{\dimexpr\f@size pt\relax}%
  \newcommand*\lineheight[1]{\fontsize{\fsize}{#1\fsize}\selectfont}%
  \ifx\svgwidth\undefined%
    \setlength{\unitlength}{1508.78652498bp}%
    \ifx\svgscale\undefined%
      \relax%
    \else%
      \setlength{\unitlength}{\unitlength * \real{\svgscale}}%
    \fi%
  \else%
    \setlength{\unitlength}{\svgwidth}%
  \fi%
  \global\let\svgwidth\undefined%
  \global\let\svgscale\undefined%
  \makeatother%
  \begin{picture}(1,0.33416911)%
    \lineheight{1}%
    \setlength\tabcolsep{0pt}%
    \put(0,0){\includegraphics[width=\unitlength,page=1]{TubularNeighbNotation.pdf}}%
    \put(0.69617708,0.23027269){\color[rgb]{0,0,0}\makebox(0,0)[lt]{\begin{minipage}{0.24846337\unitlength}\raggedright $\Gamma_t$\end{minipage}}}%
    \put(0.71456951,0.02862498){\color[rgb]{0,0,0}\makebox(0,0)[lt]{\begin{minipage}{0.11628105\unitlength}\raggedright \end{minipage}}}%
    \put(0.66796816,0.02053619){\color[rgb]{0,0,0}\makebox(0,0)[lt]{\begin{minipage}{0.25106745\unitlength}\raggedright $\Gamma_t(2\delta)$\end{minipage}}}%
    \put(0.88993015,0.07012571){\color[rgb]{0,0,0}\makebox(0,0)[lt]{\begin{minipage}{0.11147396\unitlength}\raggedright $\Omega$\end{minipage}}}%
    \put(0,0){\includegraphics[width=\unitlength,page=2]{TubularNeighbNotation.pdf}}%
    \put(0.81930142,0.13305747){\color[rgb]{0,0,0}\makebox(0,0)[lt]{\begin{minipage}{0.03188552\unitlength}\raggedright \end{minipage}}}%
    \put(0.76587991,0.09095668){\color[rgb]{0,0,0}\makebox(0,0)[lt]{\begin{minipage}{0.03816226\unitlength}\raggedright $r$\end{minipage}}}%
    \put(0.70821147,0.1270984){\color[rgb]{0,0,0}\makebox(0,0)[lt]{\begin{minipage}{0.03816226\unitlength}\raggedright $s$\end{minipage}}}%
    \put(0,0){\includegraphics[width=\unitlength,page=3]{TubularNeighbNotation.pdf}}%
    \put(0.17970555,0.22767765){\color[rgb]{0,0,0}\makebox(0,0)[lt]{\begin{minipage}{0.24846337\unitlength}\raggedright $\Gamma_t$\end{minipage}}}%
    \put(0.38341264,0.07899184){\color[rgb]{0,0,0}\makebox(0,0)[lt]{\begin{minipage}{0.11147396\unitlength}\raggedright $\Omega$\end{minipage}}}%
    \put(0,0){\includegraphics[width=\unitlength,page=4]{TubularNeighbNotation.pdf}}%
    \put(0.33198391,0.02236895){\color[rgb]{0,0,0}\makebox(0,0)[lt]{\begin{minipage}{0.23895738\unitlength}\raggedright $N_{\partial\Omega}$\end{minipage}}}%
    \put(0,0){\includegraphics[width=\unitlength,page=5]{TubularNeighbNotation.pdf}}%
    \put(0.20299912,0.08734316){\color[rgb]{0,0,0}\makebox(0,0)[lt]{\begin{minipage}{0.0926439\unitlength}\raggedright $\alpha$\end{minipage}}}%
    \put(0,0){\includegraphics[width=\unitlength,page=6]{TubularNeighbNotation.pdf}}%
    \put(0.16859311,0.28759364){\color[rgb]{0,0,0}\makebox(0,0)[lt]{\begin{minipage}{0.10243568\unitlength}\raggedright $\alpha$\end{minipage}}}%
    \put(0,0){\includegraphics[width=\unitlength,page=7]{TubularNeighbNotation.pdf}}%
    \put(0.24719523,0.16508643){\color[rgb]{0,0,0}\makebox(0,0)[lt]{\begin{minipage}{0.1622638\unitlength}\raggedright $\vec{n}(.,t)$\end{minipage}}}%
    \put(0,0){\includegraphics[width=\unitlength,page=8]{TubularNeighbNotation.pdf}}%
  \end{picture}%
\endgroup%

%% file: trapez.pdf_tex
%% Accompanies image file 'trapez.pdf' (pdf, eps, ps)
\begingroup%
  \makeatletter%
  \providecommand\color[2][]{%
    \errmessage{(Inkscape) Color is used for the text in Inkscape, but the package 'color.sty' is not loaded}%
    \renewcommand\color[2][]{}%
  }%
  \providecommand\transparent[1]{%
    \errmessage{(Inkscape) Transparency is used (non-zero) for the text in Inkscape, but the package 'transparent.sty' is not loaded}%
    \renewcommand\transparent[1]{}%
  }%
  \providecommand\rotatebox[2]{#2}%
  \newcommand*\fsize{\dimexpr\f@size pt\relax}%
  \newcommand*\lineheight[1]{\fontsize{\fsize}{#1\fsize}\selectfont}%
  \ifx\svgwidth\undefined%
    \setlength{\unitlength}{607.19458873bp}%
    \ifx\svgscale\undefined%
      \relax%
    \else%
      \setlength{\unitlength}{\unitlength * \real{\svgscale}}%
    \fi%
  \else%
    \setlength{\unitlength}{\svgwidth}%
  \fi%
  \global\let\svgwidth\undefined%
  \global\let\svgscale\undefined%
  \makeatother%
  \begin{picture}(1,0.95206802)%
    \lineheight{1}%
    \setlength\tabcolsep{0pt}%
    \put(0,0){\includegraphics[width=\unitlength,page=1]{trapez.pdf}}%
    \put(0.77255863,0.32644652){\color[rgb]{0,0,0}\makebox(0,0)[lt]{\begin{minipage}{0.53153266\unitlength}\raggedright $S_{\delta,\alpha}$\end{minipage}}}%
    \put(0.62188547,0.15023565){\color[rgb]{0,0,0}\makebox(0,0)[lt]{\begin{minipage}{0.40611817\unitlength}\raggedright $S_{\delta,\alpha}^-$\end{minipage}}}%
    \put(0.5861821,0.89440785){\color[rgb]{0,0,0}\makebox(0,0)[lt]{\begin{minipage}{0.60140546\unitlength}\raggedright $S_{\delta,\alpha}^+$\end{minipage}}}%
    \put(0.50623765,0.36786608){\color[rgb]{0,0,0}\makebox(0,0)[lt]{\begin{minipage}{0.23276495\unitlength}\raggedright $r=0$\end{minipage}}}%
    \put(0.09323505,0.85080459){\color[rgb]{0,0,0}\makebox(0,0)[lt]{\begin{minipage}{0.23156497\unitlength}\raggedright $s=1$\end{minipage}}}%
    \put(0.0495278,0.12653754){\color[rgb]{0,0,0}\makebox(0,0)[lt]{\begin{minipage}{0.23156497\unitlength}\raggedright $s=-1$\end{minipage}}}%
    \put(0,0){\includegraphics[width=\unitlength,page=2]{trapez.pdf}}%
    \put(0.53974992,0.23245391){\color[rgb]{0,0,0}\makebox(0,0)[lt]{\begin{minipage}{0.23706854\unitlength}\raggedright $\alpha$\end{minipage}}}%
    \put(0,0){\includegraphics[width=\unitlength,page=3]{trapez.pdf}}%
    \put(0.5460326,0.75450375){\color[rgb]{0,0,0}\makebox(0,0)[lt]{\begin{minipage}{0.23706854\unitlength}\raggedright $\alpha$\end{minipage}}}%
    \put(0,0){\includegraphics[width=\unitlength,page=4]{trapez.pdf}}%
    \put(0.34481658,0.16706773){\color[rgb]{0,0,0}\makebox(0,0)[lt]{\begin{minipage}{0.23333964\unitlength}\raggedright $\delta$\end{minipage}}}%
  \end{picture}%
\endgroup%

%% file: opt_profile.pdf_tex
%% Accompanies image file 'opt_profile.pdf' (pdf, eps, ps)
\begingroup%
  \makeatletter%
  \providecommand\color[2][]{%
    \errmessage{(Inkscape) Color is used for the text in Inkscape, but the package 'color.sty' is not loaded}%
    \renewcommand\color[2][]{}%
  }%
  \providecommand\transparent[1]{%
    \errmessage{(Inkscape) Transparency is used (non-zero) for the text in Inkscape, but the package 'transparent.sty' is not loaded}%
    \renewcommand\transparent[1]{}%
  }%
  \providecommand\rotatebox[2]{#2}%
  \newcommand*\fsize{\dimexpr\f@size pt\relax}%
  \newcommand*\lineheight[1]{\fontsize{\fsize}{#1\fsize}\selectfont}%
  \ifx\svgwidth\undefined%
    \setlength{\unitlength}{2161.57943605bp}%
    \ifx\svgscale\undefined%
      \relax%
    \else%
      \setlength{\unitlength}{\unitlength * \real{\svgscale}}%
    \fi%
  \else%
    \setlength{\unitlength}{\svgwidth}%
  \fi%
  \global\let\svgwidth\undefined%
  \global\let\svgscale\undefined%
  \makeatother%
  \begin{picture}(1,0.28611592)%
    \lineheight{1}%
    \setlength\tabcolsep{0pt}%
    \put(0,0){\includegraphics[width=\unitlength,page=1]{opt_profile.pdf}}%
    \put(0.36141891,0.13498476){\color[rgb]{0,0,0}\makebox(0,0)[lt]{\begin{minipage}{0.05452361\unitlength}\raggedright $2$\end{minipage}}}%
    \put(0.88276975,0.13498297){\color[rgb]{0,0,0}\makebox(0,0)[lt]{\begin{minipage}{0.0436189\unitlength}\raggedright $2$\end{minipage}}}%
    \put(0.63106212,0.13399073){\color[rgb]{0,0,0}\makebox(0,0)[lt]{\begin{minipage}{0.05452361\unitlength}\raggedright $-2$\end{minipage}}}%
    \put(0.10713966,0.1329994){\color[rgb]{0,0,0}\makebox(0,0)[lt]{\begin{minipage}{0.04262757\unitlength}\raggedright $-2$\end{minipage}}}%
    \put(0.21334536,0.26694591){\color[rgb]{0,0,0}\makebox(0,0)[lt]{\begin{minipage}{0.03866217\unitlength}\raggedright $1$\end{minipage}}}%
    \put(0.19660885,0.03336986){\color[rgb]{0,0,0}\makebox(0,0)[lt]{\begin{minipage}{0.04807996\unitlength}\raggedright $-1$\end{minipage}}}%
    \put(0.72285108,0.27091573){\color[rgb]{0,0,0}\makebox(0,0)[lt]{\begin{minipage}{0.03866217\unitlength}\raggedright $1$\end{minipage}}}%
    \put(0.71235188,0.03287419){\color[rgb]{0,0,0}\makebox(0,0)[lt]{\begin{minipage}{0.04807996\unitlength}\raggedright $-1$\end{minipage}}}%
    \put(0.77658871,0.28727282){\color[rgb]{0,0,0}\makebox(0,0)[lt]{\begin{minipage}{0.11102992\unitlength}\raggedright $\theta_0'(z)$\\ \end{minipage}}}%
    \put(0.24990068,0.28614469){\color[rgb]{0,0,0}\makebox(0,0)[lt]{\begin{minipage}{0.11102992\unitlength}\raggedright $\theta_0(z)$\\ \end{minipage}}}%
    \put(0.46798691,0.13498207){\color[rgb]{0,0,0}\makebox(0,0)[lt]{\begin{minipage}{0.03172284\unitlength}\raggedright $z$\end{minipage}}}%
    \put(0.99190936,0.13646908){\color[rgb]{0,0,0}\makebox(0,0)[lt]{\begin{minipage}{0.03172284\unitlength}\raggedright $z$\end{minipage}}}%
    \put(0,0){\includegraphics[width=\unitlength,page=2]{opt_profile.pdf}}%
  \end{picture}%
\endgroup%

%% file: z_alpha_xi.pdf_tex
%% Accompanies image file 'z_alpha_xi.pdf' (pdf, eps, ps)
\begingroup%
  \makeatletter%
  \providecommand\color[2][]{%
    \errmessage{(Inkscape) Color is used for the text in Inkscape, but the package 'color.sty' is not loaded}%
    \renewcommand\color[2][]{}%
  }%
  \providecommand\transparent[1]{%
    \errmessage{(Inkscape) Transparency is used (non-zero) for the text in Inkscape, but the package 'transparent.sty' is not loaded}%
    \renewcommand\transparent[1]{}%
  }%
  \providecommand\rotatebox[2]{#2}%
  \newcommand*\fsize{\dimexpr\f@size pt\relax}%
  \newcommand*\lineheight[1]{\fontsize{\fsize}{#1\fsize}\selectfont}%
  \ifx\svgwidth\undefined%
    \setlength{\unitlength}{1372.28736697bp}%
    \ifx\svgscale\undefined%
      \relax%
    \else%
      \setlength{\unitlength}{\unitlength * \real{\svgscale}}%
    \fi%
  \else%
    \setlength{\unitlength}{\svgwidth}%
  \fi%
  \global\let\svgwidth\undefined%
  \global\let\svgscale\undefined%
  \makeatother%
  \begin{picture}(1,0.30632661)%
    \lineheight{1}%
    \setlength\tabcolsep{0pt}%
    \put(0,0){\includegraphics[width=\unitlength,page=1]{z_alpha_xi.pdf}}%
    \put(-0.78464406,0.70301942){\color[rgb]{0,0,0}\makebox(0,0)[lt]{\begin{minipage}{0.23518671\unitlength}\raggedright $S_{\delta,\alpha}$\end{minipage}}}%
    \put(-0.79067331,0.42668703){\color[rgb]{0,0,0}\makebox(0,0)[lt]{\begin{minipage}{0.10246031\unitlength}\raggedright $s=-1$\end{minipage}}}%
    \put(0,0){\includegraphics[width=\unitlength,page=2]{z_alpha_xi.pdf}}%
    \put(-0.77595329,0.47758917){\color[rgb]{0,0,0}\makebox(0,0)[lt]{\begin{minipage}{0.10489548\unitlength}\raggedright $\alpha$\end{minipage}}}%
    \put(0,0){\includegraphics[width=\unitlength,page=3]{z_alpha_xi.pdf}}%
    \put(-0.84685251,0.53242238){\color[rgb]{0,0,0}\makebox(0,0)[lt]{\begin{minipage}{0.16347734\unitlength}\raggedright $z_\alpha^-$\end{minipage}}}%
    \put(0,0){\includegraphics[width=\unitlength,page=4]{z_alpha_xi.pdf}}%
    \put(-0.78316603,0.54015664){\color[rgb]{0,0,0}\makebox(0,0)[lt]{\begin{minipage}{0.06944148\unitlength}\raggedright $s$\end{minipage}}}%
    \put(0.3615125,0.0725583){\color[rgb]{0,0,0}\makebox(0,0)[lt]{\begin{minipage}{0.05300901\unitlength}\raggedright $r$\end{minipage}}}%
    \put(0,0){\includegraphics[width=\unitlength,page=5]{z_alpha_xi.pdf}}%
    \put(0.16563204,0.14417912){\color[rgb]{0,0,0}\makebox(0,0)[lt]{\begin{minipage}{0.1714903\unitlength}\raggedright \end{minipage}}}%
    \put(0.13810992,0.02765036){\color[rgb]{0,0,0}\makebox(0,0)[lt]{\begin{minipage}{0.32543369\unitlength}\raggedright $\hat{\chi}_\alpha(\rho_{\varepsilon,\alpha},Z_{\varepsilon,\alpha}^-)=0$\end{minipage}}}%
    \put(0.01415864,0.30708763){\color[rgb]{0,0,0}\makebox(0,0)[lt]{\begin{minipage}{0.37024938\unitlength}\raggedright $\hat{\chi}_\alpha(\rho_{\varepsilon,\alpha},Z_{\varepsilon,\alpha}^-)=1$\end{minipage}}}%
    \put(0.36373556,0.30487237){\color[rgb]{0,0,0}\makebox(0,0)[lt]{\begin{minipage}{0.23518671\unitlength}\raggedright $S_{\delta,\alpha}$\end{minipage}}}%
    \put(0.36481883,0.1652359){\color[rgb]{0,0,0}\makebox(0,0)[lt]{\begin{minipage}{0.20382109\unitlength}\raggedright $\sim\varepsilon$\end{minipage}}}%
    \put(0.36481881,0.18930031){\color[rgb]{0,0,0}\makebox(0,0)[lt]{\begin{minipage}{0.20382109\unitlength}\raggedright $\sim\varepsilon$\end{minipage}}}%
    \put(0,0){\includegraphics[width=\unitlength,page=6]{z_alpha_xi.pdf}}%
    \put(0.97080914,0.01507927){\color[rgb]{0,0,0}\makebox(0,0)[lt]{\begin{minipage}{0.12969469\unitlength}\raggedright $\rho$\end{minipage}}}%
    \put(0.71063221,0.23009137){\color[rgb]{0,0,0}\makebox(0,0)[lt]{\begin{minipage}{0.12571632\unitlength}\raggedright $Z$\end{minipage}}}%
    \put(0.47002702,0.15085815){\color[rgb]{0,0,0}\makebox(0,0)[lt]{\begin{minipage}{0.17643635\unitlength}\raggedright $\hat{\chi}_\alpha=0$\end{minipage}}}%
    \put(0.46913733,0.2637515){\color[rgb]{0,0,0}\makebox(0,0)[lt]{\begin{minipage}{0.1934623\unitlength}\raggedright $\hat{\chi}_\alpha=1$\end{minipage}}}%
    \put(0,0){\includegraphics[width=\unitlength,page=7]{z_alpha_xi.pdf}}%
    \put(0.16981506,0.06160151){\color[rgb]{0,0,0}\makebox(0,0)[lt]{\begin{minipage}{0.13048882\unitlength}\raggedright $s=-1$\\ \end{minipage}}}%
    \put(0,0){\includegraphics[width=\unitlength,page=8]{z_alpha_xi.pdf}}%
    \put(0.19385141,0.10979804){\color[rgb]{0,0,0}\makebox(0,0)[lt]{\begin{minipage}{0.03437215\unitlength}\raggedright $\alpha$\end{minipage}}}%
    \put(0,0){\includegraphics[width=\unitlength,page=9]{z_alpha_xi.pdf}}%
    \put(0.07991394,0.23254101){\color[rgb]{0,0,0}\makebox(0,0)[lt]{\begin{minipage}{0.09399816\unitlength}\raggedright $z_\alpha^-$\end{minipage}}}%
    \put(0,0){\includegraphics[width=\unitlength,page=10]{z_alpha_xi.pdf}}%
    \put(0.18823279,0.25288796){\color[rgb]{0,0,0}\makebox(0,0)[lt]{\begin{minipage}{0.06944148\unitlength}\raggedright $s$\end{minipage}}}%
    \put(0,0){\includegraphics[width=\unitlength,page=11]{z_alpha_xi.pdf}}%
    \put(0.72705091,0.17484351){\color[rgb]{0,0,0}\makebox(0,0)[lt]{\begin{minipage}{0.06923621\unitlength}\raggedright $1$\end{minipage}}}%
    \put(0.76279882,0.14642188){\color[rgb]{0,0,0}\makebox(0,0)[lt]{\begin{minipage}{0.08432614\unitlength}\raggedright $\cos\alpha$\end{minipage}}}%
    \put(0,0){\includegraphics[width=\unitlength,page=12]{z_alpha_xi.pdf}}%
    \put(0.57427883,0.10548473){\color[rgb]{0,0,0}\makebox(0,0)[lt]{\begin{minipage}{0.07808151\unitlength}\raggedright $\sin\alpha$\end{minipage}}}%
    \put(0.57433346,0.06784131){\color[rgb]{0,0,0}\makebox(0,0)[lt]{\begin{minipage}{0.20041955\unitlength}\raggedright $H_0+\sin\alpha$\end{minipage}}}%
    \put(0,0){\includegraphics[width=\unitlength,page=13]{z_alpha_xi.pdf}}%
    \put(0.86764052,0.04584651){\color[rgb]{0,0,0}\makebox(0,0)[lt]{\begin{minipage}{0.03765952\unitlength}\raggedright $1$\end{minipage}}}%
    \put(0,0){\includegraphics[width=\unitlength,page=14]{z_alpha_xi.pdf}}%
  \end{picture}%
\endgroup%